\documentclass[11pt]{amsart}
\makeatletter
\renewcommand{\tocsection}[3]{%
  \indentlabel{\@ifnotempty{#2}{\bfseries\ignorespaces#1 #2\quad}}\bfseries#3}
\renewcommand{\tocsubsection}[3]{%
  \indentlabel{\@ifnotempty{#2}{\ignorespaces#1 #2\quad}}#3}

\newcommand\@dotsep{4.5}
\def\@tocline#1#2#3#4#5#6#7{\relax
  \ifnum #1>\c@tocdepth 
  \else
    \par \addpenalty\@secpenalty\addvspace{#2}%
    \begingroup \hyphenpenalty\@M
    \@ifempty{#4}{%
      \@tempdima\csname r@tocindent\number#1\endcsname\relax
    }{%
      \@tempdima#4\relax
    }%
    \parindent\z@ \leftskip#3\relax \advance\leftskip\@tempdima\relax
    \rightskip\@pnumwidth plus1em \parfillskip-\@pnumwidth
    #5\leavevmode\hskip-\@tempdima{#6}\nobreak
    \leaders\hbox{$\m@th\mkern \@dotsep mu\hbox{.}\mkern \@dotsep mu$}\hfill
    \nobreak
    \hbox to\@pnumwidth{\@tocpagenum{\ifnum#1=1\bfseries\fi#7}}\par
    \endgroup
  \fi}
\AtBeginDocument{%
\expandafter\renewcommand\csname r@tocindent0\endcsname{0pt}
}
\def\l@subsection{\@tocline{2}{0pt}{2.5pc}{5pc}{}}
\makeatother
\usepackage{hyperref}

\usepackage{enumerate}

\usepackage[utf8]{inputenc}
\usepackage[T1]{fontenc}

\usepackage[mathscr]{eucal}
\usepackage{amsmath,amsfonts,amsthm,amssymb}
\usepackage[all]{xy}
\usepackage{mathtools}
\usepackage{bbm}
\usepackage[a4paper, hmargin={2.7cm,2.7cm}, centering]{geometry}
\usepackage[capitalize]{cleveref}
\usepackage[normalem]{ulem}
\usepackage{enumitem}
\usepackage{enumerate,color}

\newtheorem*{theorem*}{Theorem}

\newtheorem{lemma}{Lemma}[section]
\newtheorem{theorem}[lemma]{Theorem}
\newtheorem{corollary}[lemma]{Corollary}
\newtheorem{question}[lemma]{Question}
\newtheorem{proposition}[lemma]{Proposition}
\newtheorem{remark}[lemma]{Remark}
\newtheorem{conjecture}[lemma]{Conjecture}
\theoremstyle{definition}
\newtheorem{convention}[lemma]{Convention}
\newtheorem{definition}[lemma]{Definition}
\newtheorem*{proposition*}{Proposition}

\newtheorem{problem}{Problem}

\theoremstyle{remark}

\newcommand{\C}{{\mathbb C}}
\newcommand{\E}{{\mathbb E}}
\newcommand{\F}{{\mathbb F}}
\newcommand{\N}{{\mathbb N}}

\newcommand{\Q}{{\mathbb Q}}
\newcommand{\R}{{\mathbb R}}
\newcommand{\T}{{\mathbb T}}
\newcommand{\Z}{{\mathbb Z}}

\def \c {\bold{c}}
\def \e {\epsilon}
\def \E {\overline{\mathbb{E}}}

\def \g {\bold{g}}
\def \h {\bold{h}}
\def \p {\bold{p}}
\def \q {\bold{q}}
\def \u {\bold{u}}
\def \v {\bold{v}}
\def \w {\bold{w}}
\def \I {\mathcal{I}}

\def \X {\bold{X}}
\def \cB {\mathcal{B}}

\def \vl {\varlimsup_{N\to \infty}}
\def \F {\sup_{\substack{  (I_{N})_{N\in\mathbb{N}} \\ \text{ F\o lner seq.} }}}

\def \ei {\mathbb{E}_{n\in I_{N}}}

\newcommand{\norm}[1]{\left\Vert #1\right\Vert}
\newcommand{\nnorm}[1]{\lvert\!|\!| #1|\!|\!\rvert}

\theoremstyle{definition}
\newtheorem*{definition*}{Definition}
\newtheorem*{conjecture*}{Conjecture}
\newtheorem*{remark*}{Remark}
\newtheorem*{remarks*}{Remarks}
\newtheorem*{claim*}{Claim}

\newtheorem{example}{Example}

\begin{document}

\title[Decomposition of multicorrelation sequences and joint ergodicity]{Decomposition of multicorrelation sequences and joint ergodicity} 

\author{Sebasti{\'a}n Donoso, Andreu Ferr\'e Moragues, Andreas Koutsogiannis and Wenbo~Sun}
\address[Sebasti{\'a}n Donoso]{\textsc{Departamento de Ingenier\'{\i}a Matem\'atica and Centro de Modelamiento Matem\'atico, Universidad de Chile \& IRL 2807 - CNRS, Beauchef 851, Santiago, Chile.}} \email{sdonoso@dim.uchile.cl}

\address[Andreu Ferr\'e Moragues]{Department of mathematics, The Ohio State University, 231 West 18th Avenue, Columbus, OH 43210-1174, USA} \email{ferremoragues.1@osu.edu}

\address[Andreas Koutsogiannis]{Department of mathematics, The Ohio State University, 231 West 18th Avenue, Columbus, OH 43210-1174, USA} \email{koutsogiannis.1@osu.edu}
\address[Wenbo Sun]{\textsc{Department of Mathematics, Virginia Tech, 902 Prices Fork RD, Blacksburg, VA, 24061, USA}}
\email{swenbo@vt.edu}

\thanks{The first author was supported by ANID/Fondecyt/1200897  and Grant Basal-ANID AFB170001.}



\maketitle

\begin{abstract}  
 We show that, under finitely many ergodicity assumptions, any multicorrelation sequence defined by invertible measure preserving $\mathbb{Z}^d$-actions with multivariable integer polynomial iterates is the sum of a nilsequence and a null sequence, extending a recent result of the second author. To this end, we develop a new seminorm bound estimate for multiple averages by improving the results in a previous work of the first, third and fourth authors.
 We also use this approach to obtain new criteria for joint ergodicity of multiple averages with multivariable polynomial iterates on $\Z^{d}$-systems.
\end{abstract}

\tableofcontents

\section{Introduction}

\subsection{Decomposition of multicorrelation sequences}
\noindent The structure and limiting behaviour of (averages of) \emph{multicorrelation sequences,} i.e., sequences of the form 
\begin{equation*}
(n_{1},\dots,n_{k})\mapsto
\int_X f_0\cdot T_1^{n_1}f_1\cdot\ldots\cdot T_k^{n_k}f_k\;d\mu,
\end{equation*}
where $k\in \mathbb{N},$ $T_1,\ldots,T_k\colon X\to X$ are invertible and commuting (i.e., $T_iT_j=T_jT_i$ for all $i,j$) measure preserving transformations on a probability space $(X,\mathcal{B},\mu),$\footnote{ We say that $T$ \emph{preserves $\mu$} if $\mu(T^{-1}A)=\mu(A)$ for all $A\in\mathcal{B}.$ The tuple $(X,\mathcal{B},\mu,T_1,\ldots,T_k)$ is a \emph{(measure preserving) system}.} $f_0,\ldots,f_k\in L^\infty(\mu)$ and $n_1,\ldots,n_k\in \mathbb{Z},$  is a central topic in ergodic theory. For $k=1$, Herglotz-Bochner's theorem implies that the sequence $\int_X f_0 \cdot T_1^nf_1 \ d\mu$ is given by the Fourier coefficients of some finite complex measure $\sigma$ on $\T$ (see \cite{khintchine} and \cite{koopmanvonneumann}). More specifically, decomposing $\sigma$ into the sum of its atomic part, $\sigma_a$, and continuous part, $\sigma_c$, we get
\[ \int_X f_0 \cdot T_1^nf_1 \ d\mu =\int_{\T} e^{2\pi i nx} \ d\sigma(x)
	 =\int_{\T} e^{2\pi i n x} d\sigma_a(x)+\int_{\T} e^{2\pi i n x} d\sigma_c(x)
 = \psi(n)+\nu(n),\]
where $(\psi(n))$  is \emph{an almost periodic sequence},\footnote{ I.e., there exists a compact abelian group $G$, a continuous function $\phi: G \rightarrow \C$, and $a \in G$ such that $\psi(n)=\phi(a^n),$ $n\in \mathbb{N}$.} and $(\nu(n))$ is a \emph{nullsequence}, i.e.,
\begin{equation}\label{besicovitch1} \lim_{N-M \to \infty} \frac{1}{N-M} \sum_{n=M}^{N-1} |\nu(n)|^2=0.
\end{equation}
More generally, after Furstenberg's celebrated ergodic theoretic proof of Szemer\'edi's theorem (\cite{fu}), for a single transformation $T$ and iterates of the form $in,$ $1\leq i\leq k,$ there has been a particular interest in the study of the corresponding multicorrelation sequences
\begin{equation}\label{BHKmulticorrelation}
    \alpha(n)=\int_X f_0\cdot T^nf_1\cdot\ldots\cdot T^{kn}f_k\;d\mu.
\end{equation}
For $T$ ergodic (i.e., every $T$-invariant set in $\mathcal{B}$ has trivial measure in $\{0,1\}$), Bergelson, Host and Kra (\cite{bhk}) showed that the sequence $(\alpha(n))$ in \eqref{BHKmulticorrelation} admits a decomposition of the form $a(n)=\phi(n)+\nu(n),$ where $(\phi(n))$ is a uniform limit of $k$-step nilsequences (see Section \ref{snn} for the definition) and $(\nu(n))$ satisfies \eqref{besicovitch1}.\footnote{ Note that $k$ is the number of linear iterates that appear in \eqref{BHKmulticorrelation}.} Leibman, in \cite{leibman1} for ergodic systems and \cite{leibman2} for general ones, extended the result of Bergelson-Host-Kra to polynomial iterates.

For $d\in\N$,
we say that a tuple $(X,\mathcal{B},\mu,(T_{n})_{n\in \Z^d})$ is a \emph{$\Z^d$-measure preserving system} (or a \emph{$\Z^d$-system}) if $(X,\mathcal{B},\mu)$ is a probability space and $T_{n}\colon X\to X,$ $n \in \Z^d,$ are measure preserving transformations on $X$ such that $T_{(0,\dots,0)}={\rm id}$ and $T_{m}\circ T_{n}=T_{m+n}$ for all $m,n\in \Z^{d}$.\footnote{ We use the notation $T_{p_i(n)}$ to stress the fact that $T$ is a $\mathbb{Z}^d$-action. If $T$ is generated by the $\mathbb{Z}$-actions $T_1,\ldots,T_d$ and $p_i=(p_{i,1},\ldots,p_{i,d}),$ we have $T_{p_i(n)}=\prod_{j=1}^d T_j^{p_{i,j}(n)}.$} It is natural to ask whether splitting results still hold for systems with commuting transformations:

\begin{question}[Question 2, \cite{klmr}]\label{q1}
Let $(X,\mathcal{B},\mu,(T_{n})_{n\in \Z^d})$ be a $\Z^d$-system, $k\in\N$, $p_{1},\dots,p_{k}\colon\Z\to\Z^{d}$ a family of polynomials, and $f_0,f_1,\dots,f_k \in L^{\infty}(\mu)$. Under which conditions on the system can the multicorrelation sequence 
\begin{equation}\label{multicorrelation_Nikos}
\int_Xf_0\cdot T_{p_1(n)}f_1\cdot\ldots\cdot T_{p_k(n)}f_k\;d\mu
\end{equation}
be decomposed as the sum of a uniform limit of nilsequences and  a nullsequence?	
\end{question}

The extension of the aforementioned results from $\mathbb{Z}$ to $\mathbb{Z}^d$-actions is, to this day, a challenging open problem. The main issue is that the proofs of the splitting theorems crucially depend on the theory of characteristic factors via the structure theory developed by Host and Kra (\cite{hostkraoriginal}), a tool that is unavailable in the more general $\mathbb{Z}^d$-setting. 
As an aside, Frantzikinakis provided a partial answer to Question \ref{q1} (for $d=1$) in \cite{fra1} that avoided the use of characteristic factors. The answer was partial in the sense that the nullsequence part was allowed to have an $\ell^{2}(\Z)$ error term. A similar decomposition result for general $d$ was proven by Host and Frantzikinakis in \cite{hostfra}.\footnote{ The third author showed in \cite{K1} the analogue to this result for integer parts, or any combination of rounding functions, of real polynomial iterates. For a refinement of this result, with the average of the error term taken along primes, see \cite{klmr}.} From the  point of view of applications, it is useful to have such splitting results for studying weighted averages, in particular for multiple commuting transformations.\footnote{ It is worth mentioning that the splitting of \eqref{BHKmulticorrelation}, where the average in the null term is taken along primes, was used by Tao and Ter\"av\"ainen to show the logarithmic Chowla conjecture for products of odd factors (\cite{TT}).} 

It was demonstrated in \cite{DKS} that under finitely many ergodicity assumptions, the characteristic factors for the corresponding averages
\begin{equation}\label{E:main_average}
\frac{1}{N}\sum_{n=1}^{N}T_{p_{1}(n)}f_1\cdot\ldots\cdot T_{p_{k}(n)}f_k,\footnote{ Such multiple ergodic averages always have $L^2$-limits as $N\to\infty$ (\cite{W}).}
\end{equation}
 as in the case of $\Z$-actions, are rotations on nilmanifolds (a similar result was obtained in \cite{johnson} under infinitely many ergodicity assumptions). So, it is reasonable  to expect that Question~\ref{q1} holds after postulating finitely many ergodicity assumptions (this is an open problem even in the  $k=2$ case--see \cite{hostfra}). 
 
 A partial answer towards this direction was obtained in \cite{afm} by the second author. Namely, \cite[Theorem~1.5]{afm} shows that for any system $(X,\mathcal{B},\mu, T_1,\dots,$ $T_k)$ with $T_i$ and $T_iT_j^{-1}$ ergodic (for all $i$ and $j \neq i$) and $f_0,\dots,f_k \in L^{\infty}(\mu)$, the sequence
\begin{equation}\label{BB} \int_X f_0 \cdot T_1^n f_1 \cdot \dotso \cdot T_k^n f_k \ d\mu\end{equation}
can be decomposed as a sum of a uniform limit of $k$-step nilsequences plus a nullsequence. 

For more general expressions (as in \eqref{multicorrelation_Nikos}), exploiting results from \cite{johnson}, it is also shown in \cite{afm} that, if we further assume ergodicity in all directions, i.e., $T_1^{a_1}\cdot\dotso\cdot T_d^{a_d}$ is ergodic for all $(a_1,\dots,a_d) \in \Z^d \setminus \{{\bf{0}}\}$, then for any family of pairwise distinct polynomials $p_1,\dots,p_k\colon \Z \rightarrow \Z^d$, the sequence
\begin{equation}\label{multicorrpolyintro}
     \int_X f_0 \cdot T_{p_1(n)}f_1 \cdot \dotso \cdot T_{p_k(n)} f_k \ d\mu
\end{equation}
can be decomposed as a sum of a uniform limit of $D$-step nilsequences plus a nullsequence.\footnote{ Here $D$ depends on $k,$ $d$ and the maximum degree of the $p_i$'s. It also has a connection to the number of van der Corput operations we have to run in the induction (see Remark~\ref{R:bounds} for details).} The proof of this result makes essential use of 
a seminorm bound estimate obtained in \cite{johnson}, where the (infinitely many) ergodicity assumptions are reflected (see \cite[Theorem 1.6]{afm}).

In \cite{DKS}, the first, third and fourth authors improved the seminorm bound estimates of \cite{johnson} by imposing only finitely many ergodic assumptions.
Although the results in \cite{DKS} are stronger than those in \cite{johnson},  
one cannot apply them directly to  \cite{afm} to improve the aforementioned results,
due to the incompatibility of the methods between \cite{DKS} and \cite{afm} (see Subsection~\ref{s13} for more details).

In this article, we extend results from \cite{DKS} in order to obtain splitting theorems for multicorrelation sequences involving multiparameter polynomials, postulating ergodicity assumptions which are even weaker than those in \cite{DKS} on the transformations that define the $\mathbb{Z}^d$-action in \eqref{multicorrpolyintro}; for example, we will see that the sequence $\int_X f_0 \cdot T_1^{n^2}T_{2}^{n} f_1 \cdot T_3^{n^2}T_{4}^{n} f_2 \ d\mu$ admits the desired splitting if we assume that $T_1, T_3, T_1T_3^{-1}$ are ergodic.

\subsection{The joint ergodicity phenomenon} 
In his ergodic theoretic proof of Szemer\'edi's theorem, Furstenberg studied the averages of the multicorrelation sequence \eqref{BHKmulticorrelation}. In particular, a stepping stone in the proof is the special case when the transformation $T$ is weakly mixing (i.e., $T\times T$ is ergodic for $\mu\times \mu$), in which he showed that the averages
\begin{equation}
    \frac{1}{N}\sum_{n=1}^N T^nf_1\cdot\dotso\cdot T^{kn}f_k 
\end{equation}
converge in $L^2(\mu)$ to $\prod_{i=1}^k \int_X f_i \ d\mu$ (which we will refer to as the ``expected limit'') as $N\to\infty$.\footnote{Throughout this paper, unless otherwise stated, all limits of measurable functions on a measure preserving system are taken in $L^{2}$.}  
It was Berend and Bergelson (\cite{BD}) who characterized when the average of the more general expression \eqref{BB}, i.e., for multiple commuting transformations, converges to the expected limit (and this happens exactly when $T_1 \times \dots \times T_k$ and $T_iT_j^{-1}$ for all $i\neq j$ are ergodic). 

Generalizing Furstenberg's result, Bergelson showed (in \cite{wmpet}) that, for a weakly mixing transformation $T$ and essentially distinct polynomials $p_{1},\dots,p_{k}$  (i.e., $p_{i}, p_{i}-p_{j}$ are non-constant for all $1\leq i,j\leq k, i\neq j$)
\[ \lim_{N \to \infty} \frac{1}{N}\sum_{n=1}^N T^{p_1(n)}f_1\cdot\dotso\cdot T^{p_k(n)}f_k=\prod_{i=1}^k \int_X f_i \ d\mu.\footnote{ For $T$ totally ergodic (i.e., $T^n$ is ergodic for all $n\in \mathbb{N}$) and $p_1,\ldots,p_k$ ``independent'' integer polynomials, it is proved in \cite{FK} that we have the same conclusion. This fact remains true for an ergodic $T$ and ``strongly independent'' real-valued polynomials iterates, $[p_1(n)],\ldots,[p_k(n)]$ ($[\cdot]$ denotes the floor function), as well (see \cite{KK}). These last two results also follow by a recent work of Frantzikinakis, \cite{FA}, in which, for single $T,$ we have a plethora of joint ergodicity results for a number of classes of iterates (not just polynomial). Finally, for variable real ``good'' polynomial iterates, one is referred to \cite{K3}.}\]

One can think of this last result as a strong independence
property of the sequences $(T^{p_{i}(n)})_{n\in\mathbb{Z}}, 1\leq i\leq k$ in the weakly mixing case. It is reasonable  to expect a characterization in the case of polynomial iterates, which naturally leads to a general notion of joint ergodicity:

\begin{definition}
We say that the sequence of tuples $(T_{p_{1}(n)},\dots,T_{p_{k}(n)})_{n\in\mathbb{Z}^{L}}$ is \emph{jointly ergodic} for $\mu$ if for every $f_{1},\dots,f_{k}$ $\in L^{\infty}(\mu)$ and every F{\o}lner sequence $(I_{N})_{N\in\mathbb{N}}$ of $\mathbb{Z}^{L}$,\footnote{ A sequence of finite subsets $(I_{N})_{N\in\mathbb{N}}$ of  $\mathbb{Z}^L$ with the property $\lim_{N\to\infty}\vert I_{N}\vert^{-1}\cdot\vert (g+I_{N})\triangle I_{N}\vert=0$ for all $g\in \mathbb{Z}^L$  is called a \emph{F\o lner sequence} in $\mathbb{Z}^L$.} we have that
	\begin{equation}\label{444}
	\lim_{N\to\infty}\frac{1}{\vert I_{N}\vert}\sum_{n\in I_{N}}T_{p_{1}(n)}f_{1}\cdot\ldots\cdot T_{p_{k}(n)}f_{k}=\int_{X}f_{1}\,d\mu\cdot\ldots\cdot \int_{X}f_{k}\,d\mu.
	\end{equation}	 When $k=1$, we also say that $(T_{p_{1}(n)})_{n\in\mathbb{Z}^{L}}$ is \emph{ergodic} for $\mu.$
\end{definition}
	
The following conjecture was stated in 	\cite{DKS}:

\begin{conjecture}[Conjecture 1.5, \cite{DKS}]\label{q2}
		Let $d,k,L\in\mathbb{N},$ $ p_{1},\dots,p_{k}\colon\mathbb{Z}^{L}\to\mathbb{Z}^{d}$ be polynomials and  
		$(X,\mathcal{B},\mu, (T_{g})_{g\in\Z^{d}})$ be a $\Z^{d}$-system. Then the following are equivalent:
\begin{itemize}
    \item[(C1)] $(T_{p_{1}(n)},\dots,T_{p_{k}(n)})_{n\in\mathbb{Z}^{L}}$ is jointly ergodic for $\mu$.
    \item[(C2)] The following conditions are satisfied:
    	\begin{itemize}
		\item[$(i)$] $(T_{p_{i}(n)-p_{j}(n)})_{n\in\mathbb{Z}^{L}}$ is ergodic for $\mu$ for all $1\leq i,j\leq k,$ $i\neq j$; and
		\item[$(ii)$] $(T_{p_{1}(n)}\times\dots\times T_{p_{k}(n)})_{n\in\mathbb{Z}^{L}}$ is ergodic for the product measure $\mu^{\otimes k}$ on $X^{k}$.
	\end{itemize}
\end{itemize}
\end{conjecture}

Answering a question of Bergelson, it was shown in \cite[Theorem 1.4]{DKS} that, for a polynomial $p:\mathbb{Z}^L\to\mathbb{Z},$ the sequence $(T_{1}^{p(n)},\dots,T_{k}^{p(n)})_{n\in\mathbb{Z}^{L}}$ is jointly ergodic for $\mu$ if and only if $((T_{1}\times \dots\times T_{k})^{p(n)})_{n\in\mathbb{Z}^{L}}$ is ergodic for $\mu^{\otimes k}$ and $T_{i}T^{-1}_{j}$ is ergodic for $\mu$ for all $i\neq j.$
In this paper, the strong decomposition results that we obtain allow us to deduce joint ergodicity results for a larger family of polynomials (see Theorems~\ref{j2} and \ref{j3}), thus addressing some additional cases in the aforementioned conjecture.

\section{Main results}

In this section we state the main results of the paper, and provide a number of examples to better illustrate them. We also comment on the approaches that we follow. 

\subsection{Splitting results.}\label{Ss:SR}
\noindent 
Our first main concern is to resolve the incompatibility  between \cite{DKS} and \cite{afm}, and improve the method in \cite{DKS}, in order to obtain an extension of the results in \cite{afm}.

\medskip

Before we state our first result, we need to introduce some notations:

For $d,L\in\mathbb{N}$, the polynomial $q=(q_1,\ldots,q_d): \mathbb{Z}^L\to\mathbb{Z}^d$  is \emph{non-constant} if some $q_i$ is non-constant.

The polynomials $p_{1},\dots,p_{k}\colon\mathbb{Z}^{L}\to\mathbb{Z}^{d}$ are called \emph{essentially distinct} if they are non-constant and $p_{i}-p_{j}$ is non-constant for all $i\neq j$.\footnote{In general, a polynomial $q\colon\mathbb{Z}^{L}\to\mathbb{Z}^{d}$ has rational coefficients (i.e., vectors with rational coordinates).} 

For a subset $A$ of $\mathbb{Q}^{d}$, we denote $G(A)\coloneqq \text{span}_{\Q}\{a\in A\}\cap \Z^{d}.$ 
The following subgroups of $\Z^{d}$ play an important role in this paper:

\begin{definition}\label{d21}
Let 
${\mathbf{p}}=(p_{1},\dots,p_{k}), p_{1},\dots,p_{k}\colon\mathbb{Z}^{L}\to \mathbb{Z}^{d}$ be a family of essentially distinct polynomials  with $p_{i}(n)=\sum_{v\in\mathbb{N}_0^{L}}b_{i,v}n^{v}$ for some $b_{i,v}\in \mathbb{Q}^{d}$ with at most finitely many $b_{i,v},v\in\mathbb{N}_0^{L}$ nonzero.\footnote{Here we denote $n^{v}\coloneqq n_{1}^{v_{1}}\dots n_{L}^{v_{L}}$ for $n=(n_{1},\dots,n_{L})\in\Z^{L}$ and $v=(v_{1},\dots,v_{L})\in\N_0^{L}$,where $0^0:=1.$}
For convenience, we artificially denote $p_{0}$ as the constant zero polynomial and $b_{0,v}\coloneqq0$ for all $v\in\N_{0}^{L}$.
For $0\leq i,j\leq k$, set  $d_{i,j}\coloneqq \deg(p_{i}-p_{j})$ and $G_{i,j}(\mathbf{p})\coloneqq G(\{b_{i,v}-b_{j,v}\colon \vert v\vert=d_{i,j}\})$, where, for $v=(v_1,\ldots,v_L)\in\N_{0}^{L},$ we write $\vert v\vert=v_1+\ldots+v_L.$
 \end{definition}

Our main result provides an affirmative answer to Question \ref{q1} under finitely many ergodicity assumptions on the groups  $G_{i,j}(\mathbf{p})$, which generalizes \cite[Theorem~1.5]{afm}. We say that the group $G_{i,j}(\mathbf{p})$ is \emph{ergodic for $\mu$} if any function $f \in L^2(\mu)$ that is $T_{a}$-invariant for all $a \in G_{i,j}(\mathbf{p})$ is constant.

\begin{theorem}[Decomposition theorem under finitely many ergodicity assumptions]\label{mainthm}
	For $d,k,K,L\in\mathbb{N},$ let $\mathbf{p}=(p_{1},\dots,p_{k}),$ where $p_{1},\dots,p_{k}\colon\mathbb{Z}^{L}\to\mathbb{Z}^{d}$ is a family of essentially distinct polynomials of degree at most $K,$ and let  
	$(X,\mathcal{B},\mu, (T_{n})_{n\in\Z^{d}})$ be a $\Z^{d}$-system.
	If  
	$G_{i,j}(\mathbf{p})$ is ergodic for $\mu$ for all  $0\leq i,j\leq k,i\neq j,$ 
	then, for all $f_0,\dots,f_k \in L^{\infty}(\mu)$, the multicorrelation sequence
	\[ a(n)\coloneqq \int_X f_0\cdot T_{p_1(n)}f_1 \cdot \dotso \cdot T_{p_k(n)} f_k \ d\mu\]
	can be decomposed as a sum of a uniform limit of $D$-step nilsequences and a nullsequence,\footnote{The precise definition of a $D$-step nilsequence will be given in the following section. Furthermore, we say that $a\colon \Z^L \to \C$ is a nullsequence if for any F\o lner sequence $(I_N)$ we have $\lim_{N \to \infty} \frac{1}{|I_N|}\sum_{n \in I_N} |a(n)|^2=0.$} where $D\in\N$ is a constant depending only on $d,k,K,L$. 
\end{theorem}

Note that Theorem \ref{mainthm} goes beyond Question~\ref{q1} as it deals with multi-variable polynomial iterates (i.e., $L>1$).

\begin{example}\label{Ex:34} It was proved in \cite[Theorem~1.5]{afm}  that for any probability space $(X,\mathcal{B},\mu)$ 
	and commuting transformations $T_{1},\dots,T_{k}$ acting on $X$, if
	$T_i$ and $T_iT_j^{-1}$ ergodic (for all $i$ and all $j \neq i$ respectively), then for all $f_0,\dots,f_k \in L^{\infty}(\mu)$, the multicorrelation sequence
	\[ a(n)\coloneqq \int_X f_0 \cdot T_1^n f_1 \cdot \dotso \cdot T_k^n f_k \ d\mu\]
	can be decomposed as a sum of a uniform limit of $k$-step nilsequences plus a nullsequence.  
	Theorem \ref{mainthm} implies a similar result.\footnote{While Theorem \ref{mainthm} does not specify the step $D$ of the nilsequence, a quick argument shows that in this case one can indeed take $D=k$ (see Remark~\ref{R:step} for details).}
\end{example}

The following example shows that Theorem~\ref{mainthm} is stronger than \cite[Theorem~1.6]{afm}, which deals with single variable essentially distinct polynomial iterates:

\begin{example}\label{Ex:33} Let $(X,\mathcal{B},\mu,T_{1},\dots,T_{6})$ be a system with commuting transformations $T_{1},\dots,T_{6}$ and $f_{0},f_{1},\dots,f_{4}\in L^{\infty}(\mu)$. Using \cite[Theorem~1.6]{afm}, we have that the multicorrelation sequence 
	\begin{equation}\label{multicorrelation_Nikos2}
	\alpha(n)=\int_X f_0\cdot T_{1}^{n^{2}}T_{2}^{n}f_1\cdot T_{1}^{n^{2}}T_{3}^{n}f_2\cdot T_{4}^{n^{3}}f_3\cdot T_{5}^{n^{3}}T_{6}^{n}f_4\;d\mu
	\end{equation} 
	can be decomposed as the sum of a uniform limit of nilsequences and  a nullsequence if $T^{a_{1}}_{1}\cdot\ldots\cdot T^{a_{6}}_{6}$ is ergodic for all $(a_{1},\dots,a_{6})\in\Z^{6}\backslash\{{\bf{0}}\}$. In contrast, via Theorem~\ref{mainthm}, one can get the same conclusion by only assuming that $T_{1},T_{2}T_{3}^{-1}, T_{4},T_{5},T_{4}T^{-1}_{5}$ are ergodic. (Indeed,  denoting $T_{(a_{1},\dots,a_{6})}\coloneqq T_{1}^{a_{1}}\cdot\ldots\cdot T_{6}^{a_{6}}$, and $e_{i}$ the vector whose $i$-th entry is 1 and all other entries are 0,
	since ${\bf{p}}=((n^2,n,0,0,0,0),(n^2,0,n,0,0,0),(0,0,0,n^3,0,0),(0,0,0,0,n^3,n)),$ we have that $G_{1,0}(\p)=G_{2,0}(\p)=G(e_{1})$, $G_{1,3}(\p)=G_{2,3}(\p)=G_{3,0}(\p)=G(e_{4})$, $G_{1,4}(\p)=G_{2,4}(\p)=G_{4,0}(\p)=G(e_{5})$, $G_{1,2}(\p)=G(e_{2}-e_{3})$, $G_{3,4}(\p)=G(e_{4}-e_{5})$.)
\end{example}

\subsection{Convergence to the expected limit.}\label{s12}

In \cite[Theorem 1.4]{DKS}, the first, third and fourth authors proved the following case of Conjecture~\ref{q2}: 
  If $T_{1},\dots,T_{k}$ are commuting transformations acting on a probability space $(X,\mathcal{B},\mu)$,
   then $(T_{1}^{p(n)},\dots,T_{k}^{p(n)})_{n\in\mathbb{Z}^{L}}$ is jointly ergodic for $\mu$ if and only if $((T_{1}\times \dots\times T_{k})^{p(n)})_{n\in\mathbb{Z}^{L}}$ is ergodic for $\mu^{\otimes k}$ and $T_{i}T^{-1}_{j}$ is ergodic for $\mu$ for all  $i\neq j$. 
   In this paper, we further extend this result: 

 \begin{theorem}\label{j2}
 	Let $\mathbf{p}=(p_{1}(n)v_{1},\dots, p_{k}(n)v_{k}),$ $p_{1},\dots, p_{k}\colon \mathbb{Z}^{L}\to \mathbb{Z}$, $v_{1},\dots, v_{k}\in \Z^{d}$ be a family of essentially distinct polynomials. Suppose that for all $1\leq i,j\leq k$, if $\deg(p_{i})=\deg(p_{j})$, then either $v_{i}$ and $v_{j}$ are linearly dependent over $\mathbb{Z}$, or $p_{i}(n)$ and $p_{j}(n)$ are linearly dependent over $\mathbb{Z}$ (i.e., there is a non-trivial linear combination of them over $\mathbb{Z}$ which equals to a constant). 
 	Let $(X,\mathcal{B},\mu, (T_{g})_{g\in\mathbb{Z}^{d}})$ be a $\mathbb{Z}^{d}$-system. Then the following are equivalent:
 	\begin{itemize}
 	    \item[(C1)] $(T_{p_{i}(n)v_{i}}\colon 1\leq i\leq k)_{n\in\Z^{L}}$ is jointly ergodic for $\mu$
 	    \item[(C2')] The following conditions hold:
	\begin{enumerate}
 		\item[$(i)'$] $(T_{p_{i}(n)v_{i}-p_{j}(n)v_{j}})_{n\in\Z^{L}}$ is ergodic for $\mu$ for all $1\leq i,j\leq k, i\neq j$ with $\deg(p_{i})=\deg(p_{j})$;
 		\item[$(ii)$] $(T_{p_{1}(n)v_{1}}\times\dots\times T_{p_{k}(n)v_{k}})_{n\in\Z^{L}}$ is ergodic for $\mu^{\otimes k}$.
 	\end{enumerate}	
 	\end{itemize}
 	Moreover, (C2') is equivalent to 
 	\begin{itemize}
 	    \item[(C2)] The following conditions hold: 
 	    	\begin{enumerate}
 		\item[$(i)$] $(T_{p_{i}(n)v_{i}-p_{j}(n)v_{j}})_{n\in\Z^{L}}$ is ergodic for $\mu$ for all $1\leq i,j\leq k, i\neq j$;
 		\item[$(ii)$] $(T_{p_{1}(n)v_{1}}\times\dots\times T_{p_{k}(n)v_{k}})_{n\in\Z^{L}}$ is ergodic for $\mu^{\otimes k}$.
 	\end{enumerate}
 	\end{itemize}
 \end{theorem}

Note that the conditions in (C2) are consistent with those in Conjecture~\ref{q2}. On the other hand, the reason we provide an alternative set of equivalent conditions (C2') is that these conditions are easier to check in practise.

 We now give some examples to illustrate Theorem \ref{j2}.
 The first one is for polynomials of distinct degrees:
 \begin{example}\label{Ex:1} Let $(X,\mathcal{B},\mu,T_{1},\dots,T_{k})$ be a system. Using Theorem~\ref{j2}, we conclude that $(T^{n}_{1},T^{n^{2}}_{2},\dots,T^{n^{k}}_{k})_{n\in\Z}$ is jointly ergodic if and only if $(T^{n}_{1}\times\dots\times T^{n^{k}}_{k})_{n\in \Z}$ is ergodic for $\mu^{\otimes k}$, and all the $T_{i}$'s are ergodic for $\mu$. 
 \end{example}
 We remark that Example \ref{Ex:1} can also be proved by using arguments from \cite{CFH}.
 We next present two examples in which the polynomials can be taken to be not necessarily of different degrees, and so, cannot be recovered by the  methods of \cite{CFH}:

 \begin{example}\label{Ex:2} Let $(X,\mathcal{B},\mu,T_{1},T_2,T_3,T_{4})$ be a system. Theorem~\ref{j2} implies that $(T^{n}_{1},T^{n}_{2},T^{n^{2}}_{3},$ $T^{n^{2}}_{4})_{n\in\Z}$ is jointly ergodic if and only if $(T^{n}_{1}\times T^{n}_{2}\times T^{n^{2}}_{3}\times T^{n^{2}}_{4})_{n\in \Z}$ is ergodic for $\mu^{\otimes 4}$, and both $T_{1}T^{-1}_{2}$ and $((T_{3}T^{-1}_{4})^{n^2})_{n\in \N}$ are ergodic for $\mu$. 
 \end{example}

 \begin{example}\label{Ex:23} Let $(X,\mathcal{B},\mu,T_{1},T_2,T_3)$ be a system.  Theorem~\ref{j2} implies that $(T^{n^{4}+n^{2}}_{1}$, $T^{2n^{4}+3n}_{1},$ $T^{2n^{2}+2n+1}_{2}, T^{3n^{2}+3n}_{3})_{n\in\Z}$ is jointly ergodic if and only if $(T^{n^{4}+n^{2}}_{1}\times T^{2n^{4}+3n}_{1}\times T^{2n^{2}+2n+1}_{2}\times  T^{3n^{2}+3n}_{3})_{n\in\Z}$ is ergodic for $\mu^{\otimes 4}$, and both sequences $(T^{-n^{4}+n^{2}-3n}_{1})_{n\in\Z}$ and $((T^{2}_{2}T_{3}^{-3})^{n^{2}+n})_{n\in\Z}$  are ergodic for $\mu$.  
 \end{example}
 
 Another direction for the joint ergodicity problem is verifying whether (C1) implies (C2) in Conjecture~\ref{q2}. 
 Namely, assuming that $(T_{p_{1}(n)}\times\dots\times T_{p_{k}(n)})_{n\in\mathbb{Z}^{L}}$ is ergodic for $\mu^{\otimes k},$ to find a condition, say (P), of certain sequences of actions to be ergodic, under which, we have that $(T_{p_{1}(n)},\dots,T_{p_{k}(n)})_{n\in\mathbb{Z}^{L}}$ is jointly ergodic for $\mu$.
 By combining existing results from \cite{hostkraoriginal,johnson,L05} (see also \cite[Proposition~1.2]{DKS}),
 (P) can be taken to be: ``$T_{g}$ is ergodic for $\mu$ for all $g\in\Z^{d}\backslash\{\bf{0}\}$.''
 Denoting
 $p_{i}(n)=\sum_{v\in\N_0^{L},0\leq \vert v\vert\leq K}b_{i,v}n^{v}$ for some $b_{i,v}\in \mathbb{Q}^{d}$ and $K\in\N_0$,
  this result was extended in \cite[Theorem~1.3]{DKS}, where the previous property is replaced by: ``$T_{g}$ is ergodic for $\mu$ for all $g$ that belongs to the finite set $R$,''
 where 
 \[ R=\bigcup_{0< \vert v\vert\leq K}\{b_{i,v}, b_{i,v}-b_{j,v}\colon 1\leq i, j\leq k\}\backslash\{{\bf 0}\}.\]
 
 In this paper, we replace the latter condition with an even weaker one: 
 
 \begin{theorem}\label{j3}
 	Let $d,k,L\in\mathbb{N},$ $\mathbf{p}=(p_{1},\dots, p_{k}),$  $p_{1},\dots, p_{k}\colon\mathbb{Z}^{L}\to \mathbb{Z}^{d}$  be a family of essentially distinct polynomials and $\mathbf{X}=(X,\mathcal{B},\mu, (T_{g})_{g\in\mathbb{Z}^{d}})$ a $\mathbb{Z}^{d}$-system. Then $(T_{p_{1}(n)},\dots,$ $T_{p_{k}(n)})_{n\in\Z^{L}}$ is jointly ergodic for $\mu$ if both of the following conditions hold:
 	\begin{enumerate}
 		\item[$(i)$] 
	$G_{i,j}(\mathbf{p})$ is ergodic for $\mu$ for all  $0\leq i,j\leq k,i\neq j$;
 		\item[$(ii)$] $(T_{p_{1}(n)}\times\dots\times T_{p_{k}(n)})_{n\in\Z^{L}}$ is ergodic for $\mu^{\otimes k}$.
 	\end{enumerate}
 \end{theorem}

 The last example for this section reflects the stronger nature of the previous theorem compared to what was previously known.
 
 \begin{example}\label{Ex:3} Let $(X,\mathcal{B},\mu,T_{1},T_2,T_3,T_{4})$ be a system. Then, \cite[Theorem~1.3]{DKS} implies that $(T^{n^2}_{1} T^{n}_{2},T^{n^2}_{3} T^{n}_{4})_{n\in\Z}$ is jointly ergodic if $((T^{n^2}_{1} T^n_2)\times (T^{n^2}_3 T^n_{4}))_{n\in \mathbb{Z}}$ is ergodic for $\mu^{\otimes 2}$, and all $T_{1},T_{2},T_{3},T_{4},T_{1}T_{3}^{-1},T_{2}T_{4}^{-1}$ are ergodic for $\mu$. Using Theorem~\ref{j3}, we conclude that $(T^{n^2}_{1} T^{n}_{2},$ $T^{n^2}_{3} T^{n}_{4})_{n\in\Z}$ is jointly ergodic if we instead only assume that $((T^{n^2}_{1} T^n_2)\times (T^{n^2}_3 T^n_{4}))_{n\in \mathbb{Z}}$ is ergodic for $\mu^{\otimes 2}$, and all $T_{1},T_{3},T_{1}T_{3}^{-1}$ are ergodic for $\mu$.
 \end{example}
 
\subsection{Strategy of the paper}\label{s13}
   The central ingredient in proving the main results of the paper (Theorems \ref{mainthm}, \ref{j2} and \ref{j3}) is to find proper \emph{characteristic factors} for the average (\ref{E:main_average}), i.e., sub-$\sigma$-algebras $\mathcal{D}_{1},\dots,\mathcal{D}_{k}$ of $\mathcal{B}$ such that the average (\ref{E:main_average}) remains invariant if we replace each $f_{i}$ by its conditional expectation (see Section \ref{s3} for the definition) with respect to $\mathcal{D}_{i}$. An important type of characteristic factor, called the \emph{Host-Kra characteristic factor}, was invented in \cite{hostkraoriginal} to study multiple averages for $\Z$-systems (see Section \ref{s3} for the definition of these factors). This concept was generalized to systems with commuting transformations in \cite{host1} (see also \cite{Sun}). 
   The main tool used in our results, special cases of which have been studied extensively in the past (see for example \cite{CFH,FK,host1,hostkraoriginal,johnson}), is the following: 

  \begin{theorem}\label{5555}
  	Let $d,k,K,L,s\in\mathbb{N},$ $\mathbf{p}=(p_{1},\dots, p_{k}),$ $p_{1},\dots, p_{k}\in\mathbb{Z}^{L}\to \mathbb{Z}^{d}$  be a family of essentially distinct polynomials of degrees at most $K.$ 
  	There exists $D\in\mathbb{N}_0$ depending only on $d,k,K,L,s$ such that for every $\mathbb{Z}^{d}$-system $\mathbf{X}=(X,\mathcal{B},\mu, (T_{n})_{n\in\mathbb{Z}^{d}})$, every $f_{1},\dots,$ $f_{k}\in L^{\infty}(\mu)$, and every F\o lner sequence $(I_{N})_{N\in\N}$ of $\Z^{L}$, if $f_{i}$ is orthogonal to the Host-Kra characteristic factor ${Z}_{\{G_{i,j}(\mathbf{p})\}^{\times D}_{0\leq j\leq k, j\neq i}}(\mathbf{X})$ for some $1\leq i\leq k$ (i.e. the conditional expectation of $f_{i}$ under ${Z}_{\{G_{i,j}(\mathbf{p})\}^{\times D}_{0\leq j\leq k, j\neq i}}(\mathbf{X})$ is $0$), then  we have that  
  	\begin{equation}\label{123}
  	\lim_{N\to\infty}\frac{1}{|I_N|}\sum_{n\in I_{N}}\prod_{i=1}^k T_{p_{i}(n)}f_i=0.
  	\end{equation}
  	In particular, if for some $1\leq i\leq k$, $G_{i,j}(\mathbf{p})$ is ergodic for $\mu$ for all $0\leq j\leq k, j\neq i$ and $f_{i}$ is orthogonal to the Host-Kra characteristic factor ${Z}_{(\Z^{d})^{\times kD}}(\mathbf{X})$, then \eqref{123} holds. 
  \end{theorem}
  It is worth noting that the factor ${Z}_{\{G_{i,j}(\mathbf{p})\}^{\times D}_{0\leq j\leq k, j\neq i}}(\mathbf{X})$ we obtain in Theorem~\ref{5555} is not optimal, but it is good enough for our purposes. 
  
  A special case of Theorem~\ref{5555} was proved in \cite[Theorem~5.1]{DKS}. In particular, Theorem~\ref{5555} generalizes \cite[Theorem~5.1]{DKS} in the following ways: 
  \begin{enumerate}
  	\item[(I)] The characteristic factor obtained in Theorem \ref{5555} is of finite step, whereas the one in \cite[Theorem~5.1]{DKS} is of infinite step.
  	\item[(II)] The groups $G_{i,j}(\mathbf{p})$ involved in Theorem \ref{5555} are larger than those in \cite[Theorem~5.1]{DKS}, which makes the characteristic factors in Theorem~\ref{5555} smaller.
  \end{enumerate}	

We remark that the aforementioned technical distinctions have significant influences in the applications of Theorem~\ref{5555}. First, the essential reason why one cannot directly use \cite[Theorem~5.1]{DKS} to improve \cite[Theorem~1.5]{afm} is that the  method used in \cite{afm} requires a   characteristic factor of finite step. This problem is resolved by (I), enabling us to extend \cite[Theorem~1.5]{afm} in this paper. 
Second, \cite[Theorem~5.1]{DKS} does not provide a strong enough characteristic factor in certain circumstances. For example, in the case of Example \ref{Ex:1},
\cite[Theorem~6.5]{CFH} suggests that the Host-Kra seminorms controlling \eqref{123} depend only on the transformations $T_{1},\dots,T_{k}$, whereas the upper bound provided by \cite[Theorem~5.1]{DKS} depends not only on the transformations $T_{1},\dots,T_{k}$ but also on many compositions of them. With the help of (I) and (II), we are able to obtain (and generalize) the aforementioned upper bound of \cite[Theorem~6.5]{CFH}.

 Roughly speaking,
 the achievement of (I) relies on a sophisticated development of a Bessel-type inequality first obtained by Tao and Ziegler in \cite[Proposition~\ref{bessel}]{TZ}.
 The most technical part of this paper is the approach we use to get  (II). In \cite{DKS}, a method was introduced to keep track of the coefficients of the polynomials while running a variation of the PET induction. However, the tracking provided there is not strong enough to imply Theorem \ref{5555}. To overcome this difficulty, we introduce more sophisticated machinery in order to have a better control of the coefficients. 
  
The paper is organized as follows: We provide some background material in Section~\ref{s3}. In Section~\ref{s:3}, we present the variation of PET induction that we use. In Section \ref{s4}, we address how (I) and (II) above can be achieved with  Propositions~\ref{pet3} and \ref{pet}, which improve Propositions 5.6 and 5.5 of \cite{DKS} respectively. 
We conclude the section by proving Theorem~\ref{5555}.
This is the bulk of the paper.
In Section \ref{s6}, we use Theorem~\ref{5555} to deduce Theorems \ref{mainthm}, \ref{j2}, and \ref{j3}, which are the main results of the paper. We conclude with some discussions on future directions in Section~\ref{S:Future_direction}.
  
\subsection{Notation}
We denote with $\mathbb{N},$ $\mathbb{N}_0,$ $\mathbb{Z},$ $\mathbb{Q},$ $\mathbb{R}$ and $\mathbb{C}$ the set of positive integers, non-negative integers, integers, rational numbers, real numbers and complex numbers respectively. If $X$ is a set, and $d\in \mathbb{N}$, $X^d$ denotes the Cartesian product $X\times\cdots\times X$ of $d$ copies of $X$. 

We will denote by $e_i$ the vector which has $1$ as its $i$th coordinate and $0$'s elsewhere. We use in general lower-case letters to symbolize both numbers and vectors but bold letters to symbolize vectors of vectors to highlight this exact fact. The only exception to this convention is the vector ${\bf 0}$ (\emph{i.e.}, the vector with coordinates only $0$'s) which we always symbolize in bold.

Throughout this article, we use the following notation for averages:
Let  $(a(n))_{n\in \Z^L}$ be a sequence of complex numbers, or a sequence of measurable functions on a probability space $(X,\mathcal{B},\mu)$.
We let 

$\mathbb{E}_{n\in A}a(n)  \coloneqq  \frac{1}{|A|} \sum_{n \in A } a({n}),  \; \text{ where A is a finite subset of} \;\mathbb{Z}^L;$

$\E^{\square}_{n\in \Z^{L}}a(n)  \coloneqq \vl \mathbb{E}_{n \in [-N,N]^L } a(n);$\footnote{We use the symbol $\square$ to highlight the fact that the averages are taken along the boxes $[-N,N]^{L}.$}

$\E_{n\in\mathbb{Z}^{L}}a(n) \coloneqq \sup_{\substack{  (I_{N})_{N\in\mathbb{N}} \\ \text{ F\o lner seq.} }}\varlimsup_{N\to\infty} \mathbb{E}_{n\in I_N} a(n);$

$\mathbb{E}^{\square}_{n\in \Z^{L}}a(n)  \coloneqq \lim_{N\to\infty} \mathbb{E}_{n \in [-N,N]^L } a(n) \; \text{ (provided that the limit exists)};$ and

$\mathbb{E}_{n\in\mathbb{Z}^{L}}a(n) \coloneqq \lim_{N\to\infty} \mathbb{E}_{n\in I_N} a(n) \text{ (prov. the limit exists for all F\o lner seq. $(I_{N})_{N\in\N}$)}.$\footnote{It is worth noticing that if the limit $\lim_{N\to\infty} \mathbb{E}_{n\in I_N} a(n)$ exists for all F\o lner sequences (in $\Z^L$), then this limit does not depend on the chosen F\o lner sequence.} 

\medskip

We also consider \emph{iterated} averages: Let $(a(h_{1},\dots,h_{s}))_{h_{1},\dots,h_{s}\in \Z^L}$ be a multi-parameter sequence. We let
\[\E_{h_{1},\dots,h_{s}\in \mathbb{Z}^{L}}a(h_{1},\dots,h_{s})\coloneqq \E_{h_{1}\in\mathbb{Z}^{L}}\ldots\E_{h_{s}\in\mathbb{Z}^{L}}a(h_{1},\dots,h_{s})\]
and adopt similar conventions for $\mathbb{E}_{h_{1},\dots,h_{s}\in \mathbb{Z}^{L}}$, $\E^{\square}_{h_{1},\dots,h_{s}\in \mathbb{Z}^{L}}$ and $\mathbb{E}^{\square}_{h_{1},\dots,h_{s}\in \mathbb{Z}^{L}}$ respectively.

We end this section by recalling the notion of a system indexed by an abelian group $(G,+)$.
We say that a tuple $(X,\mathcal{B},\mu,(T_{g})_{g\in G})$ is a \emph{$G$-measure preserving system} (or a \emph{$G$-system}) if $(X,\mathcal{B},\mu)$ is a probability space and $T_{g}\colon X\to X$ are measurable, measure preserving transformations on $X$ such that $T_{e_{G}}={\rm id}$ ($e_G$ is the identity element of $G$) and $T_{g}\circ T_{h}=T_{g+h}$ for all $g,h\in G$. A $G$-system will be called \emph{ergodic} if for any $A\in\mathcal{B}$ such that $T_{g}A=A$ for all $g\in G$, we have that $\mu(A)\in \{0,1\}$. In this paper, we are mostly concerned about $\Z^{d}$-systems and $L^2(\mu)$-norm limits of (multiple) ergodic averages. For the corresponding norm, when it is clear from the context, we will write $\norm{\cdot}_2$ instead of $\norm{\cdot}_{L^2(\mu)}$.

\section{Background Material}\label{s3}

In this section we recall some background material and prove some intermediate results that will be used later throughout the paper.

\subsection{Host-Kra Seminorms and factors}  Host-Kra seminorms and their associated factors are arguably the main tools used to analyze the behaviour of multiple averages and correlation sequences. In what follows we give general results about these seminorms and factors, following the notation used in \cite{DKS}.
 
We first recall the notions of a factor and of the conditional expectation with respect to a factor. We say that the $\Z^d$-system $(Y,\mathcal{D},\nu,(S_{g})_{g\in \Z^d})$ is a \emph{factor} of $(X,\mathcal{B},\mu,(T_{g})_{g\in \Z^d})$ if there exists a measurable map $\pi \colon (X,\mathcal{B},\mu)\to (Y,\mathcal{D},\nu)$ such that $\mu(\pi^{-1}(A))=\nu(A)$ for all $A\in \mathcal{D}$, and $\pi\circ T_{g}=S_{g}\circ \pi$ for all $g\in\Z^{d}$.

A factor $(Y,\mathcal{D},\nu,(S_{g})_{g\in \Z^d})$ of $(X,\mathcal{B},\mu,(T_{g})_{g\in \Z^d})$ can be identified with an invariant sub-$\sigma$-algebra $\mathcal{B}'$ of $\mathcal{B}$ by setting $\mathcal{B}'\coloneqq \pi^{-1}(\mathcal{D})$.
Given two $\sigma$-algebras $\mathcal{B}_1$ and $\mathcal{B}_2$, their \emph{joining} $\mathcal{B}_1\vee \mathcal{B}_2$ is the $\sigma$-algebra generated by $B_1\cap B_2$ for all $B_1\in \mathcal{B}_1$ and $B_2\in\mathcal{B}_2$, \emph{i.e.}, the smallest $\sigma$-algebra containing both $\mathcal{B}_1$ and $\mathcal{B}_2$.
 
Given a factor $\pi\colon (X,\mathcal{B},\mu) \to (Y,\mathcal{D},\nu)$ and a function $f\in L^{2}(\mu)$, the {\em conditional expectation of $f$ with respect to $Y$} is the function $g\in L^{2}(\nu),$ which we denote by $\mathbb{E}(f \mid Y),$ with the property 
\[ \int_A g\circ \pi \ d\mu=\int_A f \ d\mu  \textrm{ for all } A \in \pi^{-1}\left(\mathcal{D}\right).\]
Let $\mathbf{X}=(X,\mathcal{B},\mu,(T_{g})_{g\in \Z^d})$ be a $\Z^d$-system and let $\cB_1$ be an invariant sub-$\sigma$-algebra of $\cB$. The \emph{relatively independent joining} of $\mathbf{X}$ with itself with respect to $\cB_1$ is the measure preserving system obtained by considering the product space with the \emph{relatively independent joining}, denoted by $\mu \times_{\cB_1} \mu$, which is given by the formula:
\[ \int_{X \times X} f_1 \otimes f_2 \ d(\mu \times_{\cB_1} \mu)=\int_X \mathbb{E}(f_1|\cB_1) \mathbb{E}(f_2| \cB_1) \ d\mu, \]
for all $f_1, f_{2} \in L^{\infty}(\mu)$.

For a $G$-system $\mathbf{X}=(X,\mathcal{B},\mu,(T_{g})_{g\in G})$, if $H$ is a subgroup of $G,$ we denote by
 $\mathcal{I}(H)(\mathbf{X})$ the set of $A\in\mathcal{B}$ such that $T_{g}A=A$ for all $g\in H$. When there is no confusion, we write $\mathcal{I}(H)$. 
 
 For a $\Z^d$-system $(X,\mathcal{B},\mu,(T_{g})_{g\in \Z^d}),$ define 
\[\mu_{H_{1}}=\mu\times_{\I(H_{1})}\mu\]
and for $k>1$, let
\[\mu_{H_{1},\dots,H_{k}}=\mu_{H_{1},\dots,H_{k-1}}\times_{\mathcal{I}(H_{k}^{[k-1]})}\mu_{H_{1},\dots,H_{k-1}},\]
where $H^{[k-1]}_{k}$ denotes 
the subgroup of $(\Z^{d})^{2^{k-1}}$ consisting of all the elements of the form 
 $h_{k}\times\dots\times h_{k}$ ($2^{k-1}$ copies of $h_{k}$) for some $h_{k}\in H_{k}$.  The \emph{characteristic factor} $Z_{H_{1},\dots,H_{k}}(\mathbf{X})$ is defined to be the sub-$\sigma$-algebra of $\mathcal{B}$ characterized by
 \[\mathbb{E}(f\vert Z_{H_{1},\dots,H_{k}}(\mathbf{X}))=0 \text{ if and only if } \nnorm{f}_{H_{1},\dots,H_{k}}^{2^{k}}\coloneqq \int_{X^{[k]}}\bigotimes_{\e\in\{0,1\}^{k}}\mathcal{C}^{\vert\e\vert}f\,d\mu_{H_{1},\dots,H_{k}}=0\]
 for all $f \in L^{\infty}(\mu)$,
where $X^{[k]}=X\times\cdots\times X$ ($2^k$ copies of $X$), $\vert\e\vert=\e_{1}+\dots+\e_{k}$ for $\e=(\e_{1},\dots,\e_{k})\in\{0,1\}^{k}$, and $\mathcal{C}^{2r+1}f=\overline{f},$ the complex conjugate of $f$, $\mathcal{C}^{2r}f=f$ for all $r\in\Z$. The quantity $\nnorm{f}_{H_1,\dots,H_k}$ denotes the \emph{Host-Kra seminorm} of $f$ with respect to the subgroups $H_1,\dots,H_k$.
Similar to the proof of \cite[Lemma 4]{host1} (or \cite[Lemma~4.3]{hostkraoriginal}), one can show that $Z_{H_{1},\dots,H_{k}}(\mathbf{X})$ is well defined.

We summarize some basic properties of the Host-Kra seminorms and their associated factors.  

\begin{proposition}[Lemma~2.4, \cite{DKS}]\label{prop:basic_properties}
	Let $\X=(X,\mathcal{B},\mu,(T_{g})_{g\in \Z^d})$ be a $\Z^{d}$-system, $H_{1},\dots,H_{k},H'$ be subgroups of $\Z^{d}$ and $f\in L^{\infty}(\mu)$. 
	\begin{itemize}
		\item[$(i)$] For every permutation $\sigma\colon\{1,\dots,k\}\to\{1,\dots,k\}$, we have that \[Z_{H_{1},\dots,H_{k}}(\mathbf{X})=Z_{H_{\sigma(1)},\dots,H_{\sigma(k)}}(\mathbf{X}),\] hence the corresponding seminorm does not depend on the particular order taken for the subgroups $H_1,\ldots,H_k.$
		\item[$(ii)$] If $\I(H_{j})=\I(H')$, then $Z_{H_{1},\dots,H_{j},\dots,H_{k}}(\mathbf{X})=Z_{H_{1},\dots,H_{j-1},H',H_{j+1},\dots,H_{k}}(\mathbf{X})$.
		\item[$(iii)$] For $k\geq 2$ we have that
		\[\nnorm{f}^{2^{k}}_{H_{1},\dots,H_{k}}
		=\mathbb{E}_{g\in H_{k}}\nnorm{f\cdot T_{g}\overline{f}}^{2^{k-1}}_{H_{1},\dots,H_{k-1}},\] while for $k=1,$	
		\[\nnorm{f}^{2}_{H_{1}}
		=\mathbb{E}_{g\in H_{1}}\int_{X} f\cdot T_{g}\overline{f}\,d\mu.\]
		\item[$(iv)$] Let $k\geq 2$. If $H'\leq H_{j}$ is of finite index, then \[Z_{H_{1},\dots,H_{j},\dots,H_{k}}(\mathbf{X})=Z_{H_{1},\dots,H_{j-1},H',H_{j+1},\dots,H_{k}}(\mathbf{X}).\]
		\item[$(v)$] If $H'\leq H_{j}$, then $Z_{H_{1},\dots,H_{j},\dots,H_{k}}(\mathbf{X})\subseteq Z_{H_{1},\dots,H_{j-1},H',H_{j+1},\dots,H_{k}}(\mathbf{X})$.
		
	\item[$(vi)$] For $k\geq 2$,  $\nnorm{f}_{H_1,\ldots,H_{k-1}}\leq \nnorm{f}_{H_1,\ldots,H_{k-1},H_k}$ and thus \[Z_{H_1,\ldots,H_{k-1}}(\mathbf{X})\subseteq  Z_{H_1,\ldots,H_{k-1},H_k}(\mathbf{X}).\]
	
	 	\item[$(vii)$] 	For $k\geq 1$, if $H_1',\ldots, H_k'$ are subgroups of $
	 		\Z^d$, then \[Z_{H_1,\ldots,H_k}(\mathbf{X}) \vee Z_{H_1',\ldots,H_k'}(\mathbf{X}) \subseteq Z_{H_1',\ldots,H_k',H_1,\ldots,H_k}(\mathbf{X}).\] 
	\end{itemize}	
\end{proposition}

As an immediate corollary of Proposition \ref{prop:basic_properties} (iv), we have:

\begin{corollary}[Corollary 2.5, \cite{DKS}]\label{st1}
	Let $H_{1},\dots,H_{k}$ be subgroups of $\Z^{d}$. If  the $H_{i}$-action $(T_{g})_{g\in H_{i}}$ is ergodic on $\mathbf{X}$ for all $1\leq i\leq k$, then $Z_{H_{1},\dots,H_{k}}(\mathbf{X})=Z_{({\Z^{d}})^{\times k}}(\mathbf{X})$.
\end{corollary}

\begin{convention}\label{co1}
	Thanks to \cref{prop:basic_properties} we may adopt a flexible and convenient notation while writing the Host-Kra characteristic factors. For example, if $A=\{H_{1},H_{2}\}^{\times 3}$, then the notation $Z_{A,H_{3},H^{\times 2}_{4},(H_{i})_{i=5,6}}(\X)$ refers to  $Z_{H_{1},H_{1},H_{1},H_{2},H_{2},H_{2},H_{3},H_{4},H_{4},H_{5},H_{6}}(\X)$. 
\end{convention}

Recall that for a subgroup $H\subseteq \Z^d$,  $H^{[1]}$ denotes the subgroup $\{ (h,h)\colon h\in H\}\subseteq \Z^d\times \Z^d$. 

\begin{lemma}\label{lem1}
Let $d \in \N$. Let $(X,\mathcal{B},\mu,(T_g)_{g\in\Z^d})$ be a $\Z^{d}$-system and $H_1,\ldots,H_k,H$ be subgroups of $\Z^d$. Let $f \in L^{\infty}(\mu)$. Then,
\[ \nnorm{f \otimes \bar{f}}_{H_1^{[1]},\dots,H_{k}^{[1]}} \leq \nnorm{ f}_{H_1,\dots,H_{k},H} ^2,\]
where in the left hand side we consider the product space $(X\times X,\mathcal{B}\otimes \mathcal{B},\mu\times \mu)$. 
\end{lemma}

\begin{proof}
We proceed by induction on $k$. For $k=1,$ using the Cauchy-Schwartz inequality, we have
\begin{align*}
 \nnorm{f\otimes \overline{f}}_{H_1^{[1]}}^2&= \mathbb{E}_{g\in H_1} \int f\otimes \overline{f} \cdot (T_g\times T_g )\overline{f}\otimes f \ d(\mu\times \mu) \\
 &= \mathbb{E}_{g\in H_1} \Bigl| \int T_g f \cdot \overline{f} d\mu \Bigr|^2 = \mathbb{E}_{g\in H_1} \Bigl| \int \mathbb{E}(T_g f \cdot \overline{f} | \I(H)) d\mu \Bigr|^2  \\
 & \leq \mathbb{E}_{g\in H_1}  \int |\mathbb{E}(T_g f \cdot \overline{f} | \I(H))|^2 d\mu = \mathbb{E}_{g\in H_1} \nnorm{T_gf\cdot \overline{f}}_{H}^2 = \nnorm{f}_{H,H_1}^4=\nnorm{f}_{H_1,H}^4,
 \end{align*}
from where we conclude the required relation by taking square roots. 

Suppose that the result holds for $k-1$. By \cref{prop:basic_properties} and the induction hypothesis, 
\begin{align*} \nnorm{f \otimes \overline{f}}_{H_1^{[1]},\dots,H_{k}^{[1]}}^{2^k} & =\mathbb{E}_{g\in H_k}\nnorm{(T_g\times T_{g}) f \otimes \overline{f} \cdot \overline{f} \otimes {f} }^{2^{k-1}}_{H_1^{[1]},\dots,H_{k-1}^{[1]}} 
\\ & = \mathbb{E}_{g\in H_k}\nnorm{T_g f \cdot \overline{f} \otimes T_g \overline{f} \cdot f}^{2^{k-1}}_{H_1^{[1]},\dots,H_{k-1}^{[1]}} \\
& \leq \mathbb{E}_{g\in H_k}\nnorm{T_g f \cdot \overline{f}}^{2^{k}}_{H_1,\dots,H_{k-1},H}
\\ & = \nnorm{f}_{H_1,\dots,H_{k-1},H,H_{k}  } = \nnorm{f}_{H_1,\dots,H_{k-1},H_k,H}
\end{align*}
and the claim follows.
\end{proof}

\subsection{Nilsystems, nilsequences and Structure Theorem}\label{snn}

Let $X=N/\Gamma$, where $N$ is a ($k$-step) nilpotent Lie group and $\Gamma$ is a discrete cocompact subgroup of $N$. Let $\mathcal{B}$ be the Borel $\sigma$-algebra of $X,$ $\mu$ the Haar measure on $X,$ and for $n\in \Z^{d},$ let $T_{n}\colon X\to X$ with $T_{n}x=b_{n}\cdot x$ for some group homomorphism $n\mapsto b_{n}$ from $\Z^{d}$ to $N$. We say that $\mathbf{X}=(X,\mathcal{B},\mu,(T_{n})_{n\in \Z^{d}})$ is a \emph{($k$-step) $\Z^{d}$-nilsystem}. 
 For $k\geq 1$, we say that $(a_{n})_{n\in\Z^{d}}$ is a \emph{($k$-step) $\Z^{d}$-nilsequence} if there exist a ($k$-step) $\Z^{d}$-nilsystem $\mathbf{X}=(X,\mathcal{B},\mu,(T_{n})_{n\in \Z^{d}})$, a function $F \in C(X)$ and $x\in X$ such that $a_{n}=F(T_{n}x)$ for all $n\in\Z^{d}$. For $k=0$, a
   \emph{$0$-step nilsequence} is a constant sequence. 
An important reason which makes the Host-Kra characteristic factors powerful is their connection with nilsystems.   The following is a slight generalization of \cite[Theorem~3.7]{Z} (see \cite[Theorem~3.7]{Sun}), which is a higher dimensional version of Host-Kra structure theorem (\cite{hostkraoriginal}).

\begin{theorem}\label{st2}
	Let $\mathbf{X}$ be an ergodic $\mathbb{Z}^{d}$-system. Then $Z_{(\mathbb{Z}^{d})^{\times k}}(\mathbf{X})$ is an inverse limit of $(k-1)$-step $\mathbb{Z}^{d}$-nilsystems. 
\end{theorem} 

\subsection{Bessel's inequality}
An essential difference in the study of multiple ergodic averages between $\Z$-systems and $\Z^{d}$-systems is that in the former case, one can usually bound the average by some Host-Kra seminorm of a function $f$ appearing in the average, whereas in the latter, one can only bound the averages by an average of a family of Host-Kra seminorms of $f$. 
To overcome this difficulty, inspired by the work of Tao and Ziegler (\cite{TZ}), in this subsection we derive an upper bound for expressions of the form $\E_{i\in I}\nnorm{f}_{H_{i,1},\dots,H_{i,s}}$, where $I$ is a finite set and $H_{i,j}$ are subgroups of $\Z^{d}$.

The proof of the following statement is similar to \cite[Corollary~1.22]{TZ}: 
\begin{proposition}[Bessel's inequality]\label{bessel}
	Let $s\in\N$, $(X,\mathcal{B},\mu,(T_{g})_{g\in\mathbb{Z}^{d}})$ be a $\mathbb{Z}^{d}$-system, $I$ be a finite set of indices, and $H_{i,j}, i\in I, 1\leq j\leq s$ be subgroups of $\Z^{d}$. Then for all $f\in L^{\infty}(\mu)$,
	$$\mathbb{E}_{i\in I}\| \mathbb{E}(f\vert Z_{H_{i,1},\dots,H_{i,s}})\|^{2}_{2}
	\leq \Vert f\Vert_{2}\cdot\Bigl(\mathbb{E}_{i,j\in I}\left\Vert \mathbb{E}(f\vert Z_{\{H_{i,i'}+H_{j,j'}\}_{1\leq i',j'\leq s}})\right\Vert^{2}_{2}\Bigr)^{1/2}.$$
\end{proposition}	
\begin{proof}
For convenience, let $f_{i}\coloneqq \mathbb{E}(f\vert Z_{H_{i,1},\dots,H_{i,s}})$. Then
\[\mathbb{E}_{i\in I}\Vert \mathbb{E}(f\vert Z_{H_{i,1},\dots,H_{i,s}})\Vert^{2}_{2}=\langle f,\mathbb{E}_{i\in I}f_{i}\rangle\]
which, by the Cauchy-Schwartz inequality  is bounded by 
\[\Vert f\Vert_{2}\cdot\Bigl\vert\mathbb{E}_{i,j\in I}\langle f_{i},f_{j}\rangle\Bigr\vert^{1/2}. \]	
By \cite[Corollary~1.21]{TZ}, $L^{\infty}(Z_{H_{i,1},\dots,H_{i,s}})$ and $L^{\infty}(Z_{H_{j,1},\dots,H_{j,s}})$ are orthogonal on the orthogonal complement of $L^{\infty}(Z_{\{H_{i,i'}+H_{j,j'}\}_{1\leq i',j'\leq s}})$, hence 
$$\langle f_{i},f_{j}\rangle=\left\| \mathbb{E}(f\vert Z_{\{H_{i,i'}+H_{j,j'}\}_{1\leq i',j'\leq s}})\right\|_{2}^{2}$$
and we have the conclusion. 
\end{proof}	

By repeatedly using Proposition \ref{bessel}, we have:

\begin{corollary}\label{bessel2}
	Let $s,t\in\N$, $(X,\mathcal{B},\mu,(T_{g})_{g\in\mathbb{Z}^{d}})$ be a $\mathbb{Z}^{d}$-system, $I$ be a finite set of indices, and $H_{i,j},$ $i\in I,$ $1\leq j\leq s,$ be subgroups of $\Z^{d}$. Then for all $f\in L^{\infty}(\mu)$ and $T\coloneqq 2^t$ we have 
	\[\Bigl(\mathbb{E}_{i\in I}\Vert \mathbb{E}(f\vert Z_{H_{i,1},\dots,H_{i,s}})\Vert^{2}_{2}\Bigr)^{T}
	\leq \Vert f\Vert^{2T-2}_{2}\cdot \mathbb{E}_{i_{1},\dots,i_{T}\in I}\left\Vert \mathbb{E}(f\vert Z_{\{\sum_{j=1}^{T}H_{i_{j},i'_{j}}\}_{1\leq i'_{1},\dots,i'_{T}\leq s}})\right\Vert^{2}_{2}.\]
\end{corollary}

The next proposition provides an upper bound for $\mathbb{E}_{i\in I}\nnorm{f}_{H_{i,1},\dots,H_{i,s}}$ which can be combined with the previous two statements.

\begin{proposition}\label{bessel3}
		Let $s,t\in\N$, $(X,\mathcal{B},\mu,(T_{g})_{g\in\mathbb{Z}^{d}})$ be a $\mathbb{Z}^{d}$-system, $I$ be a finite set of indices,  and $H_{i,j}, i\in I, 1\leq j\leq s$ be subgroups of $\Z^{d}$. Then for all $f\in L^{\infty}(\mu)$, with $\Vert f\Vert_{L^{\infty}(\mu)}\leq 1$,
	$$\mathbb{E}_{i\in I}\nnorm{f}_{H_{i,1},\dots,H_{i,s}}\leq (\mathbb{E}_{i\in I}\Vert (\mathbb{E}(f|Z_{H_{i,1},\dots,H_{i,s}})\Vert^{2}_{2})^{1/2^{s}}.$$
\end{proposition}		
\begin{proof}
	Note that 
	\begin{equation}
	\begin{split}
	&\quad \nnorm{f}_{H_{i,1},\dots,H_{i,s}}\leq \Vert f\Vert_{L^{2^{s}}(\mu)}\leq \Vert f\Vert_{2}^{1/2^{s-1}}.
	\end{split}
	\end{equation}
	Also, for all $i$ we have   
	 \begin{align*} 
	\nnorm{f}_{H_{i,1},\dots,H_{i,s}}&\leq \nnorm{f-\mathbb{E}(f|Z_{H_{i,1},\dots,H_{i,s}})}_{H_{i,1},\dots,H_{i,s}} + \nnorm{\mathbb{E}(f|Z_{H_{i,1},\dots,H_{i,s}})}_{H_{i,1},\dots,H_{i,s}}\\ 
	&= \nnorm{\mathbb{E}(f|Z_{H_{i,1},\dots,H_{i,s}})}_{H_{i,1},\dots,H_{i,s}},
	\end{align*}    so 
	\begin{equation}
	\begin{split}
	\mathbb{E}_{i\in I}\nnorm{f}_{H_{i,1},\dots,H_{i,s}}
	 \leq & \mathbb{E}_{i\in I}\nnorm{ \mathbb{E}(f|Z_{H_{i,1},\dots,H_{i,s}}) }_{H_{i,1},\dots,H_{i,s}}
	\\\leq & \mathbb{E}_{i\in I}\Vert \mathbb{E}(f|Z_{H_{i,1},\dots,H_{i,s}})\Vert_{2}^{1/2^{s-1}}
	\leq (\mathbb{E}_{i\in I}\Vert (\mathbb{E}(f|Z_{H_{i,1},\dots,H_{i,s}})\Vert^{2}_{2})^{1/2^{s}},
	\end{split}
	\end{equation}
	as was to be shown.
\end{proof}

\subsection{General properties of subgroups of $\Z^d$ and properties of polynomials}\label{pp}

Recall that for a subset $A$ of $\mathbb{Q}^{d}$, we denote $G(A):= \text{span}_{\Q} \{a\in A\}\cap \Z^{d}.$
Next we summarize some properties of these sets.  

\begin{lemma}\label{lemma:basic_properties_G()}
The following properties hold:

\begin{enumerate}[label={(\roman*)},ref=$(\roman*)$] 
\item[$(i)$] \label{item0:prop_G()} For any set $A\subseteq \Z^d$, $G(A)$ is a subgroup of $\Z^d$.
    \item[$(ii)$] \label{item1:prop_G()} Let $A$ be a finite set and $M(A)$ the matrix whose columns are the elements of $A$. Then $G(A)=(M(A)\cdot \Q^{|A|})\cap \Z^d$. 
    \item[$(iii)$] \label{item2:prop_G()} If $H\subseteq \Z^d$ is the subgroup generated by $h_1,\ldots,h_k\in \Z^d$, then $G(H)=G(\{h_1,\ldots,h_k\})$. In particular, letting $M(h_1,\ldots,h_k)$ be the matrix whose columns are $h_1,\ldots,h_k$, we have that $G(\langle h_1,\ldots,h_k\rangle)=(M(h_1,\ldots,h_k)\cdot \Q^{k}) \cap \Z^d$. 

    \item[$(iv)$] \label{item3:prop_G()} For any subgroup $H\subseteq \Z^d$, $H$ has finite index in $G(H)$. Moreover, $G(H)$ is the largest subgroup of $\Z^d$ which is a finite index extension of $H$.

    \item[$(v)$] \label{item4:prop_G()}  If not all of $a_1,\ldots,a_k$ belong to a common proper subspace of $\Q^d$, then $G(\{a_1,\ldots,a_k\})$ $=\Z^d$.   
    \end{enumerate}
\end{lemma}

\begin{proof}
Properties (i), (ii) and (iii) follow directly from the definitions. 

To prove (iv), let $\{g_1,\ldots,g_k\}$ be a set such that $\langle g_1,\ldots,g_k\rangle=G(H)$. For each $i=1,\ldots,k$ there exist $m_i$ and $h_i\in H$ such that $\displaystyle g_i=\frac{h_i}{m_i}$. The group $\langle m_1g_1,\ldots,m_kg_k\rangle$ is of finite index in $\langle g_1,\ldots,g_k\rangle=G(H)$ and is contained in $H$. Therefore $H$ is of finite index in $G(H)$.    

To see that $G(H)$ is the largest finite index extension of $H$, take $H'$ to be any finite index extension of $H$ and take $h'\in H'$. Since $H'$ is a finite index extension of $H$, we have that there exists $n\in \N$ such that $nh'\in H$. This implies that $h'\in G(H)$.  

To show (v), reordering $a_1,\ldots,a_k$ if needed, we may assume that $a_1,\ldots,a_d$ are linearly independent vectors over $\Q$. It follows that $\text{span}_{\Q}(\{a_1,\ldots,a_d
\})=\Q^{d}$ and then $G(\{a_1,\ldots,a_k\})\supseteq G(\{a_1,\ldots,a_d\})=\Z^d$.
\end{proof}

\begin{remark}
If $H_1$ and $H_2$ are subgroups of $\Z^d$, then $G(H_1)+G(H_2)\subseteq G(H_1+H_2)$, with the inclusion possibly being strict. For instance, for $H_1=\langle (1,2)\rangle$, $H_2=\langle (2,1) \rangle$ we have that $G(H_1)=H_1$, $G(H_2)=H_2$ and $H_1+H_2 \subsetneq G(H_1+H_2)=\Z^2$. Nevertheless, \cref{lemma:basic_properties_G()} implies that that $G(H_1)+G(H_2)$ has finite index in $G(H_1+H_2)$. 
\end{remark}

In the remainder of the section, we provide some algebraic lemmas that will be used later in the paper. For a set $E\subseteq \Z^d,$ we define its \emph{upper Banach density} (or just \emph{upper density} when there is no confusion) with $d^\ast(E)\coloneqq \vl \max_{t\in\mathbb{Z}^d}\frac{
|(E-t)\cap  \{1,\ldots, N \}^d|}{N^d}.$ If the limit exists, we say that its value is the \emph{Banach density} (or just \emph{density}) of $E$. 
The proof of the following lemma is routine (see also \cite[Lemma~2.12]{DKS} for a more general version):

\begin{lemma}[Lemma 2.12, \cite{DKS}]\label{ag}
	Let $\c\colon(\mathbb{Z}^{L})^{s}\to\mathbb{R}$ be a polynomial. Then either $\c\equiv 0$ or the set of $\h\in (\mathbb{Z}^{L})^{s}$ such that $\c(\h)=0$
	is of (upper) Banach density $0$.
\end{lemma}

\begin{lemma}\label{misc01}
	Let $v_{i}\in\Z^{L}, 1\leq i\leq k$ and $U$ be a subset of $\Z^{k}$ of positive density. Then \begin{equation} \label{equation:G density}
	G\Biggl(\Biggl\{\sum_{1\leq i\leq k}h_{i} v_{i}\colon \bold{h}=(h_1,\ldots,h_k)\in U\Biggr\}\Biggr)=G(\{v_{i}\colon 1\leq i\leq k\}).\end{equation}
\end{lemma}
\begin{proof}
Note that in \eqref{equation:G density} the right hand side clearly includes the left hand side. To prove the converse inclusion it suffices to show that 
	\begin{equation}\label{lemmamisc01eqn1}
	span_{\mathbb{Q}}\{\bold{h}\colon \bold{h}\in U\}=\mathbb{Q}^{k}.
	\end{equation}
	Since $U$ has positive density, it cannot be contained in any hyperplane of $\Q^k$, so it must have at least $k$ elements that are linearly independent over $\Q$. Thus,  \eqref{lemmamisc01eqn1} follows immediately.
\end{proof}

 \begin{definition}\label{d1}
 Let $P\colon\Z^{LK}\to\R$ be a polynomial. Denote by $\Delta P\colon \Z^{L(K+1)}\to\R$ the polynomial given by $\Delta P(n,h_{1},\dots,h_{K})\coloneqq P(n+h_{K},h_{1},\dots,h_{K-1})-P(n,h_{1},\dots,h_{K-1})$  for all $n,h_{1},\dots,h_{K}\in \Z^{L}$. For a polynomial $P\colon\Z^{L}\to\R$, let $\Delta^{K}P\coloneqq \Delta\cdot\ldots\cdot\Delta P$  (where $\Delta$ acts $K$ times).
 \end{definition}

 \begin{lemma}\label{alg3}
 	Let $K\in\N$ and $Q\colon\Z^{L}\to\R$ be a homogeneous polynomial with $\deg(Q)>K$. If  $Q(n)\notin \mathbb{Q}[n]+\R$, then the set of $(h_{1},\dots,h_{K})\in (\Z^{L})^{K}$ such that $\Delta^{K}Q(\cdot,h_{1},\dots,h_{K})\notin \mathbb{Q}[n]+\R$ is of density 1.
 \end{lemma}

 \begin{proof}
 	We may write $Q(n)=\sum_{i=1}^{M}a_{i}Q_{i}(n)$ for some $M\in\N$, homogeneous polynomials $Q_{1},\dots,Q_{M}$ in $\Q[n]$ of degrees $\deg(Q)$, and real numbers $a_{1},\dots,a_{M}\in\R$ which are linearly independent over $\Q$ (this can be done by taking $a_1\ldots,a_M$ to be a basis of the $\Q$-span of the coefficients of $Q$).
 	Since $Q(n)\notin \mathbb{Q}[n]+\R$, there exists some $1\leq i\leq M$ such that $a_{i}\notin\Q$ and $Q_{i}\not\equiv 0$. Without loss of generality assume that $i=1.$ Since $\deg(Q_{1})>K$, we have that $\Delta^{K}Q_{1}\not\equiv 0$.
 	
 	Suppose that $\Delta^{K}Q(\cdot,h_{1},\dots,h_{K})\in \mathbb{Q}[n]+\R$ for some $(h_{1},\dots,h_{K})\in (\Z^{L})^{K}$.
 	Note that $\Delta^{K}Q(\cdot,h_{1},\dots,h_{K})=\sum_{i=1}^{M}a_{i}\Delta^{K}Q_{i}(\cdot,h_{1},\dots,h_{K})$. Since each $\Delta^{K}Q_{i}(\cdot,h_{1},\dots,h_{K})$ is a rational polynomial of degree $\deg(Q)-K$ and $a_{1},\dots,a_{M}\in\R$ are linearly independent over $\Q$, we must have that  $\Delta^{K}Q_{1}(\cdot,h_{1},\dots,h_{K})\equiv 0$. So if the set of $(h_{1},\dots,h_{K})\in (\Z^{L})^{K}$ such that $\Delta^{K}Q(\cdot,h_{1},\dots,h_{K})\in \mathbb{Q}[n]+\R$ has positive density, then the set of $(n,h_{1},\dots,h_{K})\in (\Z^{L})^{K}$ such that $\Delta^{K}Q_{1}(n,h_{1},\dots,h_{K})=0$ has positive density too. By \cite[Lemma~2.12]{DKS}, $\Delta^{K}Q_{1}\equiv 0$, a contradiction. This finishes the proof. \end{proof}

\section{PET induction}\label{s:3} 

In this section we present the method we use to reduce the complexity of the polynomial iterates, i.e., PET induction,\footnote{ PET is an abbreviation for ``Polynomial Exhaustion Technique.''
} which was first introduced in \cite{wmpet}.  To this end, we start by recalling a variation of van der Corput's lemma from \cite{DKS} that is convenient for our study. We then continue by presenting the inductive scheme via the use of van der Corput operations.

\subsection{The van der Corput lemma} The standard tool used in reducing the complexity of polynomial families of iterates is van der Corput's lemma (also known as ``van der Corput's trick''). We will use the following variation of it, the proof of which can be found in \cite[Lemma~2.2]{DKS}: 

\begin{lemma}[van der Corput lemma] \label{lemma:iteratedVDC}
	Let $(a(n;h_1,\ldots,h_s))_{(n;h_1,\ldots,h_s)\in (\Z^{L})^{s+1}}$ be a bounded sequence by $1$ in a Hilbert space $\mathcal{H}$.\footnote{ We use this unorthodox notation to separate the variable $n$ from the $h_i$'s. The variable $n$ plays a different, compared to the $h_i$'s,  role in our study.}
	 Then, for $\tau\in \mathbb{N}_0$,
	\begin{align*}
	&\overline{\mathbb{E}}^{\square}_{h_1,\ldots,h_s\in \Z^{L}}\F\varlimsup_{N\to \infty}\left \| \mathbb{E}_{n\in I_{N}} a(n;h_1,\ldots,h_s) \right \|^{2\tau}\\
	&\leq 4^{\tau}\overline{\mathbb{E}}^{\square}_{h_1,\ldots,h_s,h_{s+1}\in \Z^{L}}\F\varlimsup_{N\to \infty} \left \vert   \mathbb{E}_{n\in I_{N}} \left \langle a(n+h_{s+1};h_1,\ldots,h_s) ,  a(n;h_1,\ldots,h_s)  \right \rangle \right \vert^{\tau}.
	\end{align*}
\end{lemma} 

We also provide two applications of Lemma \ref{lemma:iteratedVDC} for later use. 
The first one is to get an upper bound for single averages with polynomial iterates and a polynomial exponential weight.
Let $\exp(x)\coloneqq e^{2\pi i x}$ and recall Definition \ref{d1} for the polynomial $\Delta^{K}P$.

 \begin{lemma}\label{Kron}
 	Let $P\colon\Z^{L}\to\R$ and $p\colon\Z^{L}\to\Z^{d}$ be polynomials. Let $\X=(X,\mathcal{B},\mu, (T_{g})_{g\in\mathbb{Z}^{d}})$ be a $\mathbb{Z}^{d}$-system and $f\in L^{\infty}(\mu)$ be a function bounded by 1. 
 For all $K\in\N_0$ and $\tau>0$, there exists a universal constant $C_{K,\tau}>0$ such that 
 		\begin{align*}
 		&\F\varlimsup_{N\to \infty}\left\Vert\mathbb{E}_{n\in I_{N}}\exp(P(n))T_{p(n)}f\right\Vert_{2}^{2\tau}
 		\\&\leq C_{K,\tau}\mathbb{E}^{\square}_{\bold{h}=(h_{1},\dots,h_{K})\in(\Z^{L})^{K}}\F\varlimsup_{N\to \infty}\left\Vert\mathbb{E}_{n\in I_{N}}\exp(\Delta^{K}P(n,\bold{h}))T_{\Delta^{K}p(n,\bold{h})}f\right\Vert^{\tau}_{2}.
 		\end{align*}
 \end{lemma}	
 \begin{proof}
 	When $K=0$, there is nothing to prove. We now assume that the relation holds for some $K\in\N_0$ and we show it for $K+1$.  Using Lemma \ref{lemma:iteratedVDC} and the $T$-invariance of $\mu$, we get
 	\begin{align*}
 	&\overline{\mathbb{E}}^{\square}_{\bold{h}=(h_1,\ldots,h_{K})\in (\Z^{L})^{K}}\F\varlimsup_{N\to \infty}\left \| \mathbb{E}_{n\in I_{N}} \exp(\Delta^{K}P(n,\bold{h}))T_{\Delta^{K}p(n,\bold{h})}f \right \|^{2\tau}_{2}\\
 	&\leq 4^{\tau}\overline{\mathbb{E}}^{\square}_{\bold{h}=(h_1,\ldots,h_{K+1})\in (\Z^{L})^{K+1}}\F\varlimsup_{N\to \infty} \left\vert  \mathbb{E}_{n\in I_{N}}  \int_{X}\exp(\Delta^{K+1}P(n,\bold{h}))T_{\Delta^{K+1}p(n,\bold{h})}f\cdot \overline{f}\,d\mu \right\vert^{\tau}
 	\\&\leq 4^{\tau}\overline{\mathbb{E}}^{\square}_{\bold{h}=(h_1,\ldots,h_{K+1})\in (\Z^{L})^{K+1}}\F\varlimsup_{N\to \infty}\left\Vert  \mathbb{E}_{n\in I_{N}}  \exp(\Delta^{K+1}P(n,\bold{h}))T_{\Delta^{K+1}p(n,\bold{h})}f\right\Vert^{\tau}_{2},
 	\end{align*}	
 	hence the result (the constant that appears depends only on $\tau$ and $K$).
 \end{proof}

 The second application of Lemma \ref{lemma:iteratedVDC} provides an upper bound for single averages, with linear iterates and an exponential weight evaluated at a linear polynomial, on a product system.
 The proof is inspired by \cite[Lemma~5.2]{DKS} and \cite[Proposition~2.9]{hostkraseminorms}.

 \begin{lemma}
 \label{Kron2}
 Let $(X,\mathcal{B},\mu)$ be a probability space, $k,L\in\N$ and $T_{i,j}, 1\leq i\leq k, 1\leq j\leq L$ be commuting measure preserving transformations on $X$. Denote $S_{j}=T_{1,j}\times\dots\times T_{k,j}$ for all $1\leq j\leq L$. Let $G_{i}$ be the group generated by $T_{i,1},\dots,T_{i,L}$. Then for any polynomial $P\colon\Z^{L}\to\R$ of degree 1 and $f_{1},\dots,f_{k}\in L^{\infty}(\mu)$ bounded by 1, we have that
 	\begin{equation}\label{711}
 		\F\varlimsup_{N\to \infty}\Vert\mathbb{E}_{n\in I_{N}}\exp(P(n))R_n f\Vert_{L^{2}(\mu^{\otimes k})}\leq 2\min_{1\leq i\leq k}\nnorm{f_{i}}_{G_{i}^{\times 2}},
 	\end{equation}	where, $f=f_{1}\otimes\dots\otimes f_{k}$ and for $n=(n_{1},\dots,n_{L}),$ $R_n\coloneqq S^{n_{1}}_{1}\cdot\ldots\cdot S^{n_{L}}_{L}$.
 \end{lemma}
 \begin{proof}
 	 Fix $1\leq i\leq k$ and let $P(n)=a\cdot n+b$ for some $a\in \R^{L},b\in\R$.  Then, by Lemma \ref{lemma:iteratedVDC}, the 4th power of the left hand side of (\ref{711}) is bounded by 
 		\begin{align*}
 		& \quad 16\cdot\mathbb{E}^{\square}_{h\in\Z^{L}}\F\varlimsup_{N\to \infty}\left\vert\int_{X}\mathbb{E}_{n\in I_{N}}\exp(P(n+h)-P(n))R_{n+h}f\cdot R_{n}\overline{f}\,d\mu^{\otimes k}\right\vert^{2}
 			\\&= 16\cdot\mathbb{E}^{\square}_{h\in\Z^{L}}\F\varlimsup_{N\to \infty}\left\vert\int_{X}\mathbb{E}_{n\in I_{N}}\exp(a\cdot h)R_{h}f\cdot \overline{f}\,d\mu^{\otimes k}\right\vert^{2}
 		\\&= 16\cdot\mathbb{E}^{\square}_{h\in\Z^{L}}\left\vert\int_{X}R_{h}f\cdot \overline{f}\,d\mu^{\otimes k}\right\vert^{2} \leq 16\cdot\mathbb{E}^{\square}_{h=(h_{1},\dots,h_{L})\in\Z^{L}}\left\vert\int_{X}S_{i}^{h_{i}}f_{i}\cdot \overline{f}_{i}\,d\mu\right\vert^{2}
 			\\&\leq 16\cdot\mathbb{E}^{\square}_{h=(h_{1},\dots,h_{L})\in\Z^{L}}\left\vert\int_{X}\mathbb{E}(S_{i}^{h_{i}}f_{i}\cdot \overline{f}_{i}\vert \mathcal{I}(G_{i}))\,d\mu\right\vert^{2}
 			\\&=16\cdot\mathbb{E}^{\square}_{h_{i}\in\Z}\left\vert\int_{X}\mathbb{E}(S_{i}^{h_{i}}f_{i}\cdot \overline{f}_{i}\vert \mathcal{I}(G_{i}))\,d\mu\right\vert^{2}=16\nnorm{f_{i}}^{4}_{G_{i}^{\times 2}},
 		\end{align*}
 		from where the result follows.
 \end{proof}

 \subsection{The van der Corput operation}

To review the PET induction scheme, we will follow, and slightly modify, the approach from \cite{DKS}. To this end, we extend the definitions that we have already given on the polynomial families of interest (see the beginning of Subsection~\ref{Ss:SR}), taking into account that we treat the first $L$-tuple of variables of the polynomials differently.

\begin{definition}
	For a polynomial $p(n;h_{1},\dots,h_{s})\colon(\mathbb{Z}^{L})^{s+1}\to\mathbb{Z}$, we denote with $\deg(p)$ \emph{the degree of $p$ with respect to $n$} (for example, for $s=1, L=2$, the degree of $p(n_{1},n_{2};h_{1,1},h_{1,2})=h_{1,1}h_{1,2}n_{1}^{2}+h_{1,1}^{5}n_{2},$ is 2).
	
	For a polynomial $p(n;h_{1},\dots,h_{s})=(p_{1}(n;h_{1},\dots,h_{s}),\dots,p_{d}(n;h_{1},\dots,h_{s}))\colon(\mathbb{Z}^{L})^{s+1}\to\mathbb{Z}^{d},$ we let $\deg(p)=\max_{1\leq i\leq d}\deg(p_{i})$ and we say that $p$ is \emph{non-constant} if $\deg(p)>0$ (i.e., some $p_i$ is a non-constant function of $n$). The polynomials
	$q_{1},\dots,q_{k}\colon(\mathbb{Z}^{L})^{s+1}\to\mathbb{Z}^{d}$ are called \emph{essentially distinct} if they are non-constant and $q_i-q_j$ is non-constant for all $i\neq j$.  Finally, for a tuple $\q=(q_{1},\dots,q_{k}),$ we let $\deg(\q)=\max_{1\leq i\leq k}\deg(q_{i}).$\footnote{For clarity, we use non-bold letters for vectors (of polynomials) and bold letters for vectors of vectors (of polynomials).} 

\medskip

Let $(X,\mathcal{B},\mu,(T_{g})_{g\in\mathbb{Z}^{d}})$ be a $\mathbb{Z}^{d}$-system,  
$q_{1},\dots,q_{k}\colon(\mathbb{Z}^{L})^{s+1}\to\mathbb{Z}^{d}$ be polynomials and $g_{1},\dots,$ $g_{k}\colon X\times (\mathbb{Z}^{L})^{s}\to\mathbb{R}$ be functions such that each $g_{m}(\cdot;h_{1},\dots,h_{s})$ is an $L^{\infty}(\mu)$ function bounded by $1$  for all $h_{1},\dots,h_{s}\in\mathbb{Z}^L, 1\leq m\leq k$. If  $\q=(q_{1},\dots,q_{k})$ and $\g=(g_{1},\dots,g_{k}),$ we say that $A=(L,s,k,\g,\q)$
is a \emph{PET-tuple}, and for $\tau\in \N_0$ we set 
\begin{equation}\nonumber
\begin{split}
&S(A,\tau)\coloneqq \overline{\mathbb{E}}^{\square}_{h_{1},\dots,h_{s}\in\mathbb{Z}^{L}}\sup_{\substack{  (I_{N})_{N\in\mathbb{N}} \\ \text{ F\o lner seq.} }}\varlimsup_{N\to \infty}\Bigl\Vert\mathbb{E}_{n\in I_{N}}\prod_{m=1}^{k}T_{q_{m}(n;h_{1},\dots,h_{s})}g_{m}(x;h_{1},\dots,h_{s})\Bigr\Vert^{\tau}_{2}.
\end{split}
\end{equation}
We define $\deg(A)=\deg(\q)$, and say that $A$ is \emph{non-degenerate} if $\q$ is a family of essentially distinct polynomials (for convenience, $\q$ will be called \emph{non-degenerate} as well).
For $1\leq m\leq k$, the tuple $A$ is \emph{$m$-standard} for $f\in L^\infty(\mu)$ if $\deg(A)=\deg(q_{m})$ and $g_{m}(x;h_{1},\dots,h_{s})=f(x)$ for every $x,h_1,\ldots,h_s$.
 That is, $f$ is the $m$-th function in $\g$, only depending on the first variable, and the polynomial $q_m$ that acts on $f$ is of the highest degree.\footnote{ Here, we say $m$-standard for $f$ to highlight the function of interest as, after running the vdC-operation, the position of the functions in the expression we deal with changes.}
 The tuple $A$ will be called \emph{semi-standard} for $f$ if there exists $1\leq m\leq k$ such that $g_{m}(x;h_{1},\dots,h_{s})=f(x)$ for every $x,h_1,\ldots,h_s$. In this case we do not require the function $f$ to have a specific position in $\g$ nor that the polynomial acting on $f$ to be of the highest degree.  
\end{definition}

For each non-degenerate PET-tuple $A=(L,s,k,\g,\q)$ and polynomial $q\colon(\mathbb{Z}^{L})^{s+1}\to\mathbb{Z}^{d}$, we define the \emph{vdC-operation}, $\partial_{q}A$, according to the following three steps:\footnote{ Actually, the vdC-operation can be defined for any PET-tuple, not just for non-degenerate ones. Similarly, being a procedure that reduces complexity, PET induction can be applied to any family of polynomials. As the expressions of interest in this paper correspond to non-degenerate tuples, we consider only this case.}

\medskip

{\bf Step 1}:  For all $1\leq m\leq k$, let $g^{\ast}_{m}=g^{\ast}_{m+k}=g_{m},$ and  $q^{\ast}_1,\ldots,q^{\ast}_{2k} \colon(\mathbb{Z}^{L})^{s+2}\to\mathbb{Z}^{d}$ be the polynomials defined as $$\displaystyle q^{\ast}_m(n;h_1,\ldots,h_{s+1})=\left\{ \begin{array}{ll} q_m(n+h_{s+1};h_1,\ldots,h_{s})-q(n;h_1,\ldots,h_{s}) & \; , 1\leq m\leq k\\ q_{m-k}(n;h_1,\ldots,h_{s})-q(n;h_1,\ldots,h_{s})  & \; , k+1\leq m\leq 2k\end{array} \right.,$$ i.e., we subtract the polynomial $q$ from the first $k$ polynomials after we have shifted by $h_{s+1}$ the first $L$ variables, and for the second $k$ ones we subtract $q$.\footnote{ In practice, this $q$ will be one of the $q_i$'s of minimum degree.}
Denote $\q^{\ast}=(q^{\ast}_{1},\dots,q^{\ast}_{2k})$.

\medskip

{\bf Step 2}: 
We remove from $q^{\ast}_{1}(n;h_{1},\dots,h_{s+1}),\dots,q^{\ast}_{2k}(n;h_{1},\dots,h_{s+1})$ the polynomials which are constant and the corresponding terms with these as iterates (this will be justified via the use of the Cauchy-Schwarz inequality and the fact that the functions $g_m$ are bounded), and then put the non-constant ones in groups $J_{i}=\{\tilde{q}_{i,1},\dots,\tilde{q}_{i,t_{i}}\},$ $ 1\leq i\leq k'$ for some $k',$ $t_{i}\in\mathbb{N}$ such that two polynomials are essentially distinct if and only if they belong to different groups.\footnote{ After removing the constant polynomials, the terms from $A$ that are grouped are of degree 1.} Next, we write $\tilde{q}_{i,j}(n;h_{1},\dots,h_{s+1})=\tilde{q}_{i,1}(n;h_{1},\dots,h_{s+1})+\tilde{p}_{i,j}(h_{1},\dots,h_{s+1})$ for some polynomial $\tilde{p}_{i,j}$ for all $1\leq j\leq t_{i},$ $1\leq i\leq k'$. For convenience, we also relabel $g^{\ast}_{1},\dots, g^{\ast}_{2k}$ accordingly as $\tilde{g}_{i,j}$ for all $1\leq j\leq t_{i},$ $1\leq i\leq k'$.  

\medskip

{\bf Step 3}: 
For all $1\leq i\leq k'$, 
let $q'_{i}=\tilde{q}_{i,1}$ and \[g'_{i}(x;h_{1},\dots,h_{s+1})=\tilde{g}_{i,1}(x;h_{1},\dots,h_{s+1})\prod^{t_{i}}_{j=2}T_{\tilde{p}_{i,j}(h_{1},\dots,h_{s+1})}\tilde{g}_{i,j}(x;h_{1},\dots,h_{s+1}).\]
We set $\q'=(q'_{1},\dots,q'_{k'})$, $\g'=(g'_{1},\dots,g'_{k'})$ and we denote the new PET-tuple by $\partial_{q}A\coloneqq (L,s+1,k',\g',\q')$.\footnote{ Here we  abuse the notation by writing  $\partial_{q}A$ to denote any such tuple, obtained from Step 1 to 3. Strictly speaking, $\partial_{q}A$ is not uniquely defined as the order of the grouping of $q'_{1},\dots,q'_{2k}$ in Step 2 is ambiguous. However, this is done without loss of generality, since  the order does not affect the value of $S(\partial_{q}A, \cdot)$.} It follows from the construction that $\partial_{q}A$ is non-degenerate.

If $q=q_{t}$ for some $1\leq t\leq\ell$, we write $\partial_{t}\q$  instead of $\q'$  to highlight the fact that we have subtracted the polynomial $q_t$; we also write $\partial_{t}A$ instead of $\partial_{q_{t}}A$ to lighten the notation.

We say that the operation $A\to\partial_{t}A$ is \emph{1-inherited} if $q'_{1}=q^{\ast}_{1}$ and $g'_{1}=f_{1}$, i.e., if we did not drop $q^{\ast}_{1}$ or group it with any other $q^{\ast}_{i}$ in Step 2.

\begin{example}\label{ex020}
Let $\bold{p}=(p_{1},p_{2})$
 	with $p_{1}, p_{2}\colon\Z\to\Z^{d}$ be polynomials given by $p_{i}(n)=b_{i,2}n^{2}+b_{i,1}n$ for some $b_{i,1},b_{i,2}\in\Z^{d}$ for $1\leq i\leq 2$ with $b_{1,2},b_{2,2},b_{1,2}- b_{2,2}\neq \bold{0}$ (hence, we have that $L=s=1$ and $k=2$).
 Subtracting $p_2$ in the Step 1 of the vdC operation, we have that $\partial_2{\bf{p}}=(q_{1},q_{2},q_{3})$, is a tuple of 3
   polynomials, $q_{1},q_{2},q_{3}\colon\Z^{2}\to\Z^{d},$  given by
 	\begin{equation}\nonumber
 	\begin{split}
 	& q_{1}(n,h_1)=(b_{1,2}-b_{2,2})n^{2}+2b_{1,2}nh_1+(b_{1,1}-b_{2,1})n+b_{1,1}h_1+b_{1,2}h_1^{2},
 	\\& q_{2}(n,h_1)=2b_{2,2}nh_1+b_{2,1}h_1+b_{2,2}h_1^{2},
 	\\& q_{3}(n,h_1)=(b_{1,2}-b_{2,2})n^{2}+(b_{1,1}-b_{2,1})n,
 	\end{split}
 	\end{equation}
 	where we removed 1 essentially constant polynomial in Step 2 of the vdC operation.\footnote{Here we use $q_i$'s instead of $p'_i$'s in the first step to ease the notation of Example~\ref{ex01} that is given in the next section.}
 Actually, after using a series of vdC-operations, one can convert ${\bf{p}}$ into a PET-tuple of linear polynomials.

 Indeed, if we run the vdC-operation once more by subtracting $q_2$ in the Step 1 of the vdC operation,
 we have that $\partial_2\partial_2{\bf{p}}=(q'_{1},\dots,q'_{4})$
  is a tuple of 4
   polynomials, $q'_{1},\dots,q'_{4}\colon\Z^{3}\to\Z^{d},$ given by
   \begin{equation}\nonumber
   \begin{split}
   & q'_{1}(n,h_{1},h_{2})=(b_{1,2}-b_{2,2})n^{2}+2(b_{1,2}-b_{2,2})nh_{1}+2(b_{1,2}-b_{2,2})nh_{2}+(b_{1,1}-b_{2,1})n+r'_{1}(h_{1},h_{2}),
   \\& q'_{2}(n,h_{1},h_{2})=(b_{1,2}-b_{2,2})n^{2}-2b_{2,2}nh_{1}+2(b_{1,2}-b_{2,2})nh_{2}+(b_{1,1}-b_{2,1})n+r'_{2}(h_{1},h_{2}),
    \\& q'_{3}(n,h_{1},h_{2})=(b_{1,2}-b_{2,2})n^{2}+2(b_{1,2}-b_{2,2})nh_{1}+(b_{1,1}-b_{2,1})n+r'_{3}(h_{1},h_{2}),
   \\& q'_{4}(n,h_{1},h_{2})=(b_{1,2}-b_{2,2})n^{2}-2b_{2,2}nh_{1}+(b_{1,1}-b_{2,1})n+r'_{4}(h_{1},h_{2}),
   \end{split}
   \end{equation}
   where $r'_{i}\colon\Z^{2}\to\Z^{d}, 1\leq i\leq 4,$ are polynomials in $h_1, h_2$, and  we removed 2 essentially constant polynomials (i.e. $q_{2}(n,h_1)-q_{2}(n,h_1)$ and $q_{2}(n+h_{2},h_1)-q_{2}(n,h_1)$) in Step 2 of the vdC operation.

   Finally, if we apply vdC-operation again by subtracting $q'_4$ in Step 1 of the vdC operation, we have that
   $\partial_4\partial_2\partial_2{\bf{p}}=(q''_{1},\dots,q''_{7})$ is a tuple of 7 polynomials, $ q''_{1},\dots,q''_{7}\colon\Z^{4}\to\Z^{d},$ given by
   \begin{equation}\nonumber
   \begin{split}
    & q''_{1}(n,h_{1},h_{2},h_{3})=2b_{1,2}nh_{1}+2(b_{1,2}-b_{2,2})nh_{2}+2(b_{1,2}-b_{2,2})nh_{3}+r''_{1}(h_{1},h_{2},h_{3}),
    \\& q''_{2}(n,h_{1},h_{2},h_{3})=2(b_{1,2}-b_{2,2})nh_{2}+2(b_{1,2}-b_{2,2})nh_{3}+r''_{2}(h_{1},h_{2},h_{3}),
    \\& q''_{3}(n,h_{1},h_{2},h_{3})=2b_{1,2}nh_{1}+2(b_{1,2}-b_{2,2})nh_{3}+r''_{3}(h_{1},h_{2},h_{3}),
    \\& q''_{4}(n,h_{1},h_{2},h_{3})=2(b_{1,2}-b_{2,2})nh_{3}+r''_{4}(h_{1},h_{2},h_{3}),
   \\& q''_{5}(n,h_{1},h_{2},h_{3})=2b_{1,2}nh_{1}+2(b_{1,2}-b_{2,2})nh_{2}+r''_{5}(h_{1},h_{2},h_{3}),
   \\& q''_{6}(n,h_{1},h_{2},h_{3})=2(b_{1,2}-b_{2,2})nh_{2}+r''_{6}(h_{1},h_{2},h_{3}),
   \\& q''_{7}(n,h_{1},h_{2},h_{3})=2b_{1,2}nh_{1}+r''_{7}(h_{1},h_{2},h_{3}),
   \end{split}
   \end{equation} 
   where $r''_{i}\colon\Z^{3}\to\Z^{d},$ $1\leq i\leq 7,$ are polynomials in $h_1, h_2, h_3$, and  we removed 1 essentially constant polynomial   in the Step 2 of the vdC operation. It is clear that $\deg(\partial_4\partial_2\partial_2{\bf{p}})=1$.
 \end{example}	
The vdC-operation provides us with a non-degenerate tuple, the value $S(\cdot, \cdot)$ of which 
satisfies the following:

\begin{proposition}[Proposition~4.1, \cite{DKS}]\label{induction}
	Let	$(X,\mathcal{B},\mu,(T_{g})_{g\in\mathbb{Z}^{d}})$ be a $\mathbb{Z}^{d}$-system, $A=(L,s,k,\g,$ $\q)$ a PET-tuple, and $q\colon(\mathbb{Z}^{L})^{s+1}\to\mathbb{Z}^{d}$ a polynomial. Then $\partial_{q}A$ is  non-degenerate and $S(A,2\tau)\leq 4^{\tau} S(\partial_{q}A,\tau)$ for every $\tau \in \N_0$. 
\end{proposition}

The following crucial result (cf. \cite[Theorem~4.2]{DKS}) shows that when we start with a PET-tuple which is 1-standard for a function, then, after finitely many vdC-operations, we arrive at a new PET-tuple of degree 1 which is still 1-standard for the same function, so we can then use some Host-Kra seminorm to bound the $\limsup$ of the average of interest.
 
 \begin{theorem}\label{PET} 
 	Let	$(X,\mathcal{B},\mu,(T_{g})_{g\in\mathbb{Z}^{d}})$ be a $\mathbb{Z}^{d}$-system and $f\in L^{\infty}(\mu)$. Let $A=(L,s,k,\g,\q)$ be a non-degenerate PET-tuple which is  1-standard for $g_1$. Then, there exist $\rho_{1},\dots,\rho_{t}\in\mathbb{N},$ for some $t\in \mathbb{N}_0,$ such that
 for all $1\leq t'\leq t$, $\partial_{\rho_{t'}}\dots\partial_{\rho_{1}}A$ is a  non-degenerate PET-tuple which is still 1-standard for $g_1$, and that $\partial_{\rho_{t'-1}}\dots\partial_{\rho_{1}}A\to\partial_{\rho_{t'}}\dots\partial_{\rho_{1}}A$ is 1-inherited. Moreover,  $\deg(\partial_{\rho_{t}}\dots\partial_{\rho_{1}}A)=1$.\footnote{ In \cite[Theorem~4.2]{DKS} the PET tuple is not required to be 1-standard nor 1-inherited; this comes at no extra cost as the polynomials chosen at each step to run the vdC-operation are of minimum degree.}
 \end{theorem}

If $A$ is an $m$-standard for the function $f$ PET-tuple then, by rearranging the terms if necessary, one can get a new tuple $A'$ which is 1-standard for $f$ with $S(A,\tau)=S(A',\tau)$. However, if $A$ is semi-standard but not standard for $f$, then the PET-induction does not work well enough to provide an upper bound for $S(A,\tau)$ in terms of some Host-Kra seminorm of $f$. To overcome this difficulty one follows \cite{DKS}. More specifically, using \cite[Proposition~6.3]{DKS}, which is a  ``dimension-increment'' argument, $A$ can be transformed into a new PET-tuple which is 1-standard for $f$ (at the cost of increasing the dimension from $L$ to $2L$ which is harmless for our approach). So, following this procedure, for any fixed function $f,$ we may assume without loss of generality that the corresponding polynomial iterate $p$ is of maximum degree, making the PET-tuple, after potential rearrangement of the terms, $1$-standard for $f$. A combination of the previous results will allow us to obtain the required upper bound for each function.

\section{Finding a characteristic factor}\label{s4}

This lengthy section is dedicated in proving Theorem~\ref{5555}. To this end, we need to show two intermediate results, i.e., Propositions~\ref{pet3} and \ref{pet}, which improve two technical results from \cite{DKS}, namely \cite[Proposition~5.6]{DKS}, and \cite[Proposition~5.5]{DKS} respectively. 

Recall that for a subset $A$ of $\mathbb{Q}^{d}$, $G(A)=\text{span}_{\Q}\{a\in A\}\cap \Z^{d}.$
\begin{convention}
For the rest of the paper, for every $\u=(u_{1},\dots,u_{L})\in(\Q^{d})^{L}$, we denote $G(\u)\coloneqq G(\{u_{1},\dots,u_{L}\}).$
\end{convention}

The first result, which enhances \cite[Proposition~5.6]{DKS}, gives a bound for the average of interest by finite step seminorms (recall Convention \ref{co1} for notions on Host-Kra seminorms). To pass from infinite step seminorms to finite step ones, we use the implications of  Propositions~\ref{bessel} and \ref{bessel3}. 

\begin{proposition}[Bounding averaged Host-Kra seminorms by a single one]\label{pet3}
	Let  $s,s',t\in\mathbb{N}$ and $\c_{m}\colon(\mathbb{Z}^{L})^{s}\to(\mathbb{Z}^{d})^{L}, 1\leq m\leq t,$ be polynomials with $\c_{m}\not\equiv {\bf 0}$  given by
	\begin{equation}\label{34}
	\begin{split}
	\c_{m}(h_{1},\dots,h_{s})=\sum_{a_{1},\dots,a_{s}\in\mathbb{N}_0^{L},\vert a_{1}\vert+\dots+\vert a_{s}\vert\leq s'}h^{a_{1}}_{1}\dots h^{a_{s}}_{s}\cdot \u_{m}(a_{1},\dots,a_{s})
	\end{split}
	\end{equation}
	for some $$\u_{m}(a_{1},\dots,a_{s})=(u_{m,1}(a_{1},\dots,a_{s}),\dots,u_{m,L}(a_{1},\dots,a_{s}))\in(\mathbb{Q}^{d})^{L}.$$
	Denote
	$$H_{m}\coloneqq G(\{u_{m,i}(a_{1},\dots,a_{s})\colon a_{1},\dots,a_{s}\in\mathbb{N}_0^{L}, 1\leq i\leq L\}).$$

	There exists $D\in\N_0$ depending only on $s,s',L$ such that for every $\mathbb{Z}^{d}$-system  $(X,\mathcal{B},\mu, (T_{g})_{g\in\mathbb{Z}^{d}})$ and every $f\in L^{\infty}(\mu)$, 
	\begin{equation}\label{36}
	\begin{split}
	\E^{\square}_{h_{1},\dots,h_{s}\in\mathbb{Z}^{L}}\nnorm{f}_{\{G(\c_{m}(h_{1},\dots,h_{s}))\}_{1\leq m\leq t}}=0\; \text { if }\;\nnorm{f}_{H^{\times D}_{1},\dots,H^{\times D}_{t}}=0.
	\end{split}
	\end{equation}
\end{proposition}
We start by explaining the idea behind Proposition \ref{pet3} with an example:
 
 \begin{example}\label{ex04}
 	Let $p_{1},p_{2}\colon\Z\to\Z^{2}$ be polynomials given by $p_{1}(n)=(n^{2}+n,0)$ and $p_{2}(n)=(0,n^{2})$, and $(X,\mathcal{B},\mu,(T_{g})_{g\in\Z^{2}})$ be a $\Z^{2}$-system.
 Consider the following expression: \[\F\vl\Bigl\Vert\mathbb{E}_{n\in I_{N}} T_{p_{1}(n)}f_1\cdot T_{p_{2}(n)}f_2 \Bigr\Vert_{2}.\]
 
 Put $e_{1}=(1,0)$, $e_{2}=(0,1)$ and $e=(1,-1)$.
 Using \cite[Proposition~5.5]{DKS}, we get 
 \begin{equation}\label{1234}
 	\begin{split}
 		\F\vl\Bigl\Vert\mathbb{E}_{n\in I_{N}} T_{p_{1}(n)}f_1\cdot T_{p_{2}(n)}f_2 \Bigr\Vert_{2}
 		 \leq C \cdot  \overline{\mathbb{E}}^{\square}_{\bold{h}\in\mathbb{Z}^{3}} \nnorm{f_{1}}_{G(\c_{1}(\bold{h})),\dots, G(\c_{7}(\bold{h}))},
 	\end{split}
 \end{equation}
 where $C$ is a universal constant, and
 \begin{equation}\nonumber
 	\begin{split}
 	& \c_{1}(h_{1},h_{2},h_{3})=-2h_{1}e_{1},
 	\\& \c_{2}(h_{1},h_{2},h_{3})=2h_{2}e,
 	\\& \c_{3}(h_{1},h_{2},h_{3})=-2h_{1}e_{1}+2h_{2}e,
 	\\& \c_{4}(h_{1},h_{2},h_{3})=2h_{3}e,
 	\\& \c_{5}(h_{1},h_{2},h_{3})=-2h_{1}e_{1}+2h_{3}e,
 	\\& \c_{6}(h_{1},h_{2},h_{3})=2(h_{2}+h_{3})e,
 	\\& \c_{7}(h_{1},h_{2},h_{3})=-2h_{1}e_{1}+2(h_{2}+h_{3})e.
 	\end{split}
 \end{equation}	
Using \cite[Proposition~5.6]{DKS}, one can show that if $\nnorm{f_{1}}_{e^{\times D}_{1},e^{\times D}}=0$ for all $D\in\N$, then the right hand side  of (\ref{1234}) is 0. In this paper, we strengthen this result by only assuming that $\nnorm{f_{1}}_{e^{\times D}_{1},e^{\times D}}=0$ for some $D\in\N$.

 Indeed, using  Proposition~\ref{bessel3} (for $I=[-N,N]^{3},$ letting $N\to\infty$) and Corollary~\ref{bessel2}, we have that the right hand side of (\ref{1234}) is 0 if
  $$\E_{\bold{h},\bold{h}'\in\Z^{3}}^{\square}\Bigl\Vert \mathbb{E}(f\vert Z_{\{G(\c_{i}(\bold{h}),\c_{j}(\bold{h'}))\}_{1\leq i,j\leq 7}})\Bigr\Vert^{2}_{2}=0.\footnote{Strictly speaking, Proposition~\ref{bessel3} only implies that the right hand side of (\ref{1234}) is 0 if
  $\E_{\bold{h},\bold{h}'\in\Z^{3}}^{\square}\Bigl\Vert \mathbb{E}(f\vert Z_{\{G(\c_{i}(\bold{h}))+G(\c_{j}(\bold{h'}))\}_{1\leq i,j\leq 7}})\Bigr\Vert^{2}_{2}=0.$ However, since $G(\c_{i}(\bold{h}))+G(\c_{j}(\bold{h'}))$ is a finite index subgroup of $G(\c_{i}(\bold{h}),\c_{j}(\bold{h'}))$, we can use Lemma \ref{lemma:basic_properties_G()} (iv) and replace it by the latter as a seminorm subindex.}$$
  On the other hand, for ``almost all'' $\bold{h},\bold{h}'\in\Z^{3}$, the group $G(\c_{i}(\bold{h}),\c_{j}(\bold{h'}))$ 
  equals $\Z e_{1}$ if $i=j=1$, $\Z e$ if $i,j\in\{2,4,6\}$ and $\Z^{2}$ otherwise. So, $Z_{\{G(\c_{i}(\bold{h}),\c_{j}(\bold{h'}))\}_{1\leq i,j\leq 7}}$ is contained in $Z_{e_{1},e^{\times 9},(\Z^{2})^{\times 39}}$,
  which is contained in $Z_{e_{1}^{\times 25},e^{\times 25}}$ by Proposition \ref{prop:basic_properties}. Hence, the right hand side of (\ref{1234}) is $0$ if $\nnorm{f_{1}}_{e_{1}^{\times 25},e^{\times 25}}=0$. 
 \end{example}

 We now prove the general case.
 \begin{proof}[Proof of Proposition \ref{pet3}] We may assume without loss of generality that $\|f\|_{L^{\infty}(\mu)}\leq 1$. 
 	For convenience, denote ${\bf{h}}\coloneqq (h_{1},\dots,h_{s})$.
 	Using Proposition~\ref{bessel3}, and Corollary~\ref{bessel2} for $I=[-N,N]^{sL},$ and then letting $N\to\infty$, we have that 
 		\begin{equation}\label{41}
 		\begin{split}
 		\Bigl(\E^{\square}_{\bold{h}\in\mathbb{Z}^{sL}}\nnorm{f}_{\{G(\c_{m}(\bold{h}))\}_{1\leq m\leq t}}\Bigr)^{2^{t}W}\leq\E^{\square}_{\bold{h}^{1},\dots,\bold{h}^{W}\in \mathbb{Z}^{sL}}\left\Vert \mathbb{E}(f\vert Z_{\{G(\c_{m_{1}}(\bold{h}^{1}),\dots,\c_{m_{W}}(\bold{h}^{W}))\}_{1\leq m_{1},\dots,m_{W}\leq t}})\right\Vert^{2}_{2}
 		\end{split}
 		\end{equation}
 		for all $W=2^{w}, w\in\N_0$. Let $w$ be the smallest integer such that $2^{w}\geq t(s'+1)^{sL}$ and $\Omega$ the set of $(\bold{h}^{1},\dots,\bold{h}^{W})\in (\Z^{sL})^{W}$ such that for every $1\leq m_{1},\dots,m_{W}\leq t$, the group $G(\c_{m_{1}}(\bold{h}^{1}),\dots,\c_{m_{W}}(\bold{h}^{W}))$ contains at least one of $H_{1},\dots,H_{t}$. By \cref{prop:basic_properties},  for every $(\bold{h}^{1},\dots,\bold{h}^{W})\in \Omega$,  
 		$Z_{\{G(\c_{m_{1}}(\bold{h}^{1}),\dots,\c_{m_{W}}(\bold{h}^{W}))\}_{1\leq m_{1},\dots,m_{W}\leq t}}$ is a factor of  $Z_{H^{\times D}_{1},\dots,H^{\times D}_{t}}$ for $D\coloneqq s^{W}$,
 		and thus
$\mathbb{E}(f\vert Z_{\{G(\c_{m_{1}}(\bold{h}^{1}),\dots,\c_{m_{W}}(\bold{h}^{W}))\}_{1\leq m_{1},\dots,m_{W}\leq t}})=0$ since $\nnorm{f}_{H^{\times D}_{1},\dots,H^{\times D}_{t}}=0$.
 		So, by \eqref{41}, it suffices to show that
 		 $\Omega$ is of upper density 1.
 		Let  $\tilde{\bold{h}}\coloneqq (\bold{h}^{1},\dots,\bold{h}^{W})\in (\Z^{sL})^{W}$ and $1\leq m_{1},\dots,m_{W}\leq t$. 
 		By the pigeonhole principle, at least $(s'+1)^{sL}$ many of the $m_{1},\dots,m_{W}$ take the same value. 
 		Assume that $m_{i_{1}}=\dots=m_{i_{W'}}=m$, where
 		$W'\leq (s'+1)^{sL}$ is 
 		the number of $a_{1},\dots,a_{s}\in\mathbb{N}_0^{L}$ with $\vert a_{1}\vert+\dots+\vert a_{s}\vert\leq s'$.
 		Note that $W'$ depends only on $s',s$ and $L$.
 		 Write $\bold{h}^{i}\coloneqq (h_{i,1},\dots,h_{i,s})\in\Z^{sL}$ and consider the $W'\times W'$ matrix
 		$$A_{i_{1},\dots,i_{W'}}(\tilde{\bold{h}})\coloneqq (h_{i_{j},1}^{a_{1}}\dots h_{i_{j},s}^{a_{s}})_{a_{1},\dots,a_{s}\in\mathbb{N}_0^{L},\vert a_{1}\vert+\dots+\vert a_{s}\vert\leq s',1\leq j\leq W'}.$$
 		If $\det(A_{i_{1},\dots,i_{W'}}(\tilde{\bold{h}}))\neq 0$, then by the definition of $\c_{m}(\bold{h})$, we have that 
 		$$G(\c_{m_{1}}(\bold{h}^{1}),\dots,\c_{m_{W}}(\bold{h}^{W}))\supseteq G(\c_{m}(\bold{h}^{i_{1}}),\dots,\c_{m}(\bold{h}^{i_{W'}}))\supseteq H_{m}.$$
 		In conclusion, $\tilde{\bold{h}}\in\Omega$ if for all $1\leq i_{1}< \dots< i_{W'}\leq W$, $\det(A_{i_{1},\dots,i_{W'}}(\tilde{\bold{h}}))\neq 0$.
 		
 		Thus, it suffices to show that for all $1\leq i_{1}< \dots< i_{W'}\leq W$, the set of $\tilde{\bold{h}}\in(\Z^{sL})^{W}$ with $\det(A_{i_{1},\dots,i_{W'}}(\tilde{\bold{h}}))=0$ is of density 0. We may assume without loss of generality that $i_{1}=1,\dots, i_{W'}=W'$. 
 		Note that $\det(A_{1,\dots,W'}(\tilde{\bold{h}}))$ is a polynomial in $h_{i,j}, 1\leq i\leq W',$ $1\leq j\leq s$. 
 		Looking at the term $h_{1,1}^{(s,0,\dots,0)}h_{2,2}^{(s,0,\dots,0)}\dots h_{W',W'}^{(s,0,\dots,0)}$, we see that $\det(A_{1,\dots,W'}(\tilde{\bold{h}}))$ is a non-constant polynomial. Therefore, the set of solutions to $\det(A_{1,\dots,W'}(\tilde{\bold{h}}))=0$ is of 0 density by Lemma~\ref{ag}, completing the argument.
 \end{proof}	
 The second statement, which strengthens \cite[Proposition~5.5]{DKS}, is the following (check Definition~\ref{d21} for the various notions appearing in the statement): 
 
 \begin{proposition}[Bounding the average by averaged Host-Kra seminorms]\label{pet}
	Let $d,k,L\in\mathbb{N},$ $\p=(p_{1},\dots,p_{k})$,  $p_{1},\dots,p_{k}\colon\mathbb{Z}^{L}\to \mathbb{Z}^{d}$  a family of essentially distinct polynomials of degrees at most $K,$ with $p_{i}(n)=\sum_{v\in\mathbb{N}_0^{L}, \vert v\vert \leq K}b_{i,v}n^{v}$ for some $b_{i,v}\in \mathbb{Q}^{d}$.
	There exist $s,s',t_{1},\dots,t_{k}\in\mathbb{N}$ depending only on $d,k,K,L$,\footnote{ The fact that $s,s'$ depend only on $d, k, K, L$ was not stated in \cite[Proposition~5.1]{DKS}, but it can be derived from its proof.}  and polynomials $\c_{i,m}\colon(\mathbb{Z}^{L})^{s}\to(\mathbb{Z}^{d})^{L}$, $1\leq i\leq k$, $1\leq m\leq t_{i}$, with $\c_{i,m}\not\equiv {\bf 0},$ such that the following hold:
	\begin{itemize}
		\item[(i)] (Control of the coefficients) Each $\c_{i,m}$ is of the form
		\[\c_{i,m}(h_{1},\dots,h_{s})=\sum_{a_{1},\dots,a_{s}\in\mathbb{N}_0^{L},\vert a_{1}\vert+\dots+\vert a_{s}\vert\leq s'}h^{a_{1}}_{1}\dots h^{a_{s}}_{s}\cdot \v_{i,m}(a_{1},\dots,a_{s})\]
		for some  $\v_{i,m}(a_{1},\dots,a_{s})=(v_{i,m,1}(a_{1},\dots,a_{s}),\dots,v_{i,m,L}(a_{1},\dots,a_{s}))\in(\mathbb{Q}^{d})^{L},$
		which is a polynomial function in terms of the coefficients of $p_{i}, 1\leq i\leq k,$ and whose degree depends only on $d,k,K,L$.

		In addition, denoting 
 		 \begin{equation}\nonumber
 		 \begin{split}
 		 H_{i,m}\coloneqq G(\{v_{i,m,j}(a_{1},\dots,a_{s})\colon a_{1},\dots,a_{s}\in\mathbb{N}_0^{L},1\leq j\leq L\}),
 		 \end{split}
 		 \end{equation}
 		we have that $H_{i,m}$ contains one of $G_{i,j}(\p)$, $0\leq j\leq k, j\neq i$.
		
		\item[(ii)] (Control of the average) For every $\mathbb{Z}^{d}$-system $(X,\mathcal{B},\mu, (T_{g})_{g\in\mathbb{Z}^{d}})$ and every $f_{1},\dots,$ $f_{k}\in L^{\infty}(\mu)$ bounded by $1,$ we have that 
		\begin{equation}\label{30}
		\F\vl\Bigl\Vert\mathbb{E}_{n\in I_{N}}\prod_{i=1}^k T_{p_{i}(n)}f_i\Bigr\Vert^{2^{t_0}}_{2}
		\leq C \cdot \min_{1\leq i\leq k}\E^{\square}_{h_{1},\dots,h_{s}\in\mathbb{Z}^{L}}\nnorm{f_{i}}_{(G(\c_{i,m}(h_{1},\dots,h_{s})))_{1\leq m\leq t_{i}}},
		\end{equation}
	\end{itemize}
	where $t_0$ and $C>0$ are constants depending only on $\p$.\footnote{ Actually both $t_0$ and $C$ depend on $d, k, L$ and the highest degree of $p_{1},\dots,p_{k}$. More specifically, $t_0$ can be chosen to be the $\max\{t_1,\ldots,t_k\},$ where $t_i$ is the number of vdC-operations that we have to perform in order for the PET tuple to be non-degenerate, $1$-standard for $f_i$ and with degree equal to $1.$ }
\end{proposition}

Proposition~\ref{pet} improves on \cite[Proposition~5.5]{DKS} as the description of the subgroup $H_{i,m}$ is much more precise than that of the set $U_{i,r}(a_{1},\dots,a_{s})$ defined in the latter. The rest of this section is devoted to proving Proposition~\ref{pet}, which is the most technical result of this paper.

\medskip
 
Next we introduce some convenient notation. 
  Let $\q=(q_{1},\dots,q_{\ell})$ be a tuple of polynomials $q_{i}\colon(\Z^{L})^{s+1}\to\Z^{d},$ $1\leq i\leq \ell,$ where
 \[q_{i}(n;h_{1},\dots,h_{s})=\sum_{b,a_{1},\dots,a_{s}\in\mathbb{N}_0^{L}} h^{a_{1}}_{1}\dots h^{a_{s}}_{s}n^{b}\cdot u_{i}(b;a_{1},\dots,a_{s})\]
  for some $u_{i}(b;a_{1},\dots,a_{s})\in\mathbb{Q}^{d},$ $1\leq i\leq \ell$. Then, by writing $$\u(b;a_{1},\dots,a_{s})\coloneqq (u_{1}(b;a_{1},\dots,a_{s}),\dots,u_{\ell}(b;a_{1},\dots,a_{s}))\in(\mathbb{Q}^{d})^{\ell},$$ we can express $\q$ as 
 \[\q(n;h_{1},\dots,h_{s})=\sum_{b,a_{1},\dots,a_{s}\in\mathbb{N}_0^{L}} h^{a_{1}}_{1}\dots h^{a_{s}}_{s}n^{b}\cdot \u(b;a_{1},\dots,a_{s}).\]
 We call $\bold{u}(b;a_{1},\dots,a_{d})$ the \emph{data} of $\bold{q}$ at \emph{level} $(b;a_{1},\dots,a_{d})$, or simply the \emph{level data} of $\bold{q}$.
 
 \medskip
 
 \emph{For the rest of the section, $\p=(p_{1},\dots,p_{k})$ denotes a family of essentially distinct polynomials $ p_{1},\dots,p_{k}\colon\mathbb{Z}^{L}\to \mathbb{Z}^{d}$ of degrees at most $K,$ with $p_{i}(n)=\sum_{v\in\mathbb{N}_0^{L}, \vert v\vert \leq K}b_{i,v}n^{v},$ where $b_{i,v}\in \mathbb{Q}^{d}$.}

\medskip 
 
 One sees that the left hand side of \eqref{30} is $S(A,1)$, where
 	$A$ is the PET-tuple $(L,0,k,\bold{p},$ $(f_{1},\dots,f_{k}))$. 
 To prove Proposition \ref{pet}, we first need to perform a series of vdC-operations to convert $A$ into a PET-tuple $\partial_{\rho_{t}}\dots \partial_{\rho_{1}}A$ of degree 1, and then compare the coefficients of the polynomials in  $A$ with those in $\partial_{\rho_{t}}\dots \partial_{\rho_{1}}A$. Even though the coefficients in the latter are very difficult to compute directly, one can keep track of the connection between them and those of the original polynomial family $\p$.
 This was first achieved in \cite{DKS} by introducing an equivalence relation pertaining to the vdC-operation (see \cite[Section 5.3]{DKS} for details).
 In this paper, we introduce another approach which is more intricate than the one used in \cite{DKS}, but that achieves a better tracking of the coefficients, which in turn gives us a stronger upper bound for the multiple averages. 
 \medskip

Recalling that $b_{w,v}$, $1\leq w\leq k, v\in\N_0^{L}$ are the coefficients that arise from the  family $\p$ (we also put $b_{0,v}\coloneqq\bold{0}\in \Q^{d}$ for all $v\in\N_0^{L}$),
 for $r\in\mathbb{Q}, v\in\N_0^{L}$ and $0\leq i\leq k$, we set
 $$Q_{r,i,v}\coloneqq \{r(b_{w,v}-b_{i,v})\colon 1\leq w\leq k\}.$$
 
 \begin{definition}[Types and symbols of level data]
 For all $b,a_{1},\dots,a_{s},v\in\N_0^{L}$, $r\in\mathbb{Q}$, and $0\leq i\leq k$, we say that $\u(b;a_{1},\dots,a_{s})$ is of \emph{type $(r,i,v)$} if $$u_{1}(b;a_{1},\dots,a_{s}),\dots,u_{\ell}(b;a_{1},\dots,a_{s})\in Q_{r,i,v},\;\text{and}\;\; u_{1}(b;a_{1},\dots,a_{s})=r(b_{1,v}-b_{i,v}).$$ 
 We say that  $\u(b;a_{1},\dots,a_{s})$ is \emph{non-trivial} if at least one of $u_{m}(b;a_{1},\dots,a_{s})$ is nonzero.
 
 Let $\u(b;a_{1},\dots,a_{s})$ be of type $(r,i,v)$. Suppose that \[ (u_{1}(b;a_{1},\dots,a_{s}),\dots,u_{\ell}(b;a_{1},\dots,a_{s}))=
 	(r(b_{w_{1},v}-b_{i,v}),\dots,r(b_{w_{\ell},v}-b_{i,v})) ,\] for some $0\leq w_{1},\dots,w_{\ell}\leq k$. We call $\w:=(w_{1},\dots,w_{\ell})$ a \emph{symbol} of $\u(b;a_{1},\dots,a_{s})$.
 \end{definition}
Note that the definition of type and symbol depend on the prefixed polynomial family $\bf p$ and that we always have $w_{1}=1$.

We use the types and symbols of level data to track the coefficients of PET-tuples. We start with an example to illustrate this concept: 

  \begin{example}\label{ex01}
 Let $\bold{p}=(p_{1},p_{2})$
 be defined as in Example \ref{ex020} and
 let $\bold{q}=\partial_{2}\bold{p}=(q_{1},q_{2},q_{3})$ (recall that by this we mean the polynomial iterates we get after running the vdC-operation, subtracting the second polynomial, $p_2$). 
  In this case, 
    \begin{equation}\nonumber
    \begin{split}
    & \u(0;1)=(b_{1,1},b_{2,1},0) \text{ is of type $(1,0,1)$ and has symbol $(1,2,0)$},
    \\&
    \u(0;2)=(b_{1,2},b_{2,2},0) \text{ is of type $(1,0,2)$ and has symbol $(1,2,0)$},
    \\& \u(1;0)=(b_{1,1}-b_{2,1},0,b_{1,1}-b_{2,1}) \text{ is of type $(1,2,1)$ and has symbol $(1,2,1)$},
 	\\& \u(1;1)=(2b_{1,2},2b_{2,2},0) \text{ is of type $(2,0,2)$ and has symbol $(1,2,0)$},
 	\\& \u(2;0)=(b_{1,2}-b_{2,2},0,b_{1,2}-b_{2,2}) \text{ is of type $(1,2,2)$ and has symbol $(1,2,1)$}.
    \end{split}
    \end{equation}
 \end{example}

 \begin{definition} \label{def:P1--P4}
 Let $S$ denote the set of all $(a,a')\in\N_0^{L}$ such that $a$ and $a'$ are both $\bold{0}$ or both different than $\bold{0}$. Let $\q$ be a polynomial family of degree at least 1. We say that $\q$ satisfies (P1)--(P4) if its level data $\u$ satisfy:

 (P1)  For all $a_{1},\dots,a_{s},b$, there exist $r,i,v$ such that 
 $\u(b;a_{1},\dots,a_{s})$ is of type $(r,i,v)$.

 (P2) Suppose that $\u(b;a_{1},\dots,a_{s})$ is of type $(r,i,v)$, then  $r=\binom{b+a_{1}+\dots+a_{s}}{a_{1},\dots,a_{s}}$  and $v=b+a_{1}+\dots+a_{s}$ (in particular, $r\neq 0$).\footnote{ For $v_{i}=(v_{i,1},\dots,v_{i,L})\in\N_0^{L}, 1\leq i\leq s$, we denote $\binom{v_{1}+\dots+v_{s}}{v_{1},\dots,v_{s-1}}\coloneqq \prod_{j=1}^{L}\frac{(v_{1,j}+\dots+v_{s,j})!}{v_{1,j}!\dots v_{s,j}!}$.}

 (P3) Suppose that $\u(b;a_{1},\dots,a_{s})$ is of type $(r,i,v)$ and  $\u(b';a'_{1},\dots,a'_{s})$ is of type $(r',i',v')$. If  $(a_{1},a'_{1}),\dots,(a_{s},a'_{s})\in S$, then $i=i'$  and $\u(b;a_{1},\dots,a_{s}),$ $\u(b';a'_{1},\dots,a'_{s})$ share a symbol $\bold{w}$.
 
 (P4) For every symbol $(w_{1},\dots,w_{\ell})$ of some $\u(b;a_{1},\dots,a_{s})$, we  have that $w_{1}=1$.

 \end{definition}
 
\medskip 
 
Once again, properties (P1)--(P4) are taken with respect to the prefixed polynomial family $\bold{p}.$ It is obvious that $\bold{p}$ itself satisfies (P1)--(P4). An important feature of the type and symbol of a level data is that properties (P1)--(P4) are preserved under vdC-operations.
 
   \begin{example}\label{ex03}
  	We will verify that the polynomial family $\bold{q}=\partial_{2}\p$ in Example \ref{ex01} satisfies all of (P1)--(P4).	Indeed, (P1) holds as all $\u(0;1),\u(0;2),\u(1;0),\u(1;1),\u(2;0)$ have a type.  For all $0\leq a,b\leq 2, a+b\leq 2$, if $\u(b;a)$ is of type $(r,i,v)$, then it is not hard to see that $r=\binom{b+a}{a}$, so (P2) holds. (P3) can be verified by comparing the types and symbols of the pairs $(\u(0,1),\u(0,2))$ and $(\u(1,0),\u(2,0))$. Finally, (P4) also holds since the first entry of every symbol in $\bold{q}$  is 1.
  \end{example}
 
   We caution the reader that the symbol and type may not be unique if the coefficients $b_{i,v}$ satisfy some algebraic relations. For example, in Example \ref{ex01}, if $b_{1,1}=b_{2,1}$, then both $(1,2,0)$ and $(1,1,0)$ are symbols of $\u(0;1)=(b_{1,1},b_{2,1},0)$. However the following result says that there is always a way to choose symbols and types so that  properties (P1)--(P4) are preserved under vdC-operations:

 \begin{proposition}\label{PET2} 
 	Let $A=(L,s,\ell,\g,\q)$ be a non-degenerate PET-tuple and $1\leq \rho\leq \ell$. 
 	Assume that $A\to\partial_{\rho}A$ is 1-inherited.
 	Suppose that there exists a choice of symbols and types for $A$ which satisfy (P1)--(P4), then there is also a choice of symbols and types for $\partial_{\rho}A$ which satisfy (P1)--(P4). 
 \end{proposition}
 \begin{proof}
 Suppose that $\q=(q_{1},\dots,q_{\ell})$. Denote $\q^{\ast}=(q^{\ast}_{1},\dots,q^{\ast}_{2\ell})$, where
 for all $1\leq i\leq \ell,$  $$q^{\ast}_{i}(n;h_{1},\dots,h_{s+1})=q_{i}(n+h_{s+1};h_{1},\dots,h_{s})-q_{\rho}(n;h_{1},\dots,h_{s})$$
 and
 $$q^{\ast}_{\ell+i}(n;h_{1},\dots,h_{s+1})=q_{i}(n;h_{1},\dots,h_{s})-q_{\rho}(n;h_{1},\dots,h_{s}).$$ 
 	Assuming that 
 	\[q_{i}(n;h_{1},\dots,h_{s})=\sum_{b,a_{1},\dots,a_{s}\in\mathbb{N}_0^{L}} h^{a_{1}}_{1}\dots h^{a_{s}}_{s}n^{b}\cdot u_{i}(b;a_{1},\dots,a_{s})\]
 	for all
 	$1\leq i\leq\ell,$ we may
 	write
 	$\q$ as 
 	\[\q(n;h_{1},\dots,h_{s})=\sum_{b,a_{1},\dots,a_{s}\in\mathbb{N}_0^{L}} h^{a_{1}}_{1}\dots h^{a_{s}}_{s}n^{b}\cdot \u(b;a_{1},\dots,a_{s}),\]
 	and define $\u^{\ast}(b;a_{1},\dots,a_{s+1})$ in a similar way.
 		
 	One can immediately check that

 	\[q_{i}(n+h_{s+1};h_{1},\dots,h_{s})=\sum_{b,a_{1},\dots,a_{s+1}\in\mathbb{N}_0^{L}} h^{a_{1}}_{1}\dots h^{a_{s+1}}_{s+1}n^{b}\cdot \binom{b+a_{s+1}}{b}u_{i}(b+a_{s+1};a_{1},\dots,a_{s}).\footnote{ For $a=(a_{1},\dots,a_{L}), b=(b_{1},\dots,b_{L})\in\mathbb{N}_0^{L}$, $\binom{a}{b}$ denotes the quantity $\prod_{m=1}^{L}\binom{a_{m}}{b_{m}}$.}\]
 	Then
 	\begin{equation}\label{i2}
 	u^{\ast}_{i}(b;a_{1},\dots,a_{s+1})=\left\{
 	\begin{aligned}
 	u_{i}(b;a_{1},\dots,a_{s})-u_{\rho}(b;a_{1},\dots,a_{s}) & & (a_{s+1}=\bold{0}) \\
 	\binom{b+a_{s+1}}{b}u_{i}(b+a_{s+1};a_{1},\dots,a_{s}) & & (a_{s+1}\neq \bold{0})
 	\end{aligned}
 	\right. 
 	\end{equation}
 	and
 		\begin{equation}\label{i1}
 	u^{\ast}_{i+\ell}(b;a_{1},\dots,a_{s+1})=\left\{
 	\begin{aligned}
 	u_{i}(b;a_{1},\dots,a_{s})-u_{\rho}(b;a_{1},\dots,a_{s}) & & (a_{s+1}=\bold{0}) \\
 	\bold{0} & & (a_{s+1}\neq \bold{0})
 	\end{aligned}
 	\right. .
 	\end{equation}
 	
 	We first show that $\q$ satisfying (P1)--(P4) implies the same for $\q^{\ast}$.
 	
 	Fix $a_{1},\dots,a_{s+1}$. Since $\q$ satisfies (P1)--(P4), there exist $0\leq i\leq k$ and a symbol $\w$ such that 
 	for all $b$, there is $v\in\N_0^{L}$ and $r\in\mathbb{Q}$ such that $\u(b+a_{s+1};a_{1},\dots,a_{s})$ is of type $(r,i,v)$ and has symbol $\w=(1,w_{2},\dots,w_{\ell})$. 

We have that  $u_{m}(b;a_{1},\dots,a_{s})=r(b_{w_{m},v}-b_{i,v})$.
By (\ref{i2}) and (\ref{i1}), it is not hard to see that $\u^{\ast}(b;a_{1},\dots,a_{s},a_{s+1})$ is
 	\begin{equation}\label{i3}
 	\left\{
 	\begin{aligned}
 	\text{of type $(r,w_{\rho},v)$ and has symbol $(\w,\w)$   } & & (a_{s+1}=\bold{0}) \\
 	\text{of type $(r\binom{b+a_{s+1}}{b},i,v)$ and has symbol $(\w,i,\dots,i)$} & & (a_{s+1}\neq \bold{0})
 	\end{aligned}
 	\right. .
 	\end{equation}
 	So each $\u^{\ast}(b;a_{1},\dots,a_{s},a_{s+1})$ is of type $(\ast,\ast,v)$ for some $v$ (meaning that it is of type $(r',i',v)$ for some $r'\in\mathbb{Q}$ and $i'\in\N$) and thus $\q^{\ast}$ satisfies (P1). 
 	
 	To check (P2) for $\q^{\ast}$,
 	suppose that $\u^{\ast}(b;a_{1},\dots,a_{s+1})$ is
 	of type $(r,i,v)$. By (\ref{i3}), we have that $\u(b+a_{s+1};a_{1},\dots,a_{s})$ is
 	of type $(r\binom{b+a_{s+1}}{b}^{-1},i',v)$ for some $i'$. Since
 	$\q$ satisfies (P2), we have the same for $\q^{\ast}$ as
 	$$r=\binom{b+a_{s+1}}{b}\binom{(b+a_{s+1})+a_{1}+\dots+a_{s}}{a_{1},\dots,a_{s}}=\binom{b+a_{1}+\dots+a_{s+1}}{a_{1},\dots,a_{s+1}}$$
 	and 
 	$$v=(b+a_{s+1})+a_{1}+\dots+a_{s}=b+a_{1}+\dots+a_{s+1}.$$
 	
 To show (P3),	let $(a_{1},a'_{1}),\dots,(a_{s+1},a'_{s+1})\in S$ and suppose that $\u(b+a_{s+1};a_{1},\dots,a_{s})$ is of type $(r,i,v)$ and $\u(b'+a'_{s+1};a'_{1},\dots,a'_{s})$ of type $(r',i',v')$.
 Since $\q$ satisfies (P3), we have that $i=i'$ and that $\u(b+a_{s+1};a_{1},\dots,a_{s})$ and $\u(b'+a'_{s+1};a'_{1},\dots,a'_{s})$ can be made to have the same symbol $\w$. Since $,(a_{s+1},a'_{s+1})\in S$, by (\ref{i3}) have  that $\q^{\ast}$ satisfies (P3).

 	Finally, it is straightforward from (\ref{i3}) that Property (P4) is also heritable.
 	Since the polynomials in $\partial_{\rho}A$ are obtained by removing some terms from the tuple $\q^{\ast}$ (but not removing the first one, since $A\to\partial_{\rho} A$ is 1-inherited), the fact that  $\q^{\ast}$ satisfies (P1)--(P4) implies that $\partial_{\rho}A$ also satisfies (P1)--(P4).
 \end{proof}

For the family of essentially distinct polynomials $\p=(p_{1},\dots,p_{k}),$  Proposition \ref{PET2} implies that $\partial_{i_{k}}\dots \partial_{i_{1}}\bold{p}$ satisfies (P1)--(P4) for all $k,i_{1},\dots,i_{k}\in\N$. In the special case when $\partial_{i_{k}}\dots \partial_{i_{1}}\bold{p}$ is of degree 1, properties (P1)--(P4) provide us with some information on the seminorm we use to bound $S(\partial_{i_{k}}\dots \partial_{i_{1}}A,1)$.
To be more precise, if (P1)--(P4) hold for some non-degenerate $\q$, then there is some connection between the level data of $\q$ and the groups $G_{1,0}(\p),G_{1,2}(\p),\dots,G_{1,k}(\p)$:

 \begin{proposition}\label{12-3}
 	Suppose that (P1)--(P4) hold for some non-degenerate $\q$. Then for all $0\leq m\leq \ell, m\neq 1$, the group
 	 	$$H_{1,m}(\q)\coloneqq G(\{u_{1}(b;a_{1},\dots,a_{s})-u_{m}(b;a_{1},\dots,a_{s})\colon (b,a_{1},\dots,a_{s})\in(\N_0^{L})^{s+1},b\neq \bold{0}\})$$ 
 contains at least one of the groups $G_{1,j}(\p), 0\leq j\leq k, j\neq 1$.
 \end{proposition}	
 
 We remark that although Proposition \ref{12-3} holds for all non-degenerate $\q$, we will use it for the case $\deg(\q)=1$.
 We first give an example to explain the idea behind it.

 \begin{example}\label{ex02}
 	Let $\bold{p}=(p_{1},p_{2})$
 	be  as in Example \ref{ex020}.
 	Then $G_{1,0}(\p)=G(b_{1,2})$ and  $G_{1,2}(\p)=G(b_{1,2}-b_{2,2})$.
 	Let $\u(b;a)$ be the level data of $\partial_{2}\bold{p}$. Then 
 	\begin{equation}\nonumber
 	\begin{split}
 	& \u(1;0)=(b_{1,1}-b_{2,1})\cdot(1,0,1) \text{ is of type $(1,2,1)$ and has symbol $(1,2,1)$},
 	\\& \u(1;1)=2b_{1,2}\cdot(1,0,0)+2b_{2,2}\cdot(0,1,0) \text{ is of type $(2,0,2)$ and has symbol $(1,2,0)$},
 	\\& \u(2;0)=(b_{1,2}-b_{2,2})\cdot(1,0,1) \text{ is of type $(1,2,2)$ and has symbol $(1,2,1)$}.
 	\end{split}
 	\end{equation}
 	Here we will not compute $\u(b;a)$ for $b=0$ as it is irrelevant to our purposes.
 	It is easy to see that $H_{1,0}(\partial_{2}\p)=G(b_{1,1}-b_{2,1},b_{1,2},b_{1,2}-b_{2,2})\supseteq G_{1,0}(\p)\cup G_{1,2}(\p)$, $H_{1,2}(\partial_{2}\p)=G(b_{1,1}-b_{2,1},b_{1,2}-b_{2,2})\supseteq G_{1,2}(\p)$,
 	$H_{1,3}(\partial_{2}\p)=G(b_{1,2})=G_{1,0}(\p)$.
 	So, Proposition \ref{12-3} holds for $\partial_{2}\p$. 

    Let $\u'(b;a_{1},a_{2})$ denote the level data of $\partial_{2}\partial_{2}\bold{p}$. Then 
   \begin{equation}\nonumber
   \begin{split}
   & \u'(1;0,0)=(b_{1,1}-b_{2,1})\cdot(1,1,1,1) \text{ is of type $(1,2,1)$ and has symbol $(1,1,1,1)$},
   \\& \u'(1;1,0)=2(b_{1,2}-b_{2,2})\cdot(1,0,1,0)-2b_{2,2}\cdot(0,1,0,1) \text{ has type $(2,2,2)$ and sym. $(1,0,1,0)$},
   \\& \u'(1;0,1)=2(b_{1,2}-b_{2,2})\cdot(1,1,0,0) \text{ is of type $(2,2,2)$ and has symbol $(1,1,2,2)$},
   \\& \u'(2;0,0)=(b_{1,2}-b_{2,2})\cdot(1,1,1,1) \text{ is of type $(1,2,2)$ and has symbol $(1,1,1,1)$}.
   \end{split}
   \end{equation}
   (We do not compute the types and symbols for $\u'(b;a_{1},a_{2})$ for $b=0$.) 
   It is easy to see that $H_{1,0}(\partial_{2}\partial_{2}\p)=G(b_{1,1}-b_{2,1},b_{1,2}-b_{2,2})\supseteq G_{1,2}(\p)$, $H_{1,2}(\partial_{2}\partial_{2}\p)=G(b_{1,2})= G_{1,0}(\p),$
   $H_{1,3}(\partial_{2}\partial_{2}\p)=G(b_{1,2}-b_{2,2})\supseteq G_{1,2}(\p),$ $H_{1,4}(\partial_{2}\partial_{2}\p)=G(b_{1,2},b_{1,2}-b_{2,2})\supseteq G_{1,0}(\p)\cup G_{1,2}(\p)$.
   So, Proposition \ref{12-3} holds for $\partial_{2}\partial_{2}\p$. 
   
   Finally, let $\u''(b;a_{1},a_{2},a_{3})$ denote the level data of $\partial_{4}\partial_{2}\partial_{2}\bold{p}$. Then $\deg(\partial_{4}\partial_{2}\partial_{2}\p)=1$ and
   \begin{equation}\nonumber
   \begin{split}
   & \u''(1;0,0,0)=(0,0,0,0,0,0,0) \text{ is trivial},
   \\& \u''(1;1,0,0)=2b_{1,2}\cdot(1,0,1,0,1,0,1) \text{ is of type $(2,0,2)$ and symbol $(1,0,1,0,1,0,1)$},
   \\& \u''(1;0,1,0)=2(b_{1,2}-b_{2,2})\cdot(1,1,0,0,1,1,0) \text{ has type $(2,2,2)$ and symbol $(1,1,2,2,1,1,2)$},
   \\& \u''(1;0,0,1)=2(b_{1,2}-b_{2,2})\cdot(1,1,1,1,0,0,0) \text{ has type $(2,2,2)$ and symbol $(1,1,1,1,2,2,2)$}.
   \end{split}
   \end{equation}
   (Once more, we do not compute the types and symbols for $\u''(b;a_{1},a_{2},a_{3})$ for $b=0$.)
   It is easy to see that $H_{1,0}(\partial_{4}\partial_{2}\partial_{2}\p)=H_{1,4}(\partial_{4}\partial_{2}\partial_{2}\p)=H_{1,6}(\partial_{4}\partial_{2}\partial_{2}\p)=G(b_{1,2},b_{1,2}-b_{2,2})\supseteq G_{1,1}(\p)$, $H_{1,2}(\partial_{4}\partial_{2}\partial_{2}\p)=G(b_{1,2})= G_{1,1}(\p)$, and
   $H_{1,3}(\partial_{4}\partial_{2}\partial_{2}\p)=H_{1,5}(\partial_{4}\partial_{2}\partial_{2}\p)=H_{1,7}(\partial_{4}\partial_{2}\partial_{2}\p)=G(b_{1,2}-b_{2,2})= G_{1,2}(\p)$.
   So, Proposition \ref{12-3} holds for $\partial_{4}\partial_{2}\partial_{2}\p$. 
 \end{example}

 To briefly explain why Proposition \ref{12-3} holds for Example \ref{ex02}, we explain, for convenience, why $H_{1,0}(\q)$ contains either $G_{1,0}(\p)$ or $G_{1,2}(\p)$ for $\q=\partial_{2}\p,\partial_{2}\partial_{2}\p$ and $\partial_{4}\partial_{2}\partial_{2}\p$.
 Let $\u(b;a_{1},$ $\dots,a_{s})$ be a level data of type $(r,i,v)$ and symbol $\w$, and $\u(b';a'_{1},\dots,a'_{s})$ be a level data of type $(r',i',v')$ and symbol $\w'$.
 We say that the level data $\u(b';a'_{1},\dots,a'_{s})$ \emph{dominates} (or \emph{strictly dominates}) 
 $\u(b;a_{1},\dots,a_{s})$ if $i=i'$, $\w=\w'$ and $\vert v'\vert\geq \vert v\vert$ (or $\vert v'\vert>\vert v\vert$ respectively).
In Example \ref{ex02}, it is not hard to see that for all $b\in \N,$ $a_{1},a_{2}\in\N_0$, if $\u'(b;a_{1},a_{2})$ is not of type $(\ast,\ast,2)$,
then there exist $b'\in \N,$ $a'_{1},a'_{2}\in\N_0$ such that $\u'(b';a'_{1},a'_{2})$ strictly dominates $\u'(b;a_{1},a_{2})$ (in this example, $\u'(1;0,0)$ is strictly dominated by $\u'(2;0,0)$). 
Similar conclusions hold for $\u(b;a)$ and $\u''(b;a_{1},a_{2},a_{3})$. In other words, the group $H_{1,0}(\q)$ must contain the elements of a level data of type $(\ast,\ast,2)$, and thus it must contain one of $G_{1,0}(\p)$ and $G_{1,2}(\p)$.
 
 \medskip
 
 We are now ready to prove Proposition \ref{12-3}. The main point is that given any nontrivial level data $\u(b;a_{1},\dots,a_{s})$, we can find another one, $\u(b';a'_{1},\dots,a'_{s}),$ which dominates $\u(b;a_{1},\dots,a_{s})$ and is of type $(\ast,\ast,v)$ with $\vert v\vert$ being as large as possible (in this step we need to exploit the properties (P1)--(P4)). After that, we use the information of the ``top'' level data $\u(b';a'_{1},\dots,a'_{s})$ to conclude.
 \begin{proof}[Proof of Proposition \ref{12-3}]
 	
 		We start with a claim.  Recall that $S$ denotes the set of all $(a,a')\in\N_0^{L}$ such that $a$ and $a'$ are both $\bold{0}$ or both different than $\bold{0}$ (check Definition~\ref{d21} for notation).
 	
 	\medskip
 	
 			\textbf{Claim.}	
 			Let $d,s\in\N$ and $b,a_{1},\dots,a_{s},v\in\N_0^{L}$. If $\vert v\vert\geq \vert b+a_{1}+\dots+a_{s}\vert$, then
 			 there exist $b',a'_{1},\dots,a'_{s}\in\N_0^{L}$ such that $(a_{1},a'_{1}),\dots,(a_{s},a'_{s})\in S, \vert b'\vert\geq \vert b\vert$ and $b'+a'_{1}+\dots+a'_{s}=v$. 
 			 
 		To show the claim, we may first assume that $\vert v\vert=\vert b+a_{1}+\dots+a_{s}\vert.$ Indeed, write $x=(x_{1},\ldots,x_{L})$ for $x=v,b,a_{1},\dots,a_{s}$. If $c\coloneqq \vert v\vert-\vert b+a_{1}+\dots+a_{s}\vert>0$, then we write $b'=b+(c,0,\dots,0)$. Then $\vert b'\vert>\vert b\vert$ and $\vert v\vert=\vert b'+a_{1}+\dots+a_{s}\vert.$
 		
 		It suffices to show that if  $\vert v-(b+a_{1}+\dots+a_{s})\vert>0$, then there exist $b',a'_{1},\dots,a'_{s}\in\N_0^{L}$ such that $(a_{1},a'_{1}),\dots,(a_{s},a'_{s})\in S$, $\vert b'\vert=\vert b\vert$, and $\vert v-(b'+a'_{1}+\dots+a'_{s})\vert<\vert v-(b+a_{1}+\dots+a_{s})\vert$.
 		Since $\vert v\vert=\vert b+a_{1}+\dots+a_{s}\vert$, $\vert v-(b+a_{1}+\dots+a_{s})\vert$ is at least 2, so, there exist $1\leq i,j\leq L, i\neq j$ such that the $t$-th coordinate of  $v-(b+a_{1}+\dots+a_{s})$ is at least $1$ for $t=i$ and is at most $-1$ for $t=j$. We may assume without loss of generality that $i=1$ and $j=2$. Then we have that $v_{1}\geq 1$, and one of $b_{2},a_{1,2},\dots,a_{s,2}$ is at least 1. If $b_{2}\geq 1$, then $b'=b+(1,-1,0,\dots,0)\in\N_0^{L}, a'_{i}=a_{i},1\leq i\leq s$ satisfy the requirement. If one of $a_{i,2}$ is positive, then $b'=b, a'_{i}=a_{i}+(1,-1,0,\dots,0)\in\N_0^{L}, a'_{j}=a_{j},1\leq j\leq s, j\neq i$ satisfy the requirement. This proves the claim.
 		
 \medskip		
 		
 			Consider the group $H_{1,m}(\q)$ for some $0\leq m\leq \ell, m\neq 1$. Since $\q$ is non-degenerate, there exist some $b,a_{1},\dots,a_{s}\in\N_0^{L}, \vert b\vert\geq 1$ such that $u_{1}(b;a_{1},\dots,a_{s})-u_{m}(b;a_{1},\dots,a_{s})\neq \bold{0}$. By (P1), we may assume that $\u(b;a_{1},\dots,a_{s})$ is of type $(r,i,v)$ and has symbol $(w_{1},\dots,w_{\ell})$. Since $w_{1}=1$,
 			$$u_{1}(b;a_{1},\dots,a_{s})-u_{m}(b;a_{1},\dots,a_{s})=r(b_{1,v}-b_{w_{m},v}).$$
 			Recall that for $\p=(p_{1},\dots,p_{k})$, $d_{1,0}=\deg(p_{1})$ and $d_{1,j}=\deg(p_{1}-p_{j})$ for $2\leq j\leq \ell$.  Since  $u_{1}(b;a_{1},\dots,a_{s})-u_{m}(b;a_{1},\dots,a_{s})\neq \bold{0}$, we have that $w_{m}\neq 1$ and $\vert v\vert=\vert b+a_{1}+\dots+a_{s}\vert\leq d_{1,w_{m}}$.
 
 			By the Claim, for all $v'\in\N_0^{L}$ with $\vert v'\vert=d_{1,w_{m}}$, there exist $b',a'_{1},\dots,a'_{s}\in\N_0^{L}$ such that $(a_{1},a'_{1}),\dots,(a_{s},a'_{s})\in S$, $\vert b'\vert\geq 1$ and $b'+a'_{1}+\dots+a'_{s}=v'$. By (P2) and (P3), $\u(b';a'_{1},\dots,a'_{s})$ is of type $(r',i,v'), r'\neq 0$ and has symbol $(w_{1},\dots,w_{\ell})$ (i.e. $\u(b';a'_{1},\dots,a'_{s})$ dominates $\u(b;a_{1},\dots,a_{s})$). So  $$u_{1}(b';a'_{1},\dots,a'_{s})-u_{m}(b';a'_{1},\dots,a'_{s})=r'(b_{1,v'}-b_{w_{m},v'}).$$
 			 In other words, for all $v'\in\N^{L}$ with $\vert v'\vert=d_{1,w_{m}}$, the group  $H_{1,m}(\q)$ contains a nonzero multiple of $b_{1,v'}-b_{w_{m},v'}$. So this group contains $G_{1,w_{m}}(\p)$ and we are done.
 \end{proof} 		
 
We are now ready to prove Proposition \ref{pet}. 
 
 \begin{proof}[Proof of Proposition \ref{pet}]
Let $A$ denote the PET-tuple $(L,0,k,(p_{1},\dots,p_{k}),(f_{1},\dots,f_{k}))$. Then for all $\tau>0$,
 		\[S(A,\tau)=\F\vl\Bigl\Vert\ei\prod_{m=1}^{k}T_{p_{m}(n)}f_{m}\Bigr\Vert^{\tau}_{L^{2}(\mu)}.\]
By assumption, $A$ is non-degenerate.
 		We only prove (\ref{30}) for $f_{1}$ as the other cases are similar.
 		We first assume that $A$ is 1-standard for $f_{1}$. By Theorem \ref{PET}, there exist $t\in\N_{0},$ depending only on $d,k,K,L,$ and finitely many
 		vdC-operations $\partial_{\rho_{1}},\dots,\partial_{\rho_{t}}, \rho_{1},\dots,\rho_{t}\neq 1$ such that
 		for all $1\leq t'\leq t$, $\partial_{\rho_{t'}}\dots\partial_{\rho_{1}}A$ is non-degenerate and 1-standard for $f_{1}$, and that $\partial_{\rho_{t'-1}}\dots\partial_{\rho_{1}}A\to \partial_{\rho_{t'}}\dots\partial_{\rho_{1}}A$ is 1-inherited. Moreover,
 		$A':=\partial_{\rho_{t}}\dots\partial_{\rho_{1}}A$ is of degree 1.
 		By Proposition \ref{induction}, $S(A, 2^{t})\leq C\cdot S(A',1)$ for some $C>0$ that depends only on $t$. We may assume that
 	\[S(A',1)=\overline{\mathbb{E}}^{\square}_{h_{1},\dots,h_{s}\in\mathbb{Z}^{L}}\F\vl\Bigl\Vert\ei \prod_{m=1}^{\ell}T_{\textbf{d}_{m}(h_{1},\dots,h_{s})\cdot n+r_{m}(h_{1},\dots,h_{s})}g_{m}(x;h_{1},\dots,h_{s})\Bigr\Vert_{2}\]
 		for some $s,\ell\in\mathbb{N}$, functions $g_{1},\dots,g_{\ell}\colon X\times (\mathbb{Z}^{L})^{s}\to\mathbb{R},$ where $g_{1}(\cdot;h_{1},\dots,h_{s})=f_1,$ such that each $g_{m}(\cdot;h_{1},\dots,h_{s})$ is an $L^{\infty}(\mu)$ function bounded by $1,$ and polynomials $\textbf{d}_{m}\colon(\mathbb{Z}^{L})^{s}\to(\mathbb{Z}^{d})^{L}$ and $r_{m}\colon(\mathbb{Z}^{L})^{s}\to\mathbb{Z}^{d},$ $1\leq m\leq \ell$, where $\textbf{d}_{m}, r_{m}$ take values vectors with integer coordinates as vdC-operations transform integer-valued polynomials to integer-valued polynomials.

 		Let $\c_{1}=-\textbf{d}_{1}$ and $\c_{m}=\textbf{d}_{m}-\textbf{d}_{1}$ for $m\neq 1$.
 		Since $A'$ is non-degenerate, we have that $\c_{1},\dots,\c_{s}\not\equiv{\bf 0}$. By \cite[Proposition~6.1]{DKS}, if $\ell\geq 2$,  we also have that 
 		\begin{eqnarray*}
 			S(A',1) & \leq & C'\cdot \overline{\mathbb{E}}^{\square}_{h_{1},\dots,h_{s}\in\mathbb{Z}^{L}} \nnorm{T_{r_{1}(h_{1},\dots,h_{s})}f_1}_{\{G(\c_{i}(h_{1},\dots,h_{s}))\}_{1\leq i\leq \ell}} \\
 			& = & C'\cdot \overline{\mathbb{E}}^{\square}_{h_{1},\dots,h_{s}\in\mathbb{Z}^{L}}\nnorm{f_{1}}_{\{G(\c_{i}(h_{1},\dots,h_{s}))\}_{1\leq i\leq \ell}}
 		\end{eqnarray*}
 		for some $C'>0$ depending only on the polynomials $p_{1},\dots,p_{k}$. If $\ell=1$, using the mean ergodic theorem (see for example \cite[Theorem~2.3]{DKS}) and \cite[Lemma~2.4 (iv), (vi)]{DKS}, we have
 		\begin{eqnarray*}
 			S(A',1) 
 			& = &  \overline{\mathbb{E}}^{\square}_{h_{1},\dots,h_{s}\in\mathbb{Z}^{L}} \Vert \mathbb{E}(T_{r_{1}(h_{1},\dots,h_{s})}f_1\vert \mathcal{I}(\c_{1}(h_{1},\dots,h_{s})))\Vert_{2} \\
 				& = &  \overline{\mathbb{E}}^{\square}_{h_{1},\dots,h_{s}\in\mathbb{Z}^{L}} \Vert \mathbb{E}(f_1\vert \mathcal{I}(\c_{1}(h_{1},\dots,h_{s})))\Vert_{2} \\
 			& = &  \overline{\mathbb{E}}^{\square}_{h_{1},\dots,h_{s}\in\mathbb{Z}^{L}} \nnorm{f_1}_{\c_{1}(h_{1},\dots,h_{s})} \\
 			& \leq &  \overline{\mathbb{E}}^{\square}_{h_{1},\dots,h_{s}\in\mathbb{Z}^{L}} \nnorm{f_1}_{G(\c_{1}(h_{1},\dots,h_{s})),\c_{1}(h_{1},\dots,h_{s})} \\
 			& = & \overline{\mathbb{E}}^{\square}_{h_{1},\dots,h_{s}\in\mathbb{Z}^{L}} \nnorm{f_1}_{G(\c_{1}(h_{1},\dots,h_{s}))^{\times 2}}. 
 		\end{eqnarray*}
 		Combining this with the fact that $S(A,2^{t})\leq C\cdot S(A',1)$, we get (\ref{30}).\footnote{
 			We are using \cite[Proposition~6.1]{DKS} for $\ell \geq 2$ (which is incorrectly stated for $\ell \geq 1$ in \cite{DKS} but does not affect any other result in that work). }  
 		
 		We now consider the groups $H_{1,m}, 0\leq m\leq \ell, m\neq 1$.
 		Suppose that $$\textbf{c}_{m}(h_{1},\dots,h_{s})=\sum_{a_{1},\dots,a_{s}\in\mathbb{N}_0^L}h^{a_{1}}_{1}\dots h^{a_{s}}_{s}\cdot \v_{m}(a_{1},\dots,a_{s})$$
 		and
 		$$\textbf{d}_{m}(h_{1},\dots,h_{s})=\sum_{a_{1},\dots,a_{s}\in\mathbb{N}_0^L}h^{a_{1}}_{1}\dots h^{a_{s}}_{s}\cdot \u_{m}(a_{1},\dots,a_{s})$$
 		for some vectors $\u_{m}(a_{1},\dots,a_{s})=(u_{m,1}(a_{1},\dots,a_{s}),\dots,u_{m,L}(a_{1},\dots,a_{s}))$, and $\v_{m}(a_{1},\dots,a_{s})$ $=(v_{m,1}(a_{1},\dots,a_{s}),\dots,v_{m,L}(a_{1},\dots,a_{s}))\in(\mathbb{Q}^{d})^{L}$ with all but finitely many terms being zero for each $m$.
 		Obviously  $A$ satisfies (P1)--(P4).
 		Since $A'=\partial_{\rho_{t}}\dots\partial_{\rho_{1}}A$,  by Proposition \ref{PET2},  $A'$ satisfies (P1)--(P4). Since $\deg(A')=1$, by Proposition \ref{12-3}, 
 		 each of  
 		 \begin{equation}\nonumber
 		 \begin{split}
 		 &\quad G(\{v_{1,j}(a_{1},\dots,a_{s})\colon (a_{1},\dots,a_{s})\in(\N_0^{L})^{s}, 1\leq j\leq L\})
 		 \\&=G(\{u_{1,j}(a_{1},\dots,a_{s})\colon (a_{1},\dots,a_{s})\in(\N_0^{L})^{s}, 1\leq j\leq L\})=H_{1,0}
 		 \end{split}
 		 \end{equation}
 		 and
 		 \begin{equation}\nonumber
 		 \begin{split}
 		 &\quad G(\{v_{m,j}(a_{1},\dots,a_{s})\colon (a_{1},\dots,a_{s})\in(\N_0^{L})^{s}\})
 		 \\&=G(\{u_{1,j}(a_{1},\dots,a_{s})-u_{m,j}(a_{1},\dots,a_{s})\colon (a_{1},\dots,a_{s})\in(\N_0^{L})^{s}\})=H_{1,m},\; 2\leq m\leq \ell
 		 \end{split}
 		 \end{equation}
		contains some of the groups $G_{1,j}(\p), 0\leq j\leq k, j\neq 1$.

 		\medskip
 		
 		Next we assume that $A$ is not 1-standard for $f_{1}$. In this case, we need to invoke the ``dimension increment'' argument to convert $A$ to be 1-standard for $f_{1}$.
 		We may assume without loss of generality that $p_{k}$ has the highest degree.
		 Since $A$ is semi-standard for $f_{1}$, by \cite[Proposition~6.3]{DKS},  there exists a PET-tuple $A'=(2L,0,\ell,\p',\g)$ which is non-degenerate and 1-standard  for $f_{1}$ such that $S(A,2\tau)\leq S(A',\tau)$ for all $\tau>0$. Moreover, $\p'$ is obtained by selecting some polynomials from the family
 		 $$\q:=(p_{1}(n)-p_{k}(n'),\dots,p_{k}(n)-p_{k}(n'),p_{1}(n')-p_{k}(n'),\dots,p_{k-1}(n')-p_{k}(n'))$$ 
 		 with $2L$-dimensional variables $(n,n'),$ where $p_{1}(n)-p_{k}(n')$ is selected in $\p'$ and is associated to $f_{1}$. It is not hard to compute that $G_{1,j}(\q)=G_{1,j}(\p)$ for $0\leq j\leq k, j\neq 1$. Moreover, for $1\leq j\leq k-1$, $G_{1,k+j}(\q)=G_{1,0}(\p)+G_{j,0}(\p)\supseteq G_{1,0}(\p)$ if $d_{1,0}=d_{j,0}$, $G_{1,k+j}(\q)=G_{1,0}(\p)=G_{1,j}(\q)$ if $d_{1,0}>d_{j,0}$, and $G_{1,k+j}(\q)=G_{j,0}(\p)=G_{1,j}(\q)$ if $d_{1,0}<d_{j,0}$. 
 		 In other words, each $G_{1,j}(\q)$ and thus each $G_{1,j}(\p')$ contains some $G_{1,j'}(\p).$
 		 Applying the previous conclusion to $A'$, we are done.

 		 Finally the fact that each
 		$\v_{i,m}(a_{1},\dots,a_{s})$ is a polynomial function in terms of the coefficients of $p_{i}, 1\leq i\leq k$ whose degree depends only on $d,k,K,L$ follows easily from the polynomial nature of the vdC-operations. 
 \end{proof}
 We now have all the ingredients needed to prove Theorem \ref{5555}. In fact,
Theorem \ref{5555} has a proof similar to that of \cite[Theorem~5.1]{DKS}.
\begin{proof}[Proof of Theorem \ref{5555}]
By Propositions \ref{pet} and \ref{pet3}, and the definition of Host-Kra characteristic factors, the left hand side of  \eqref{123} is 0 if for some $1\leq i\leq k$, $f_{i}$ is orthogonal to
$Z_{\{H_{i,m}\}_{1\leq m\leq t_{i}}^{\times D_{i}}}(\X)$ for some $t_{i},D_{i}\in\N$, where $H_{i,m}$ is defined as in Proposition \ref{pet}.  By Proposition \ref{pet}, $H_{i,m}$ is contained in one of $G_{i,j}(\p), 0\leq j\leq k, j\neq i$.
Using Proposition \ref{prop:basic_properties} (v), if some $f_{i},$  $1\leq i\leq k$, is orthogonal to
$Z_{\{G_{i,j}(\p)\}_{0\leq j\leq k, j\neq i}^{\times D_{i}t_{i}}}(\X)$, then it is also orthogonal to
$Z_{\{H_{i,m}\}_{1\leq m\leq t_{i}}^{\times D_{i}}}(\X)$, and thus the left hand side of  (\ref{123}) is 0.

The ``in particular'' part follows from Corollary \ref{st1}.
\end{proof}

 \begin{remark}\label{R:bounds}
We remark that the number $D$ derived in Theorem \ref{5555} is not optimal.

To see this, recall that this number indicates the step of the nilsequence in the splitting results. For multicorrelation sequences with general polynomial iterates, this $D$ can be taken to be equal to the number of vdC-operations we have to perform in order for all the iterates to become constant (e.g., \cite{fra1} via \cite{FHK}, and \cite{K1} via \cite{fra1}). At this point, a word of caution is necessary for the approach of this paper. Specifically, while the number $D$ in Theorem \ref{5555} can still be chosen to be the number of transformations in the case of linear iterates (given that there is no dependence on $h$--the variable arising from the vdC-operations), in the general case the picture is quite different. By carefully tracking the constants that appear in Propositions \ref{pet3} and \ref{pet}, $D$ can be chosen to be the maximum of 
$t_{i}s_{i}^{[t_{i}(s'_{i}+1)^{s_{i}L}]+1}, 1\leq i\leq k$, where $s'_{i}$ is the degree of $\p$, $t_{i}$  is the number of terms remaining when ${\bf{p}}$ is converted to a linear family which is 1-standard?  for $f_{i}$ for the first time, and $s_{i}$ is $t_{i}$ plus the number of vdC-operations needed to convert $\p$ in such a way.  (The details are left to the interested readers.)
\end{remark}

 \section{Proof of Main results}\label{s6}
Using  Theorem \ref{5555}, we prove in this section Theorems~\ref{mainthm}, \ref{j2} and \ref{j3}.

\begin{proof}[Proof of Theorem \ref{mainthm}] We follow and adapt the proof strategy in \cite[Section~3]{leibman1}.
To avoid confusion we use $\nnorm{\cdot}_{\ast}$ to denote the Host-Kra seminorm on the $\Z^{d}$-system $(X,\mathcal{B},\mu,(T_{n})_{n\in\Z^{d}})$ and $\nnorm{\cdot}'_{\ast}$ to denote the Host-Kra seminorm on the $\Z^{d}$-system $(X^{2},\mathcal{B}^{2},\mu^{2},$ $(S_{n})_{n\in\Z^{d}})$, where $S_{n}=T_{n}\times T_{n}$.
Let $(I_N)$ be a F{\o}lner sequence in $\Z^L$. Then, by Theorem \ref{5555}, we have  
\begin{equation}\label{bound1}
\begin{split}
 		&\quad \lim_{N \to \infty} \frac{1}{|I_N|} \sum_{n \in I_N} \left| \int_X f_0\cdot T_{p_1(n)}f_1 \cdot \dotso \cdot T_{p_k(n)} f_k \ d\mu \right|^2
 		\\&=
\lim_{N \to \infty} \frac{1}{|I_N|} \sum_{n \in I_N} \int_{X^2} f_0 \otimes \bar{f_0} \cdot \prod_{i=1}^k S_{p_i(n)} (f_i \otimes \bar{f_i}) \ d\mu\times \mu 
\\&\leq \nnorm{f_i \otimes \bar{f_i}}'_{\{G'_{i,j}(\p)\}^{\times D'}_{0\leq j\leq k, j\neq i}}
\end{split}
\end{equation}
for all $1 \leq i \leq k$ and some $D'\in\N$,
where $G'_{i,j}(\p)$ is the group action generated by $S_{n}$ for all $n\in G_{i,j}(\p)$.\footnote{Strictly speaking, $G_{i,j}(\p)$ and $G'_{i,j}(\p)$ are the same subgroup of $\Z^{d}$. We distinguish these two notions to indicate that $G_{i,j}(\p)$ and $G'_{i,j}(\p)$ are attached to the distinct group actions $(T_{n})_{n\in G_{i,j}(\p)}$ and $(S_{n})_{n\in G'_{i,j}(\p)}$.}
Using Lemma \ref{lem1}, the right hand side of (\ref{bound1}) is bounded by $\nnorm{f_i}^2_{\{G_{i,j}(\p)\}^{\times D'}_{0\leq j\leq k, j\neq i},\Z^{d}}$, which is equal to $\nnorm{f_i}^2_{(\Z^{d})^{\times (D+1)}}$ with $D=kD'$ by our ergodicity assumptions and Corollary~\ref{st1}.

Therefore, for $1 \leq i\leq k$ we have
\begin{equation}\label{bound2}
\lim_{N \to \infty} \frac{1}{|I_N|} \sum_{n \in I_N} \int_{X^2} f_0 \otimes \bar{f_0} \cdot \prod_{i=1}^k S_{p_i(n)} (f_i \otimes \bar{f_i}) \ d\mu \times\mu \leq C\nnorm{f_i}_{(\Z^d)^{\times (D+1)}}^2.
\end{equation}
The bound in \eqref{bound2} and Theorem \ref{st2} imply that the sequence 
\begin{equation}\label{multinull}
a(n)-\int_{X} f_0 \cdot T_{p_1(n)} \mathbb{E}(f_1 \mid Z_{(\Z^{d})^{\times (D+1)}}(\X)) \cdot \dotso \cdot T_{p_k(n)}\mathbb{E}(f_k \mid Z_{(\Z^{d})^{\times (D+1)}}(\X)) \ d\mu
\end{equation}
is a null-sequence.

Let $\varepsilon>0$. The factor $Z_{(\Z^{d})^{\times (D+1)}}(\X),$  via Theorem~\ref{st2}, is an inverse limit of $D$-step nilsystems. Thus, there exists a factor of $Z_{(\Z^{d})^{\times (D+1)}}(\X)$ with the structure of a $D$-step nilsystem $(\tilde{X},\mathcal{B}(\tilde{X}),\mu_{\tilde{X}},$ $T_1,\dots,T_d)$, on which each $T_i$ acts as a niltranslation by an element $a_i \in \tilde{X}$, such that for $\tilde{f_i}=\mathbb{E}(f_i \mid \tilde{X})$, and $\vec{a}\coloneqq (a_1,\dots,a_d),$ we have
\[ \left|\int_{X} f_0 \cdot \prod_{i=1}^k T_{p_i(n)} \mathbb{E}(f_i \mid Z_{(\Z^{d})^{\times (D+1)}}(\X)) \ d\mu-\int_{\tilde{X}} \tilde{f_0} \cdot \prod_{i=1}^k \vec{a}_{p_i(n)}\tilde{f_i} \ d\mu_{\tilde{X}} \right|<\varepsilon\]
for all $n \in \Z^L$, where, if $p_{i}=(p_{i,1},\ldots,p_{i,d}),$ then $\vec{a}_{p_i(n)}$ denotes the niltranslation by the element $(a_1^{p_{1,1}(n)},\dots,a_d^{p_{1,d}(n)})$. Therefore, there exists a nullsequence $\lambda$ such that 
\begin{equation}\label{ellinfty1}
\left|\left|a(n)-\left(\int_{\tilde{X}} \tilde{f_0} \cdot \vec{a}_{p_1(n)}\tilde{f_1} \cdot \dotso \cdot \vec{a}_{p_k(n)} \tilde{f_k} \ d\mu_{\tilde{X}}+\lambda(n)\right)\right|\right|_{\ell^{\infty}(\Z^L)}<\varepsilon.
\end{equation}
 A standard approximation argument allows us to assume without loss of generality that $\tilde{f_1},\dots,\tilde{f_k} \in C(\tilde{X})$ in \eqref{ellinfty1}. Applying \cite[Theorem~2.5]{leibman1} to the nilmanifold $\tilde{X}^k$, the diagonal subnilmanifold $\{(x,\dots,x) : x \in \tilde{X}\}$, the polynomial sequence $(\vec{a}_{p_1(n)},\dots,\vec{a}_{p_k(n)})$ and the function $f(x_1,\dots,x_k)=\tilde{f_1}(x_1)\cdot\dotso \cdot \tilde{f_k}(x_k) \in C(\tilde{X}^k)$, we obtain that the sequence
\[\psi(n):=\int_{\tilde{X}} \tilde{f_0} \cdot \vec{a}_{p_1(n)}\tilde{f_1} \cdot \dotso \cdot \vec{a}_{p_k(n)} \tilde{f_k} \ d\mu_{\tilde{X}} \]
is a sum of a $D$-step nilsequence and a nullsequence.

Therefore, for each $\varepsilon>0$ we can find a $D$-step nilsequence $\psi$, a nullsequence $\lambda$ and a bounded sequence $\delta$ with $\Vert\delta\Vert_{\ell^{\infty}(\Z^L)} \leq \varepsilon$ such that
\begin{equation}\label{decomp1} a(n)=\psi(n)+\lambda(n)+\delta(n).\end{equation}
For each $l \in \N$, consider the decomposition $a=\psi_l+\lambda_l+\delta_l$, where $\Vert\delta_l\Vert_{\ell^{\infty}(\Z^L)}<\frac{1}{l}$. For $r \neq l$, we have
\begin{equation}\label{nilsequence1} \vert\psi_l(n)-\psi_r(n)\vert=\vert(\lambda_l(n)-\lambda_r(n))+(\delta_l(n)-\delta_r(n))\vert.
\end{equation}
Now, $\lim_{|I_N| \to \infty} \frac{1}{|I_N|}\sum_{n \in I_N} \vert\lambda_l(n)-\lambda_r(n)\vert=0$ and $\sup_{n \in \Z^L} \vert\delta_r(n)-\delta_l(n)\vert \leq \frac{1}{l}+\frac{1}{r}$. Therefore,
\begin{equation}\label{nilsequence2}
    \vert\psi_l(n)-\psi_r(n)\vert\leq \frac{1}{l}+\frac{1}{r}
\end{equation}
for all $n \in \Z^L$ except potentially a subset $A \subseteq \Z^L,$ with its characteristic function, $\mathbbm{1}_A(n),$ being a nullsequence. For each $l, r \in \N$, the sequence $\psi_l(n)-\psi_r(n)$ is a nilsequence, so it follows that inequality \eqref{nilsequence2} must, in fact, hold for all $n \in \Z^L$. Hence, the sequence $(\psi_l)_{l \in \N}$ is a Cauchy sequence in $\ell^{\infty}(\Z^L)$ that consists of $D$-step nilsequences, and since we already showed that $(\delta_r)_{r \in \N}$ is a Cauchy sequence in $\ell^{\infty}(\Z^L)$ converging to a nullsequence, from where the conclusion follows.
\end{proof}

\begin{remark}\label{R:step}
It is worth noting that if the polynomials $p_{1},\dots,p_{k}$ are linear, then there is an easier proof of Theorem \ref{mainthm}, where one has $D=k$. The reason is that, instead of Theorem \ref{5555}, one can use \cite[Proposition~1]{host1} or \cite[Proposition~6.1]{DKS} to improve the right hand side of \eqref{bound2} to $\nnorm{f_{i}}^{2}_{(\Z^{d})^{\times (k+1)}}$. 
\end{remark}

Next, we provide the proof of Theorem \ref{j3} (the arguments are similar to those in the proof of \cite[Theorem~1.3, Pages 26--27]{DKS}).

\begin{proof}[Proof of Theorem \ref{j3}]
Using Property (i) and Theorem \ref{5555}, we have that there exists $D\in\N$ such that the left hand side of \eqref{444} is 0 if one of $f_i$'s satisfies $\nnorm{f_{i}}_{(\Z^{d})^{\times D}}=0.$ Thus, it suffices to show \eqref{444} under the assumption that all $f_{i}$'s are measurable with respect to $Z_{(\Z^{d})^{\times D}}$. Using Theorem~\ref{st2} and an approximation argument, we may assume without loss of generality that $\X$ is a nilsystem. As a consequence of \cite[Theorem~B]{L05} (see also \cite[Theorem~2.9]{DKS}),  Property (ii) is equivalent to (\ref{444}) on a nilsystem, which finishes the proof. 
 \end{proof}
 
 Finally, we prove Theorem \ref{j2}. We start with proving that (C1) implies (C2). In fact, we show the following more general result:
 
 \begin{proposition}\label{g22}
 	Let $d,k,L\in\mathbb{N},$ $q_{1},\dots, q_{k}\colon\mathbb{Z}^{L}\to \mathbb{Z}^{d}$  be polynomials, and $\X=(X,\mathcal{B},\mu,$ $ (T_{g})_{g\in\mathbb{Z}^{d}})$ be a $\mathbb{Z}^{d}$-system. Suppose that $(T_{q_{1}(n)},\dots,T_{q_{k}(n)})_{n\in\Z^{L}}$ is jointly ergodic for $\mu$. Then
 	for all $1\leq i,j\leq k, i\neq j$, we have that $(T_{q_{i}(n)-q_{j}(n)})_{n\in\Z^{L}}$ is ergodic for $\mu$.
 	
 	Furthermore, if there exist polynomials $p_{1},\dots, p_{k}\colon\mathbb{Z}^{L}\to \mathbb{Z}$
 	and $v_{1},\dots,v_{k}\in\Z^{d}$ such that $q_{i}(n)=p_{i}(n)v_{i}$ for all $1\leq i\leq k$, then
 	$(T_{p_{1}(n)v_{1}}\times\dots\times T_{p_{k}(n)v_{k}})_{n\in\Z^{L}}$ is ergodic for $\mu^{\otimes k}$.
 \end{proposition}	
 \begin{proof}
 The sequence $(T_{q_{i}(n)-q_{j}(n)})_{n\in\Z^{L}}$ is ergodic for $\mu$ for all $1\leq i\neq j\leq k$ by an argument similar to the one given in the proof of \cite[Proposition~5.3]{DKS}, so, we choose to omit the details.

 	We now assume that  $q_{i}(n)=p_{i}(n)v_{i}$ for all $1\leq i\leq k$ and show that 	$(T_{p_{1}(n)v_{1}}\times\dots\times T_{p_{k}(n)v_{k}})_{n\in\Z^{L}}$ is ergodic for $\mu^{\otimes k}$.
 	It suffices to show that for all $f_{i}\in L^{\infty}(\mu),1\leq i\leq k$ with $\prod_{i=1}^{k} \int_{X}f_{i}\,d\mu=0$,  we have that
 	\begin{eqnarray}\label{73}
 \F\varlimsup_{N\to \infty}
  \bigg\Vert\mathbb{E}_{n\in I_{N}}\bigotimes_{i=1}^{k}T_{p_{i}(n)v_{i}}f_{i}\bigg\Vert_{L^{2}(\mu^{\otimes k})}=0.
 	\end{eqnarray}
 	
\noindent 	\textbf{Claim.}
 	If $\mathbb{E}(f_{i}\vert Z_{\Z^{d},\Z^{d}}(\X))=0$ for some $1\leq i\leq k$, then (\ref{73}) holds.
 	
 	We may assume without loss of generality that $\deg(p_{1})\geq\deg(p_{2})\geq\dots\geq\deg(p_{k})$.
 	Suppose that we have shown that for some $1\leq k_{0}\leq k$,  (\ref{73}) holds if $\mathbb{E}(f_{i}\vert Z_{\Z^{d},\Z^{d}}(\X))=0$ for some $1\leq i\leq k_{0}-1$, where the case $k_{0}=1$ is understood to be always true. It suffices to show  that  (\ref{73}) holds if $\mathbb{E}(f_{k_{0}}\vert Z_{\Z^{d},\Z^{d}}(\X))=0$.
 	
 	By the induction hypothesis, we may assume without loss of generality that $f_{i}$ is $Z_{\Z^{d},\Z^{d}}(\X)$-measurable for all $1\leq i\leq k_{0}-1$. 
 	By \cite[Lemma 2.7]{DKS}, we can approximate each $f_{i}$ in $L^{2}(\mu)$ by an eigenfunction of $\X$. By multi-linearity, we may assume without loss of generality that each $f_{i},1\leq i\leq k_{0}-1$ is a non-constant eigenfunction of $\X$ given by $T_{n}f_{i}=\exp(\lambda_{i}(n))f_{i}$ for all $n\in \mathbb{Z}^{d}$ and some group homomorphism $\lambda_{i}\colon\mathbb{Z}^{d}\to\R$, and that $f_{i}(x)\neq 0$ $\mu$-a.e $x\in X$. Then the left hand side of (\ref{73}) is equal to
 	\begin{eqnarray}\label{735}
 	\F\varlimsup_{N\to \infty}\bigg\Vert\mathbb{E}_{n\in I_{N}}\exp(P(n))\bigotimes_{i=1}^{k_{0}-1}f_{i}\bigotimes_{i=k_{0}}^{k}T_{p_{i}(n)v_{i}}f_{i}\bigg\Vert_{L^{2}(\mu^{\otimes k})},
 	\end{eqnarray}
 	where $P(n)\coloneqq \sum_{i=1}^{k_{0}-1}\lambda_{i}(v_{i})p_{i}(n).$
 		Denote 
 	$P(n)=\sum_{j=0}^{\deg(p_{1})}Q_{j}(n)$, where $Q_{j}$ is a homogeneous polynomial of degree $j$  for all $0\leq j\leq \deg(p_{1})$. Then
 	\begin{equation}\label{737}
 	\begin{split} \Delta^{K}P(n,h_{1},\dots,h_{K})=\sum_{j=K}^{\deg(p_{1})}\Delta^{K}Q_{j}(n,h_{1},\dots,h_{K}).
 	\end{split}
 	\end{equation} 
 	
 	We first consider the case when
 	$Q_{j}(n)\notin \mathbb{Q}[n]+\R$ for some $\deg(p_{k_{0}})+1\leq j\leq \deg(p_{1})$. In this case,
 	let $K=\deg(p_{k_{0}})$ in (\ref{737}).
 	Since $\Delta^{K}p_{i}(n,h_{1},\dots,h_{K})$ is constant in $n$ for all $k_{0}\leq i\leq k$, by Lemma \ref{Kron}, to show that (\ref{735}) is 0, it suffices to show that 
 	\begin{equation}\label{736}
 	\begin{split} &  	\E^{\square}_{h_{1},\dots,h_{K}\in\Z^{L}}\F\varlimsup_{N\to \infty}\vert\mathbb{E}_{n\in I_{N}}\exp(\Delta^{K}P(n,h_{1},\dots,h_{K}))\vert=0
 	\end{split}
 	\end{equation}
 	(see Definition \ref{d1} for the definition of the polynomial $\Delta^{K}P$).

 	As $Q_{j}(n)\notin \mathbb{Q}[n]+\R$ for some $\deg(p_{k_{0}})+1\leq j\leq \deg(p_{1})$, Lemma~\ref{alg3} implies that $\Delta^{K}Q_{j}(\cdot,h_{1},\dots,h_{K})\notin \mathbb{Q}[n]+\R$ for a set of $(h_{1},\dots,h_{K})$ of density 1.
 	By Weyl's criterion and (\ref{737}), we have that (\ref{736}) holds and thus (\ref{73}) holds. 
 	
 	We now consider the case when $Q_{j}(n)\in \mathbb{Q}[n]+\R$ for all $K+1\leq j\leq \deg(p_{1})$.
 	Let $P'(n)=\sum_{j=0}^{K}Q_{j}(n)$.
 	It is not hard to see that there exists $Q\in\N$ such that for all $r\in\{0,\dots,Q-1\}^{L}$  and $n\in \mathbb{Z}^{L}$, we have that 
 	$$P(Qn+r)-P'(Qn+r)=P(r)-P'(r).$$
 	By (\ref{735}), to show (\ref{73}), it suffices to show that for all $r\in\{0,\dots,Q-1\}^{L}$, we have that 
 	\begin{eqnarray}\label{74}
 	\F\varlimsup_{N\to \infty}\Biggl\Vert\mathbb{E}_{n\in I_{N}}\exp(P'(Qn+r))\bigotimes_{i=k_{0}}^{k}T_{p_{i}(Qn+r)v_{i}}f_{i}\Biggr\Vert_{L^{2}(\mu^{\otimes t})}=0,
 	\end{eqnarray}
 	where $t=k-k_{0}+1$.
 	Fix $r\in\{0,\dots,Q-1\}^{L}$ and set $R(n)=P'(Qn+r)$. Let $p\colon \Z^{L}\to\Z^{dt}$ be the polynomial given by
 	$$p(n)=(p_{i}(Qn+r)v_{i})_{k_{0}\leq i\leq k}.$$
 	Let $(X^{t},\mathcal{B}^{t},\mu^{t},(S_{g})_{g\in \Z^{dt}})$ be the $\Z^{dt}$-system such that 
 	$$S_{(u_{i})_{k_{0}\leq i\leq k}}\coloneqq \prod_{i=k_{0}}^{k} T_{u_{i}}$$
 	for all $u_{i}\in\Z^{d}, k_{0}\leq i\leq k$, and denote $f\coloneqq \bigotimes_{i=k_{0}}^{k}f_{i}$. We may then rewrite the left hand side of (\ref{74}) as 
 	\begin{equation}\label{75}
 	\begin{split} 
 	\F\varlimsup_{N\to \infty}\Vert\mathbb{E}_{n\in I_{N}}\exp(R(n))S_{p(n)}f\Vert_{L^{2}(\mu^{\otimes t})}.
 	\end{split}
 	\end{equation}
 	For $K=\deg(p_{k_{0}})-1,$
 	using Lemma \ref{Kron}, to show that (\ref{75}) is zero, it suffices to show
 	\begin{equation}\label{76}
 	\begin{split} 	\E^{\square}_{{\bf {h}}=(h_{1},\dots,h_{K})\in(\Z^{L})^K}\F\varlimsup_{N\to \infty}\Vert\mathbb{E}_{n\in I_{N}}\exp(\Delta^{K}R(n,{\bf {h}}))S_{\Delta^{K}p(n,{\bf{h}})}f\Vert_{L^{2}(\mu^{\otimes t})}=0.
 	\end{split}
 	\end{equation}
 	
 	By assumption,  $\Delta^{K}R(n,h_{1},\dots,h_{K})$ is of degree  1 in the variable $n$.
 	Since $\deg(p)=\deg(p_{k_{0}})\geq\deg(p_{i})$ for all $i\geq k_{0}$,
 	$\Delta^{K} p(n,h_{1},\dots,h_{K})$ is also of degree 1 in the variable $n$.
 	We may thus assume that 
 	$$\Delta^{K} p(n,h_{1},\dots,h_{K})=((c_{i}(h_{1},\dots, h_{K})\cdot n+c'_{i}(h_{1},\dots, h_{K}))v_{i})_{k_{0}\leq i\leq k},$$
 	for some polynomials $c_{k_{0}},\dots,c_{k}\colon\Z^{LK}\to\Z^{L}$ and $c'_{k_{0}},\dots,c'_{k}\colon\Z^{LK}\to\Z$.
 	Write $\bold{c}(h_{1},\dots, h_{K})\coloneqq (c_{i}(h_{1},\dots, h_{K})v_{i})_{k_{0}\leq i\leq k}$ (which is viewed as a $t$-tuple of $L$-tuple of vectors in $\Z^{d}$).
 	If we write
 	$$c_{i}(h_{1},\dots, h_{K})=(c_{i,1}(h_{1},\dots, h_{K}),\dots,c_{i,L}(h_{1},\dots, h_{K}))$$
 	for some $c_{i,j}(h_{1},\dots, h_{K})\in\Z$, then, by definition, $G(\c(h_{1},\dots,c_{K}))$ is the subgroup of $\Z^{dt}$ generated by the elements
 	$$(c_{k_{0},j}(h_{1},\dots, h_{K})v_{k_{0}},\dots,c_{k,j}(h_{1},\dots, h_{K})v_{k}), 1\leq j\leq L.$$ 
 	By Lemma \ref{Kron2},  the left hand side of (\ref{76}) is bounded by a constant multiple of
 	\begin{equation}\label{77}
 	\begin{split}
 	  (\E^{\square}_{h_{1},\dots,h_{K}\in\Z^{L}}\nnorm{f_{k_0}}^{4}_{G(c_{k_{0}}(h_{1},\dots,c_{K})v_{k_{0}})^{\times 2}})^{1/4},
 	\end{split}
 	\end{equation}
 	where $G(c_{k_{0}}(h_{1},\dots,c_{K})v_{k_{0}})$ is the subgroup of $\Z^{d}$ generated by the elements
 	$$c_{k_{0},1}(h_{1},\dots, h_{K})v_{k_{0}}, \dots,c_{k_{0},L}(h_{1},\dots, h_{K})v_{k_{0}},$$
 	i.e., the entries of $c_{k_{0}}(h_{1},\dots,h_{K})v_{k_{0}}$. 
 	For any $u_{k_{0}}\in G(c_{k_{0}}(h_{1},\dots,h_{K})v_{k_{0}})$, note that $u_{k_{0}}$ is a rational multiple of $v_{k_{0}}$. So, if $c_{k_{0}}(h_{1},\dots,h_{K})\neq \bold{0}$, then 
 	$G(c_{k_{0}}(h_{1},\dots,h_{K})v_{k_{0}})=G(v_{k_{0}})$.
 	
 	Since $(T_{p_{i}(n)v_{i}})_{n\in\Z^{L}}$ is ergodic for $\mu$, we have that $T_{v_{i}}$ is ergodic for $\mu$. 
 	As $\mathbb{E}(f_{k_{0}}\vert 	Z_{\Z^{d},\Z^{d}}(\X))=0$, by \cite[Lemma~2.4]{DKS}, we have
 	that $$\nnorm{f_{k_{0}}}_{G(c_{k_{0}}(h_{1},\dots,c_{K})v_{k_{0}})^{\times 2}}=\nnorm{f_{k_{0}}}_{v_{k_{0}}^{\times 2}}=\nnorm{f_{k_{0}}}_{(\Z^{d})^{\times 2}}=0$$
 	whenever 
 	$c_{k_{0}}(h_{1},\dots,h_{K})\neq\bold{0}$. Since $K=\deg(p_{k_{0}})-1=\deg(p_{k_{0}}(Q\cdot +r))-1$, it is easy to see that $c_{k_{0}}\not \equiv\bold{0}$. By \cite[Lemma 2.12]{DKS}, the set of such  $(h_{1},\dots,h_{K})$ is of density 1. So,
 	averaging over all $h_{1},\dots,h_{K}\in\Z^{L}$, we have that (\ref{77}) is 0.
 	This finishes the proof of the claim.
 	
 	\medskip
 	
 	Using the  claim, it suffices to prove (\ref{73}) under the assumption that all $f_{i}$'s are measurable with respect to $Z_{(\Z^{d})^{\times 2}}(\X)$.
 	By Lemma 2.7 of \cite{DKS}, we can approximate each $f_{i}$ in $L^{2}(\mu)$ by an eigenfunction of $\X$. By multi-linearity, we may assume without loss of generality that each $f_{i}$ is a non-constant eigenfunction of $\X$ given by $T_{n}f_{i}=\exp(\lambda_{i}(n))f_{i}$ for all $n\in \mathbb{Z}^{d}$ for some group homomorphism $\lambda_{i}\colon\mathbb{Z}^{d}\to\R,$ and that $f_{i}(x)\neq 0$ $\mu$-a.e $x\in X$. Then, since $(T_{p_{i}(n)v_{i}}\colon 1\leq i\leq k)_{n\in\Z^{L}}$ is jointly ergodic for $\mu$, for any F\o lner sequence $(I_{N})_{N\in\N}$ of $\Z^{d}$,
 	\[ 0= \prod_{i=1}^{k}\int_{X} f_{i} d\mu=  \lim_{N\to\infty}\mathbb{E}_{n\in I_{N}}\prod_{i=1}^{k}T_{p_{i}(n)v_{i}}f_{i}=\left(\lim_{N\to\infty}\mathbb{E}_{n\in I_{N}}\prod_{i=1}^{k}\exp(\lambda_{i}(p_{i}(n)v_{i}))\right)\prod_{i=1}^{k}f_{i}.\]
 	This implies that $\lim_{N\to\infty}\mathbb{E}_{n\in I_{N}}\prod_{i=1}^{k}\exp(\lambda_{i}(p_{i}(n)v_{i}))=0$. So,
 	\[\lim_{N\to\infty}\mathbb{E}_{n\in I_{N}}\bigotimes_{i=1}^{k} T_{p_{i}(n)v_{i}}f_{i}=\left(\lim_{N\to\infty}\mathbb{E}_{n\in I_{N}}\prod_{i=1}^{k}\exp(\lambda_{i}(p_{i}(n)v_{i}))\right)\bigotimes_{i=1}^{k}f_{i}.\]
 	This finishes the proof.	
 \end{proof}	
 
 We are now ready to complete the proof of Theorem \ref{j2}.

 \begin{proof}[Proof of Theorem \ref{j2}]
 	Using Proposition \ref{g22}, we have that (C1) implies (C2). It is obvious that (C2) implies (C2'). So, it suffices to show that (C2') implies (C1).
  	
 	It is not hard to see that we may assume without loss of generality that $p_{i}(0)=0$ for all $1\leq i\leq k$.	
 	  By Theorem \ref{j3}, to show that $(T_{p_{i}(n)v_{i}}\colon 1\leq i\leq k)_{n\in\Z^{L}}$ is jointly ergodic for $\mu$, it suffices to show that $G_{i,j}(\p)$ is ergodic for $\mu$ for all $0\leq i,j\leq k, i\neq j$.	Fix any such pair $(i,j)$. We may assume without loss of generality that $i\neq 0$.
 	If $j=0$, then by (ii), $(T_{p_{i}(n)v_{i}})_{n\in\Z^{L}}$ is ergodic for $\mu$. So $T_{v_{i}}$ is ergodic for $\mu$, and thus $G_{i,0}(\p)=G(v_{i})$ is ergodic for $\mu$.
 	So, we may now assume that $j\neq 0$. 
 	
 	Assume first that $\deg(p_{i})=\deg(p_{j})$.
 	 By assumption, either  $v_{i}$ and $v_{j}$ are linearly dependent, or $p_{i}(n)$ and $p_{j}(n)$ are linearly dependent. 
 	 
 	If $v_{i}$ and $v_{j}$ are linearly dependent over $\Z$, then we may assume without loss of generality that $v_{i}=av$ and $v_{j}=bv$ for some $a,b\in\mathbb{Q}$ and $v\in\Z^{d}$. By (i),  $(T_{(ap_{i}(n)-bp_{j}(n))v})_{n\in\Z^{L}}$ is ergodic for $\mu$, which implies that $G(v)$ is ergodic for $\mu$.
 	On the other hand, $G_{i,j}(\p)$ is a group generated by some elements which are linear combinations of $v_{i}$ and $v_{j}$, which are thus multiples of $v$. Since $\p$ is non-degenerate,  $G_{i,j}(\p)$ is not the trivial group. It follows that $G_{i,j}(\p)=G(v),$ so the group $G_{i,j}(\p)$ is ergodic for $\mu$.
 	
 	If $p_{i}(n)$ and $p_{j}(n)$ are linearly dependent over $\Z$, then we may assume without loss of generality that $p_{i}(n)=ap(n)$ and $p_{j}(n)=bp(n)$ for some $a,b\in\mathbb{Q}$ and polynomial $p$. By (i),  $(T_{(p(n)(av_{i}-bv_{j})})_{n\in\Z^{L}}$ is ergodic for $\mu$, which implies that $G(av_{i}-bv_{j})$ is ergodic for $\mu$. On the other hand, $G_{i,j}(\p)$ is a group generated by some elements which are multiples of $av_{i}-bv_{j}$. Since $\p$ is non-degenerate,  $G_{i,j}(\p)$ is not the trivial group. It follows that $G_{i,j}(\p)=G(av_{i}-bv_{j}),$ so the group $G_{i,j}(\p)$ is ergodic for $\mu$.
 	
 	Finally, we consider the case when $\deg(p_{i})\neq \deg(p_{j})$. We may assume without loss of generality that $\deg(p_{i})>\deg(p_{j})$.  In this case, $G_{i,j}(\p)=G_{i,0}(\p)$, which we have shown is ergodic for $\mu$.
 	 \end{proof}
\section{Potential future directions}\label{S:Future_direction}

We close this article with two potential future directions regarding the splitting of multicorrelation sequences. The first one is for integer polynomial iterates under no assumptions on the transformations other than commutativity (see Theorem~\ref{T:no_com} for a special case of two terms). 

The second one pertains to potential results analogous to Theorem~\ref{mainthm} for iterates of the form $[p_i(n)],$ $1\leq i\leq k,$ where $p_i=(p_{i,1},\ldots,p_{i,d}):\mathbb{Z}^L\to \mathbb{R}^d$ are vectors of real valued polynomials.\footnote{ Here, for $x=(x_1,\ldots,x_L)\in \mathbb{R}^L,$ we write $[x]\coloneqq ([x_1],\ldots,[x_L]),$ where $[\cdot]$ is the floor function. In fact, one can consider any combination of rounding functions, i.e., floor, ceiling, or closest integer.} 

\subsection{The two-term case with no ergodicity assumptions}

Given the results in the appendix of \cite{CFH}, we are able to obtain the following splitting result for two commuting transformations without any ergodicity assumptions:
\begin{theorem}\label{T:no_com} Let $(X,\mathcal{B},\mu,T,S)$ be a measure preserving system with $TS=ST$. Let $f_0,f_1,f_2 \in L^{\infty}(\mu)$ and $p \in \Z[n]$ with degree $K\geq 2$. Then, the multicorrelation sequence
\[ a(n)\coloneqq \int_X f_0 \cdot T^nf_1 \cdot S^{p(n)}f_2 \ d\mu \]
can be decomposed as a sum of a uniform limit of $K$-step nilsequences plus a nullsequence.
\end{theorem}

\begin{proof} Setting $F_i=f_i \otimes \bar{f}_i,$ $i=0,1,2,$ and $\tilde{\mu}=\mu\times\mu,$ we have that
\begin{equation}\label{ratkronecker1} \frac{1}{N}\sum_{n=1}^N |a(n)|^2=\frac{1}{N}\sum_{n=1}^N \int_{X^2} F_0\cdot (T \times T)^n F_1\cdot (S \times S)^{p(n)} F_2 \ d\tilde{\mu}. \end{equation}
Using \cite[Theorem A.3]{CFH}, we get that the rational Kronecker factor\footnote{The \emph{rational Kronecker factor} is the smallest sub-$\sigma$-algebra of $\mathcal{B}$ that makes all functions with finite orbit in $L^2(\mu)$ under the transformation $T$ measurable.} is characteristic for the averages appearing in \eqref{ratkronecker1}. Consequently, we may replace $f_1$ by $P_cf_1$ and $f_2$ by $Q_cf_2$ in $a(n)$ up to a nullsequence, where $P_c$ denotes the orthogonal projection onto the compact component of the splitting associated to $T$, and $Q_c$ the one associated to $S$. (Here we make use of the Hilbert space splitting of $L^2(\mu)$ into its compact and weakly mixing components for a given unitary operator. The seeds for these results are already present in the work of Koopman and von Neumann \cite{koopmanvonneumann}. They were later generalized by Jacobs, Glicksberg and de Leeuw. See \cite[Section~16.3]{efhn} for a more modern treatment). Thus, the sequence
\[ a(n)-\int_X f_0 \cdot T^n P_cf_1 \cdot S^{p(n)} Q_c f_2 \ d\mu\]
is a nullsequence. Let $\varepsilon>0$ and choose $h_1,\dots,h_k,g_1,\dots,g_k \in L^2(\mu)$ such that $Th_i=\lambda_ih_i$ and $Sg_i=\rho_ig_i$ (for some $\lambda_1,\dots,\lambda_k,\rho_1,\dots,\rho_k \in \C$ of absolute value one) as well as $a_1,\dots,a_k,b_1,\dots,b_k \in \C$ such that
\[ \left|\int_X f_0 \cdot T^n P_cf_1 \cdot S^{p(n)} Q_c f_2 \ d\mu-\int_X f_0 \cdot T^n \sum_{i=1}^k a_ih_i \cdot S^{p(n)} \sum_{j=1}^k b_jg_j \ d\mu\right|<\varepsilon.\]
Observe that 
\begin{eqnarray*} \int_X f_0 \cdot T^n \sum_{i=1}^k a_ih_i \cdot S^{p(n)} \sum_{j=1}^k b_ig_i \ d\mu & = & \int_X f_0 \cdot \sum_{i=1}^k a_i\lambda_i^nh_i \cdot \sum_{j=1}^k b_j\rho_j^{p(n)}g_j \ d\mu
\\ & = &
\sum_{i,j=1}^k \left(a_ib_j \int_X f_0 \cdot h_i \cdot g_j \ d\mu \right)\lambda_i^n\rho_j^{p(n)},\end{eqnarray*}
which is a $K$-step nilsequence. Applying the same argument as in the proof Theorem~\ref{mainthm} we deduce the decomposition result. The rest of the details are omitted for the sake of brevity.
\end{proof}

It is natural to ask whether a result analogous to Theorem~\ref{T:no_com} holds for longer expressions (potentially via a generalization of the results in the appendix of \cite{CFH}), and with more general polynomial iterates, even without necessarily assuming they have distinct degrees. Thus, we state the following problem:

\begin{problem}\label{P_not_com}
Obtain decomposition results of the form ``uniform limit of nilsequences plus a nullsequence'' for multicorrelation sequences with (integer) polynomial iterates for general systems under no ergodicity assumptions on the transformations.
\end{problem}

\subsection{Integer part polynomial iterates} 
With a, by now, standard argument (introduced in \cite{bjw} and \cite{les} for a single term, extended for two terms in \cite{mw}, and further developed in \cite{K2}, \cite{K1} and \cite{klmr}) one has, for the vectors of real polynomials $p_{i}=(p_{i,1},\ldots,p_{i,d}),$ that the expression 
\begin{equation}\label{E:general_floor}
\frac{1}{|I_N|}\sum_{n\in I_N}\prod_{i=1}^k T_{[p_i(n)]}f_i=\frac{1}{|I_N|}\sum_{n\in I_N}\prod_{i=1}^k \prod_{j=1}^d T_j^{[p_{i,j}(n)]}f_i
\end{equation}
is ``close'' to  
\begin{equation}\label{E:general_flow}
\frac{1}{|I_N|}\sum_{n\in I_N}\prod_{i=1}^k \prod_{j=1}^d S_j^{p_{i,j}(n)}g_i,
\end{equation} where $S_j$'s are $\mathbb{R}$-flows on an ``extension system'' $Y$ of $X$, and the functions $g_i$'s are extensions of the $f_i$'s (see \cite{K1} for details).\footnote{ We say that a jointly measurable family $(S_t)_{t\in\R^d}$ of measure preserving transformations on a probability space, is an \emph{$\R^d$-action} (\emph{flow}), if it satisfies $S_{t+r}=S_t\circ S_r$ for all $t,r\in \R^d.$} As an application of Theorem~\ref{mainthm}, one can prove splitting theorems for $\R^d$-actions on the extension system. 

Indeed,\footnote{ We address the $L=1$ case for simplicity; following the same argument one can similarly get the corresponding result for the general case of $L$-variable polynomials by using an ordering on the parameters, e.g., $n_1>\ldots>n_L.$} consider the multicorrelation sequence
\begin{equation}\label{E:flow} \int_X f_0\cdot S_{p_1(n)}f_1 \cdot \dotso \cdot S_{p_k(n)} f_k \ d\mu,
	\end{equation} 
	where  $S$ is a measure preserving $\R^d$-action on the probability space $(X,\mathcal{B},\mu),$ $f_0,f_1,\ldots,f_k\in L^\infty(X),$ and $p_{1},\dots,p_{k}\colon\mathbb{Z}\to\mathbb{R}^{d}$ a non-degenerate family of polynomials of degree at most $K$ with $p_{i}(n) = \sum_{h = 0}^K a_{i,h} n^h$, $a_{i,h}\in \mathbb{R}^{d}$.
 Then 
$$S_{p_{i}(n)}=S_{\sum_{h = 0}^K a_{i,h} n^h}=\prod_{h=0}^{K}(S_{a_{i,h}})^{n^{h}}.$$
Note that $S_{a_{i,h}}, 1\leq i\leq k, 1\leq h\leq K$ generate a $\Z^{kK}$-action on $(X,\mathcal{B},\mu)$. For convenience, set $p_{0}$ to be the constant zero polynomial.
 For $0\leq i, j\leq k, i\neq j$, let $D_{i,j}$ be the largest integer $h$ such that $S_{a_{i,D_{i,j}}-a_{j,D_{i,j}}}\neq \textrm{id}$. This transformation will be denoted by $R_{i,j}$.
By Theorem~\ref{mainthm}, one can show that the desired splitting result  for the sequence (\ref{E:flow}), if all the transformations $R_{i,j}, 0\leq i, j\leq k, i\neq j$
are all ergodic (as $\Z$-actions on the extension system $Y$).

Unfortunately, even though we have the previous result for flows, the error term that arises from the approximation of \eqref{E:general_floor} by \eqref{E:general_flow} prevents us from getting the conclusion of  Theorem~\ref{mainthm} for multicorrelation sequences of the form
\begin{equation*}\label{E:multi_floor}
\int_X f_0\cdot T_{[p_1(n)]}f_1 \cdot \dotso \cdot T_{[p_k(n)]} f_k \ d\mu.\footnote{ To this day, only splittings of the form nilsequence plus an error term  that is small in uniform density are known for this class of multicorrelation sequences (for this, see \cite{K1}. One is referred to \cite{klmr} for averages along primes for the error term--in this last reference only single variable real polynomials were considered. Using the multivariable approach of \cite{hostfra} instead of \cite{fra1}, one immediately gets the aforementioned result for integer part, or indeed for combinations of any other rounding functions of multivariable real polynomial iterates). }
\end{equation*}

\begin{remark}
It is important to stress that, for integer part real polynomial iterates, one does not expect to have the desired multicorrelation splitting in general. The next example shows that, even for $k=1,$ an ergodic system, and linear iterates, it can be too much to hope for: 

\emph{Indeed, following \cite[Example~7]{klmr}, let $X=\mathbb{T}\coloneqq \mathbb{R}/\mathbb{Z}$, $T(x)=x+1/\sqrt{2}$, $p(n)=\sqrt{2}n$, $f_0(x)=e(x)$ and $f_1(x)=e(-x)$, where $e(x)\coloneqq e^{2\pi ix}$. Then, we have that
\begin{eqnarray*}
\int f_0\cdot T^{[p(n)]}f_1\, d\mu
    & = & \int e(x) e\left(-x-\frac{1}{\sqrt{2}} [\sqrt{2}n]\right)\,dx 
     =  e \left( - \frac{1}{\sqrt{2}} [\sqrt{2} n] \right)
     = e\left(\frac{1}{\sqrt{2}}\{\sqrt{2}n\}\right),
\end{eqnarray*}
which cannot be written as a uniform limit of nilsequences and a nullsequence.}
\end{remark}

\begin{remark}
One may think that the fact that $\sqrt{2}$ and $\frac{1}{\sqrt{2}}$ are not linearly independent over $\Q$ is behind the impossibility of the splitting in the example above. However, a closer examination of the proof given in \cite{klmr} shows that this is not the case, and that the failure extends quite generally. 

\emph{Indeed, we can imitate the example quoted above as follows. Let $(X,\mathcal{B},\mu,T)$ be an ergodic measure preserving system with non-trivial irrational spectrum. Let $f_1 : X \to \mathbb{S}^1$ be an eigenfunction of $T$ with eigenvalue $e^{2\pi i \beta}$, with $\beta \in \R \setminus \Q$. Put $f_0=\bar{f_1}$. Then, if we consider the multicorrelation sequence 
\begin{equation*}\label{E:brackets}
        a(n)\coloneqq \int_X f_0 \cdot T^{[\alpha n]}f_1 \ d\mu,
    \end{equation*}
with the choices made above, we observe that, in fact, 
    $a(n)=e^{2\pi i [\alpha n]\beta}.$ 
The same argument as in \cite[Example~7]{klmr} shows that $a(n)$ cannot be written as a uniform limit of nilsequences plus a nullsequence.}
\end{remark}

On the other hand, if we postulate very strong assumptions on our transformations, we do have the desired decomposition results. For example (see \cite{K1}), if $T_1,\ldots,T_k$ are commuting weakly mixing transformations on $(X,\mathcal{B},\mu)$, $q_i(n)=p_i(n)e_i,$ $1\leq i\leq k,$ where $p_i\colon \mathbb{Z}\to\mathbb{R}$ are real polynomials of distinct, positive degrees, and $f_1,\dots,f_k \in L^{\infty}(\mu),$ then we have that
\[ \lim_{N-M \to \infty} \frac{1}{N-M}\sum_{n=M}^{N-1}T_{[q_1(n)]}f_1 \cdot \dotso \cdot T_{[q_k(n)]}f_k=\prod_{i=1}^k \int_X f_i \ d\mu.\]
Hence, for any $f_0 \in L^{\infty}(\mu),$ the multicorrelation sequence
    \[
         \int_X f_0 \cdot T_{[q_1(n)]}f_1 \cdot \dotso \cdot T_{[q_k(n)]}f_k \ d\mu
    \]
    can be written as a sum of a constant (i.e., a $0$-step nilsequence) and a nullsequence.

\medskip

We conclude this article with the following problem that arises naturally:

\begin{problem}\label{P_floor}
Let $d,k,K,L\in\mathbb{N},$ $p_{1},\dots,p_{k}\colon\mathbb{Z}^{L}\to\mathbb{R}^{d}$ be a non-degenerate family of polynomials of degree at most $K$,  $(X,\mathcal{B},\mu,T_1,\dots,T_d)$ a measure preserving system and $f_0,\dots,f_k \in L^{\infty}(\mu).$ 
Find conditions, on the $p_i$'s and/or the $\mathbb{Z}^d$-action $T$ that is defined by the $T_i$'s, so that the multicorrelation sequence
\[ \int_X f_0\cdot T_{[p_1(n)]}f_1 \cdot \ldots \cdot T_{[p_k(n)]} f_k \ d\mu\]
can be decomposed as a sum of a uniform limit of $D$-step nilsequences  and a nullsequence.
\end{problem}

\end{document}